\documentclass[aos, preprint]{imsart}

\RequirePackage{amsthm,amsmath,amsfonts,amssymb}
\RequirePackage[authoryear]{natbib}
\RequirePackage[colorlinks,citecolor=blue,urlcolor=blue]{hyperref}
\RequirePackage{graphicx}

\usepackage{bm}
\usepackage{bbm}
\usepackage{url}
\usepackage{multirow}
\usepackage{subfigure}
\usepackage{makecell}
\usepackage{booktabs}
\usepackage{array}
\usepackage{float}
\usepackage[algo2e,boxed,linesnumbered,vlined]{algorithm2e}
\DontPrintSemicolon
\usepackage{dsfont}
\usepackage{mathrsfs}
\usepackage{xcolor}
\usepackage{microtype} 
\usepackage{keyval}
\usepackage{calc}
\usepackage[capitalise,noabbrev]{cleveref}
\usepackage{changepage}

\startlocaldefs

\numberwithin{equation}{section}
\theoremstyle{plain}
\newtheorem{thm}{Theorem}
\newtheorem{prop}{Proposition}
\newtheorem{lemm}{Lemma}
\newtheorem{coro}{Corollary}
\theoremstyle{remark}
\newtheorem{rem}{Remark}

\newtheorem{ex}{Example}
\newtheorem{assum}{Assumption}

\let\hat\widehat
\let\tilde\widetilde

\def\given{\mid\,}

\allowdisplaybreaks

\newcommand{\bx}{\bm{x}}
\newcommand{\by}{\bm{y}}

\newcommand{\bU}{\bm{U}}

\newcommand{\bW}{\bm{W}}
\newcommand{\bX}{\bm{X}}
\newcommand{\bY}{\bm{Y}}

\newcommand{\bone}{\bm{1}}

\newcommand{\cA}{\mathcal{A}}

\newcommand{\cF}{\mathcal{F}}
\newcommand{\cG}{\mathcal{G}}

\newcommand{\cI}{\mathcal{I}}

\newcommand{\cL}{\mathcal{L}}

\newcommand{\cS}{{\mathcal{S}}}

\newcommand{\EE}{\mathbb{E}}

\newcommand{\PP}{\mathbb{P}}

\newcommand{\BM}{\mathrm{BM}}

\newcommand{\sign}{\mathop{\mathrm{sign}}}
\newcommand{\tr}{\mathop{\mathrm{tr}}}
\newcommand{\diag}{{\rm diag}}
\newcommand{\nei}{{\rm ne}}

\newcommand{\opnorm}[2]{| \! | \! | #1 | \! | \! |_{{#2}}}
\newcommand{\normop}[1]{{\left\vert\kern-0.25ex\left\vert\kern-0.25ex\left\vert #1 
		\right\vert\kern-0.25ex\right\vert\kern-0.25ex\right\vert}}

\newcommand{\Prob}{\PP}

\newcommand{\E}{\EE}
 
\newcommand{\var}{\mathbb{V}\text{ar}}

\newcommand{\Op}[1]{O_{\mathbb P}  \left( #1 \right)}

\newcommand{\One}{\mathds{1}}
\newcommand{\Ind}[1]{ \mathds{1} \left\{ #1 \right\} }

\newcommand{\argmin}{\text{arg}\min}

\newcommand{\wc}{\rightsquigarrow}
\newcommand{\too}{\longrightarrow}

\renewcommand{\d}{d}

\newcommand{\R}{\mathbb{R}}

\newcommand{\ns}{\mathrm{ns}}
\newcommand{\gl}{\mathrm{gl}}

\newcommand{\eps}{\varepsilon}

\newcommand{\comments}[1]{}

\renewenvironment{change}{ \color{cyan} }{ \color{black} }

\newcommand{\Gam}{\hat\Gamma}
\def\g{\boldsymbol}
\newcommand{\indep}{\perp \!\!\! \perp}
\DeclareMathOperator{\MST}{mst}
\newcommand{\eglearn}{\texttt{EGlearn}}

\crefname{thm}{Theorem}{Theorems}
\crefname{lemm}{Lemma}{Lemmas}
\crefname{prop}{Proposition}{Propositions}
\crefname{coro}{Corollary}{Corollaries}
\crefname{ex}{Example}{Examples}
\crefname{con}{Condition}{Conditions}
\crefname{assum}{Assumption}{Assumptions}
\crefname{algorithm}{Algorithm}{Algorithms}
\crefname{equation}{}{}

\endlocaldefs

\begin{document}

\begin{frontmatter}

\title{Learning extremal graphical structures\\in high dimensions}
\runtitle{Learning extremal graphical structures}

\begin{aug} 

\author[A]{\fnms{Sebastian} \snm{Engelke}},
\author[B]{\fnms{Micha\"el} \snm{Lalancette}}
\and
\author[C]{\fnms{Stanislav} \snm{Volgushev}}

		\address[A]{
			Research Center for Statistics, University of Geneva
			(\href{mailto:sebastian.engelke@unige.ch}{sebastian.engelke@unige.ch})
		}
		
		\address[B]{
			D\'epartement de math\'ematiques, Universit\'e du Qu\'ebec \`a Montr\'eal
			(\href{mailto:lalancette.michael@uqam.ca}{lalancette.michael@uqam.ca})
		}
		
		\address[C]{
			Department of Statistical Sciences, University of Toronto
			(\href{mailto:stanislav.volgushev@utoronto.ca}{stanislav.volgushev@utoronto.ca})
		}

\end{aug}

\begin{abstract} 
	Extremal graphical models encode the conditional independence structure of multivariate extremes. Key statistics for learning extremal graphical structures are empirical extremal variograms, for which we prove non-asymptotic concentration bounds that hold under general domain of attraction conditions. For the popular class of H\"usler--Reiss models, we propose a majority voting algorithm for learning the underlying graph from data through $L^1$ regularized optimization. Our concentration bounds are used to derive explicit conditions that ensure the consistent recovery of any connected graph. The methodology is illustrated through a simulation study as well as the analysis of river discharge and currency exchange data.
\end{abstract}



\end{frontmatter}

\section{Introduction}

Motivated by extreme value theory, the univariate tail of a random variable can be well described by
parametric generalized extreme value and Pareto distributions.
For effective quantification of the risk related to floods \citep{kee2013}, heatwaves \citep{eng2017a} or financial crises \citep{zhou2009banks}, understanding the interplay between many different risk factors becomes crucial. 
Mathematically this amounts to an analysis of the tail dependence between components of a random vector $\g X=(X_1,\ldots,X_d)$, which can become arbitrarily complex in larger dimensions  $d$. Sparsity or dimension reduction are then required to obtain statistically sound and practically feasible methods; see \cite{eng2021} for a review of recent developments.  

One popular approach to obtaining interpretable dependence models for a random vector $\g X$ in high dimensions relies on graphical modeling \citep{L96}. Conditional independence relations between pairs of variables are imposed by the absence of edges in a graph $G = (V,E)$ with node set $V = \{1,\dots,d\}$ and edges $E\subset V \times V$. Graphical models are particularly well studied for multivariate normal distributions where conditional independence relations are encoded as zeroes in the precision matrix \citep[Chapter 5]{L96}.   

For distributions arising in the setting of extreme value analysis, graphical modeling is more challenging. Broadly speaking, there are two main approaches to modeling asymptotically dependent extremes. The first approach considers component-wise maxima of blocks of random vectors and leads to the notion of max-stable distributions; we refer the reader to \citet[][Chapter 8]{beirlant2004statistics} and \citet[][Chapter 6]{DF2006}. 
In this distribution class, \cite{gis2018,klu2019,AKLT22}, among others, study max-linear models on directed acyclic graphs and their conditional independence structures.
The distributions considered in this line of work do not have densities, and a general result by \cite{papastathopoulos2016conditional} shows that there exist no non-trivial density factorization of max-stable distributions on graphical structures.

The second approach relies on multivariate Pareto distributions, which arise as limits of conditional distributions of a random vector $\g X$ given that at least one of its components is large \citep{roo2006,roo2018}. While multivariate Pareto distributions inherit certain structural properties such as homogeneity from their definition through limits, this class is still very flexible and too large to allow for efficient and interpretable inference in high dimensions. Classical conditional independence is moreover not suited for these distributions since they are not supported on a product space. \cite{EH2020} overcome this challenge by introducing new notions of conditional independence and extremal graphical models for multivariate Pareto distributions. They show that these definitions naturally link to density factorizations and enable efficient inference in high dimensions.
Furthermore, conditional independence in the sense of \cite{EH2020} was later shown to be equivalent to classical conditional independence in the restricted multivariate Pareto distribution over any rectangular domain \citep{eng_iva_kir}.
A different but closely related approach is to study the tails of classical graphical models. Under various structural and modeling assumptions, the conditional distribution of such random vectors given that a fixed variable exceeds a high threshold inherits the graphical structure of the original graphical model, even in the limit as the threshold increases. Markov trees \citep{segers2020,asenova2021inference}, random fields on block graphs \citep{AS21} and decomposable graphical models \citep{casey2023decomposable} have been studied from this angle; see \cite{L96} for background on graphs.

The underlying graph $G$ plays a key role in graphical modeling. It is typically unknown and needs to be estimated in a data driven way. One may consider different levels of generality for the structure of the graph to be estimated. Figure~\ref{fig:intro} showcases four different classes of graphs from simple tree structures, over decomposable graphs, and up to the most general case of possibly non-decomposable graphs. In general, estimating more complex graphs requires more assumptions on the underlying distribution. In the non-extreme world, two important cases correspond to graphical models on trees, which can be estimated non-parametrically \citep{LXGGLW2011}, and multivariate normal distributions, for which general graphs can be estimated through the corresponding precision matrix; see \cite{MB06}, \cite{YL2007} and \cite{FHT2008} among many others.

\begin{figure}[h]
	\includegraphics[clip,height=.2 \textwidth, page=1]{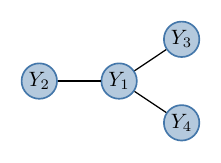}
	\includegraphics[clip,height=.2 \textwidth, page=2]{img/plots.pdf}
	\includegraphics[clip,height=.2 \textwidth, page=3]{img/plots.pdf}
	\includegraphics[clip,height=.2 \textwidth, page=4]{img/plots.pdf}
	\caption{Four graph structures on the node set $V=\{1,\dots, 4\}$. From left to right: tree graph, block graph, decomposable graph, non-decomposable graph.}
	\label{fig:intro}
\end{figure}

Learning non-extremal graphical models relies heavily on the covariance matrix of the data and its inverse, the precision matrix.
For extremal graphical models, it turns out that a more suitable statistic is the extremal variogram $\Gamma$, a $d\times d$-dimensional matrix that summarizes the extremal dependence of a multivariate Pareto distribution~\citep{EV22}. 
Indeed, this quantity plays a central role in describing and learning extremal graphical models on various structures such as trees \citep{EV22,hu2023modeling} and block graphs \citep{asenova2021inference}, or under positivity constraints \citep{rottger2023total,rottger2022local}.
The main theoretical contribution of this paper is to derive finite-sample concentration bounds for the empirical version $\hat \Gamma$ of the extremal variogram. For suitable $\delta > 0$, we provide upper bounds on the probability that $\| \hat \Gamma - \Gamma\|_\infty > \delta$ as a function of $\delta$, the effective sample size, the dimension $d$, and other model parameters. A key difficulty of our theoretical analysis is the fact that the empirical variogram $\hat \Gamma$ uses the largest observations in the sample, which are only approximate realizations of the limiting multivariate Pareto distribution and whose margins are normalized empirically, destroying independence between observations. Moreover, the special structure of the extremal variogram renders standard concentration results in the literature \citep{GSC15,CJLSS2021,LSS2021} inapplicable.
The theory for non-parametric tree recovery in growing dimensions in \cite{EV22} crucially relies on our concentration bounds. Similarly, the proofs of consistent graph recovery in the recent papers \cite{wan2023graphical} and \cite{engelke2024extremal} directly apply our results.

The second main contribution of our paper is methodological.  
While the method in \cite{EV22} consistently learns tree structures, for more general graphs such as the three graphs on the right of Figure~\ref{fig:intro}, no methods that guarantee consistent graph recovery existed in the world of extremes before the methods introduced in this paper; see Section~\ref{sec:lit} and Remark~\ref{rem:wanzhou} for a discussion of recent contributions in this direction. In the present paper, we propose a general methodology to learn arbitrary graphs for the class of H\"usler--Reiss distributions \citep{HR1989}. Those distributions share many attractive properties of multivariate normal distributions and have been widely used in modeling multivariate and spatial extremes \citep[e.g.,][]{dav2012b, EMKS2015}. They are parametrized by a $d\times d$-dimensional variogram matrix $\Gamma$, which happens to coincide with the extremal variogram.
In this distribution class, the positive semi-definite H\"usler--Reiss precision matrix $\Theta$ \citep{H21}, obtained by a transformation of $\Gamma$, contains the extremal graph structure as its zero pattern. We propose \eglearn{}, a majority voting algorithm that uses  lasso-type penalties to estimate the sparsity structure of this matrix. Using our concentration bounds for the empirical variogram $\hat \Gamma$, we prove universal high-dimensional structure recovery guarantees for \eglearn{} that hold for arbitrary graph structures.

The rest of the paper is organized as follows. Section~\ref{sec:back} contains necessary background information on multivariate extreme value theory, H\"usler--Reiss distributions and extremal graphical models. Empirical extremal variograms and their statistical properties, in particular our main concentration result, are studied in Section~\ref{sec:empvariogen}. Our estimation methodology for H\"usler--Reiss graphical models is described in detail and finite-sample performance guarantees are stated in Section~\ref{sec:metho}. The proposed methods are illustrated through a simulation study in Section~\ref{sec:sims} and the analysis of a real data set in Section~\ref{sec:appl}. Finally, potential extensions and directions for future work are explored in Section~\ref{sec:future}. Additional numerical results, a second data application and all proofs are delegated to the Supplementary Material. All references to sections, results, equations, etc.~starting with the letter ``S'' are pointing to this supplement.

The methods of this paper are implemented in the R package \texttt{graphicalExtremes} and all numerical results and figures can be reproduced using the code on \url{https://github.com/sebastian-engelke/extremal_graph_learning}.

\subsection{Related work on learning extremal graphical structures}
\label{sec:lit}

The probabilistic foundations and conditional independence structures of max-linear models as introduced by \cite{gis2018} are thoroughly studied in \cite{klu2019,gissibl2021identifiability,AHST21,AKLT22}. \cite{kluppelberg2021estimating} develop estimation procedures for model coefficients over a known underlying graph. If the graph is unknown, they propose an algorithm to learn the causal order, hence gaining partial information on the graph structure; see also \citet[Section 5]{gissibl2021identifiability}. However, they do not provide theoretical recovery guarantees. Recently, \cite{krali2023heavy} consider identification and estimation of causal relationships in max-linear models with hidden nodes, but again do not guarantee consistent structure recovery. \cite{tra2021} propose an algorithm to consistently recover root-directed tree structured max-linear models in fixed dimensions.
In contrast, our approach works with a completely different class of models and allows to estimate arbitrary (undirected) connected graphs. We moreover provide guarantees on graph recovery in increasing dimensions.

At the time of the first submission of the present paper, structure learning for undirected extremal graphical models in the sense of \cite{EH2020} had only been developed for simple tree structures. Indeed, the first procedure enjoying provably consistent graph recovery of extremal graphical structure is the minimum spanning tree algorithm based on the extremal variogram $\Gamma$ in \cite{EV22}.
Their approach is tailored to tree estimation and cannot be generalized to produce other types of estimated graphs. 
\cite{EV22} show consistency of the empirical variogram in fixed dimensions. In the present paper, we prove a much stronger result: non-asymptotic concentration bounds, which imply consistency of the empirical variogram $\hat \Gamma$ in exponentially growing dimensions with explicit convergence rates. The theory for graph recovery in growing dimensions in \cite{EV22} crucially relies on our concentration bounds. \cite{hu2023modeling} provide an algorithm for learning tree structures of a closely related model class and provide guarantees on their consistent recovery in fixed dimensions. Similarly to \cite{EV22}, their approach does not extend to estimation of graph structures beyond trees.

In the present paper, we propose the first methodology to learn the structure of arbitrary extremal graphical models in the parametric family of H\"usler--Reiss distributions \citep{HR1989}.
Since the first submission of this work, a number of papers have appeared that present other structure learning algorithms. 
\cite{wan2023graphical} propose an alternative method to learn the structure of H\"usler--Reiss graphical models, and \cite{engelke2024extremal} consider the problem of learning latent variable extremal graphical models in the same family of distributions. Both works use our concentration bounds to establish high-dimensional graph recovery guarantees; see Remarks~\ref{rem:norms} and~\ref{rem:wanzhou} for details.
Estimation of high-dimensional H\"usler--Reiss distributions by score matching was recently proposed by \cite{lederer2023extremes}. While the estimated model may exhibit some conditional independencies, this sparsity is a byproduct of their approach, and no guarantees are provided on graph recovery.

A somewhat different approach to graphical modeling of extreme values of functional neuronal connectivity is presented in \cite{chang2023subbotin}.
The authors construct a new family of multivariate Subbotin distributions and argue that dependencies in this model are mostly due to extreme observations. They derive an algorithm that provably learns the graph structure for this model class. Unlike H\"usler--Reiss and other multivariate Pareto distributions studied in the present work, the Subbotin model is not an extreme value distribution and therefore does not aim to study the tails of random vectors. Moreover, their learning theory is based on independent observations from this exact model, rather than the pseudo-observations from the domain of attraction with which we work in this paper.

\subsection*{Notation}
For a vector $\bx$ and matrix $A$, $\|\bx\|_\infty$ and $\|A\|_\infty$ is used to denote the element-wise sup-norm, whereas $\opnorm{A}{p}$ denotes the $L^p/L^p$ operator norm. If $A$ is square with real eigenvalues, let $\lambda_{\min}(A)$ denote its smallest eigenvalue. For a natural number $d \geq 1$ let $[d] := \{1,\dots,d\}$. For a vector $\g x = (x_1,\dots,x_d)^\top \in \R^d$ (or similarly a random vector $\bX = (X_1, \dots, X_d)^\top$) and subset $J = \{j_1,\dots,j_k\} \subseteq [d]$, define the vector $\g x_J := (x_{j_1},\dots,x_{j_k})^\top$ where $j_1 < j_2 < \dots < j_k$. Vectors with all entries $1$ and $0$ are denoted with $\g 1$ and $\g 0$, respectively; the dimension is clear form the context. Inequalities between vectors are understood component-wise. The notation $\g x_{\setminus J}$ is used to denote the vector $\g x_{V\backslash J}$, and if $J= \{j\}$ we write $\g x_{\backslash j}$. We use similar notations to index rows and columns of a matrix.

\section{Background} \label{sec:back}

\subsection{Multivariate Pareto distributions and domains of attraction}

Let $\g Y = (Y_j: j \in V)$ be a $d$-dimensional random vector indexed by $V=\{1,\dots, d\}$ with support contained in the set $\mathcal L = \{ \g y \geq 0: \|\g y \|_\infty >1\}$. The random vector $\g Y$ is said to have a multivariate Pareto distribution if $\mathbb P(Y_1 > 1)= \dots = \mathbb P(Y_d > 1)$ and if it satisfies the homogeneity property
\begin{align}\label{MPD}
\mathbb P( \g Y \in t A) = t^{-1}  \mathbb P( \g Y \in A), \qquad t \geq 1,
\end{align}
where for any Borel subset $A \subset \mathcal L$ we define $t A = \{t\g y : \g y\in A\}$ \citep{roo2006}.
Homogeneity implies that for any $i \in V$ the univariate conditional margin satisfies $\mathbb P(Y_i \leq  x \given Y_i > 1) = 1- 1 /x$ for $x\geq 1$, that is, $Y_i\given Y_i > 1$ follows a standard Pareto distribution.
If $\bY$ has a density $f$, homogeneity further implies that $f$ is homogeneous of order $-(d+1)$, i.e., 
\begin{equation}\label{eq:fhom}
f(t\by) = t^{-(d+1)} f(\by), \quad t \geq 0, \by, t\by \in \cL.
\end{equation}

Multivariate Pareto distributions arise as natural models in the study of multivariate extreme events, since they are the only possible limits of so-called threshold exceedances. To formalize this, consider a random vector $\g X = (X_j: j \in V)$ in $\R^d$ and define $F(\g x) = (F_1(x_1),\dots,F_d(x_d))$. Suppose that $\g X$ has eventually continuous marginal distributions $F_i$, in the sense that there exists $x_0$ such that $F_i(x_0) < 1$ and $F_i$ is continuous on $(x_0, \infty)$. If for some random vector $\g Y$ the limit relation
\begin{equation} \label{eq:mPdlimit}
  \lim_{q \downarrow 0} \Prob(F(\g X) \leq 1 - q/\g x \given F(\g X) \not\leq 1-q) = \Prob(\g Y \leq \g x)
\end{equation}
holds at all continuity points $\g x \in \mathcal L$ of $\Prob(\g Y \leq \cdot)$, we say that $\g X$ is in the domain of attraction of ${\g Y}.$
In this case, the limit $\g Y$ is necessarily a multivariate Pareto distribution, and conversely any multivariate Pareto distribution appears as such limit; see \citet[][Proposition 6]{EV22} for a more formal statement.
Since $F$ standardizes the margins of $\g X$, this is an assumption on the extremal dependence structure of $\g X$. The notation $ F(\g X) \not\leq 1-q$ means that at least one component of the vector $F(\g X)$ exceeds $1-q$, or equivalently at least one component $X_j$, $j \in V$, of $\g X$ exceeds its high quantile $F_j^{-1}(1-q)$. The limit~\eqref{eq:mPdlimit} can thus be interpreted as a model for realizations of $\g X$ with at least one extreme component.

The class of multivariate Pareto distributions is very rich and contains many parametric sub-families that are used for statistical modeling. In fact, it is equivalent to the class of extreme value copulas \citep{seg2010}. We discuss two examples here and refer to \cite{seg2010} for a more extensive list. A classical model with only one parameter is the extremal logistic model \citep{T1990}.
\begin{ex}[Extremal logistic model] \label{ex:logistic}
	The extremal logistic distribution with parameter $\theta\in(0,1)$ is a multivariate Pareto distribution with density
	\begin{align}\label{dens_log}
	  f_{\g Y}(\g y; \theta) \propto \prod_{i=1}^d y_i^{{{-1-1/\theta}}} \left(y_1^{-1/\theta}+\dots+y_d^{-1/\theta}\right)^{\theta-d}, \quad \g y \in \mathcal L.
	\end{align}
\end{ex}
One of the most popular models of multivariate Pareto distributions is the H\"usler--Reiss family \citep{HR1989}. 
\begin{ex}[H\"usler--Reiss model] \label{ex:HR}
	The family of H\"usler--Reiss distributions is parametrized by a symmetric, strictly conditionally negative definite matrix $\Gamma$ in the cone
	\begin{align}
		\mathcal C^d
		=
		\left\{
			\Gamma \in [0, \infty)^{d\times d}:		
			\Gamma = \Gamma^\top
			,\;\;
			\diag{(\Gamma)} = \mathbf{0}
			,\;\;
			v^\top \Gamma v < 0
			\;\forall\,
			\mathbf{0} \neq v \perp \mathbf{1}
		\right\}		.
	\end{align}	
	Equivalently, the model can be parametrized by the positive semi-definite precision matrix $\Theta = (P (-\Gamma/2) P)^+$, where $P = I_d - \mathbf{1}\mathbf{1}^\top /d$ is the projection matrix onto the orthogonal complement of $\mathbf{1}$, and $A^+$ is the Moore--Penrose pseudo-inverse of a matrix $A$; see \cite{H21} for details. The H\"usler--Reiss distribution is then defined through its density
	\begin{align}\label{HR_density}
        f_{\g Y}(\g y; \Theta)
        &\propto
	\prod_{i=1}^d y_i^{-1-1/d} \exp\left\{ -\frac12 (\log \by - \mu_\Theta)^\top \Theta (\log \by - \mu_\Theta) \right\}, \quad \g y \in \mathcal L,
    \end{align}
	where $\mu_\Theta = - P \Gamma \g 1/(2d)$. The normalizing constant is provided in \cite{H21}.
	
\end{ex}

The importance of H\"usler--Reiss distributions comes from both theory and applications. From a theoretical perspective, they are the only possible limits of rescaled and normalized componentwise maxima of Gaussian random vectors \citep{HR1989, KSH2009}, and thus a natural object to study.
H\"usler--Reiss distributions are often considered to be the analogue of Gaussian distributions in the study of extremes. One reason for this is that they are the only continuous multivariate Pareto distributions that exhibit the structure of a pairwise interaction model \citep{lalancette2023pairwise}. Moreover, as we will see later, they are natural tools for graphical modeling since conditional independencies can be identified from zeros in the precision matrix $\Theta$ \citep{EH2020}.  
From a practical perspective, the popularity of this model class is due to its flexibility in terms of statistical modeling \citep[e.g.,][]{deFondeville2018, thibaud_davison}. Their spatial extensions, Brown--Resnick processes, are standard models for spatial extreme events \citep{dav2012b, EMKS2015}.

\subsection{Extremal graphical models}

Since the support space $\mathcal L$ of a multivariate Pareto distribution $\g Y$ is not a product space, the classical notion of conditional independence fails to hold in general between components of $\bY$. However, the auxiliary conditional laws $\g Y^{(m)} = (\g Y \given Y_m > 1)$, $m\in V$, are supported on a product space. They form the basis for the definition of conditional independence for the random vector $\g Y$ in \cite{EH2020}. For disjoint sets $A,B,C\subset V$, we say that $\bY_A$ is conditionally independent of $\bY_C$ given $\bY_B$ in the extremal sense if the usual conditional independence holds for all auxiliary vectors:
\begin{equation}\label{ext_CI}
  \forall m \in V: \bY_A^{(m)} \indep \bY_C^{(m)} \given \bY_B^{(m)}.
\end{equation}
We denote this extremal conditional independence by $\bY_A \perp_e \bY_C \given \bY_B$. It turns out that~\eqref{ext_CI} holds for every $m \in V$ as long as it holds for a single $m\in V$ \citep{eng_iva_kir}. Additional comments on the connection of~\eqref{ext_CI} to classical conditional independence and on the utility of the resulting graphical models are provided in Section~\ref{sec:connectionscalssicalCI} of the Supplementary Material.

In graphical modeling, conditional independence is connected to graph structures to define sparse probabilistic models \citep{L96}. An undirected graph $G = (V,E)$ is a set of nodes $V = \{1,\dots, d\}$ and a collection of edges $E \subset V \times V$ of unordered pairs of distinct nodes.
With the notion of extremal conditional independence~\eqref{ext_CI}, we define an extremal graphical model as a multivariate Pareto distribution $\g Y$ that satisfies the extremal pairwise Markov property
\begin{equation*}
	Y_i \perp_e Y_j \given \bY_{\backslash\{i, j\}}, \quad (i, j) \notin E.
\end{equation*}
When $\g Y$ has a positive continuous Lebesgue density $f$, the graph $G$ is necessarily connected.
In this case, the extremal pairwise Markov property defined above is equivalent to a stronger extremal global Markov property. If $G$ is in addition decomposable, then $f$ factorizes on the graph into lower-dimensional densities, and inference is considerably more efficient thanks to the sparsity; see \cite{EH2020,AS21} for details. This is analogous to the standard Markov properties and undirected graphical models defined with respect to the usual notion of conditional independence \citep[][Chapter 3]{L96}.

\begin{ex}[Extremal logistic model, continued]\label{ex:logisticGM}
An extremal logistic distributed $\bY$ (Example~\ref{ex:logistic}) contains no extremal conditional independence, regardless of the dependence parameter $\theta$ \citep{EH2020}, that is, $Y_i \not\perp_e Y_j \given \bY_{\backslash \{i, j\}}$ for any $i, j \in V$. Hence $\bY$ can only be an extremal graphical model on the fully connected graph.
\end{ex}

\begin{ex}[H\"usler--Reiss model, continued] \label{ex:HRGM}
For a H\"usler--Reiss distributed $\bY$ (Example~\ref{ex:HR}), its precision matrix $\Theta$ encodes the graphical structure through the simple relation
\begin{equation} \label{eq:HRtheta}
Y_i \perp_e Y_j \given \bY_{\backslash \{i, j\}}  \quad  \Longleftrightarrow \quad  \Theta_{ij} = 0,
\end{equation}
for $i,j\in V$; see \cite{H21} for details. Therefore, $\bY$ is an extremal (or H\"usler--Reiss) graphical model on $G := (V, E)$ if and only if $\Theta_{ij} = 0$ whenever $(i, j) \notin E$. Note that such models exist for any connected graph $G$.
\end{ex}

Extremal graphical models allow the construction of sparse models from existing low-dimensional multivariate Pareto distributions. An example is given by models on tree structures.

\begin{ex}[Extremal tree models] \label{ex:treeGM}
	A tree $T = (V,E)$ is a connected graph without cycles; see the graph on the left of Figure~\ref{fig:intro} for an example.
	The density of an extremal graphical model $\g Y$ on tree $T$, also called an extremal tree model, factorizes as 
	\begin{align}\label{tree_fact}
	f_{\g Y}(\g y) \propto \Big( \prod_{\{i,j\}\in E} \frac{f_{\{i,j\}}(y_i, y_j)}{y_i^{-2} y_j^{-2}} \Big) \prod_{i \in V} y_i^{-2}, \quad \g y \in [1, \infty)^d,
	\end{align}
	where for $i,j\in E$, $f_{\{i,j\}}$ is the density of the bivariate Pareto distribution arising as the limit of $(X_i,X_j)$ in~\eqref{eq:mPdlimit}. This factorization can be extended to $\cL$ by homogeneity of $f_{\g Y}$.

	Conversely, given a tree $T$ and a collection of densities of bivariate Pareto distributions $f_{\{i,j\}}, i,j \in E$, \eqref{tree_fact} defines a valid $d$-dimensional density of an extremal tree model on $T$ \citep{EH2020}. This provides a nonparametric construction principle of sparse extreme value models in higher dimensions. We can also use parametric models for $f_{\{i,j\}}$: for instance, an extremal tree model with H\"usler--Reiss bivariate margins is itself a H\"usler--Reiss graphical model as in \cref{ex:HRGM}. In contrast, using bivariate extremal logistic densities $f_{\{i,j\}}$ leads to a new model that is not a $d$-variate extremal logistic distribution, since the latter necessarily corresponds to a fully connected graph; see Example~\ref{ex:logisticGM}.
\end{ex}

\subsection{Extremal variograms}

In the setting of Examples~\ref{ex:HRGM} and~\ref{ex:treeGM}, a natural question is whether the underlying graph can be estimated from data. Here, the extremal variogram \citep{EV22} turns out to play a crucial role. The extremal variogram rooted at node $m\in V$ is defined as the matrix $\Gamma^{(m)}$ with entries
\begin{align}\label{vario}
\Gamma_{ij}^{(m)} &= \var \left\{ \log Y^{(m)}_{i} - \log Y^{(m)}_j \right\}, \quad i,j\in V,
\end{align}
whenever the right-hand side exists and is finite.

For extremal tree models $\g Y$ on $T$, \cite{EV22} show that the extremal variograms form a tree metric, i.e.,
\begin{align}\label{tree_metric}	
    \Gamma_{ij}^{(m)}
    =
    \sum_{(s,t) \in \text{path}(i,j)} \Gamma_{st}^{(m)}
    ,\quad i,j\in V,
\end{align} 
where $\text{path}(i,j)$ is the set of edges that form the unique path between $i$ and $j$ in the tree. They leverage this property to provide an algorithm for learning the underlying tree structure. Deriving theoretical guarantees for their proposal crucially relies on concentration bounds for estimators of the extremal variogram that we provide in the present paper (Theorem~\ref{thm:concentration}). This tree metric property continues to hold for the combined extremal variogram $\Gamma$ defined as
\begin{equation}\label{eq:avevario}
	\Gamma_{ij} := \frac{1}{d}\sum_{m=1}^d \Gamma_{ij}^{(m)},\quad i,j\in V.
\end{equation}

For H\"usler--Reiss distributions, estimating the underlying graph structure can be reduced to finding the zero entries of the precision matrix $\Theta$. It turns out that for any $m \in V$, the extremal variogram matrix $\Gamma^{(m)}$ (and hence the combined version $\Gamma$) is equal to the matrix $\Gamma$ from Example~\ref{ex:HR} \citep{EV22}. Thus estimating $\Gamma^{(m)}$ provides a stepping stone for learning the sparsity pattern of $\Theta$.
While this resembles Gaussian graphical modeling, we note that $\Theta$ is not invertible, which makes it impossible to directly apply classical algorithms from Gaussian graphical modeling. Those challenges are addressed in more detail in Section~\ref{sec:metho}.

Due to the central role that the variogram plays in parametrizing H\"usler--Reiss distributions, it has multiple applications ranging from inference via matrix completion problems \citep{H21}, estimation of H\"usler--Reiss distributions under total positivity constraints \citep{rottger2023total}, and other special types of graphical models for H\"usler--Reiss distributions \citep{rottger2023parametric,rottger2022local}; we refer to \cite{engelke2024review} for a recent review. The concentration bounds that we establish in the next section are also directly used for an alternative method for extremal graph learning in~\cite{wan2023graphical}, and for learning latent variable extremal graphs in~\cite{engelke2024extremal}.

\section{Estimation of extremal variograms and concentration bounds} \label{sec:empvariogen}

\subsection{The empirical extremal variogram}
\label{sec:empvario}

The extremal variogram matrix $\Gamma^{(m)}$ rooted at node $m$ in \eqref{vario} is defined through the multivariate Pareto distribution $\bY$. In applications, we do not observe data from $\g Y$ but rather from $\g X$ in the domain of attraction of $\g Y$ in the sense of \cref{eq:mPdlimit}. It then follows that as $q \downarrow 0$, the conditional law
\[
	\frac{q}{1-F(\g X)} \, \Big | \, \{F_m(X_m) > 1-q\}
\]
converges in distribution to $\bY^{(m)}$, where we recall that $F(\bx)$ denotes $(F_1(x_1), \dots, F_d(x_d))$; see \cref{prop:doa} in the Supplementary Material. Let $(\bX_t := (X_{t1}, \dots, X_{td}): t \in \{1,\dots, n\})$ be a sample of independent copies of $\bX$, and let $k < n$ be a positive integer. Denoting by $\tilde F_i$ the empirical distribution function of $X_{1i}, \dots, X_{ni}$ and by $\tilde F(\g x)$ the vector $(\tilde F_1(x_1), \dots, \tilde F_{d}(x_d))$,
\[
	\bigg\{ \frac{k/n}{1 - \tilde F(\g X_t) + 1/n} : \tilde F_m(X_{tm}) > 1 - k/n \bigg\}
\]
therefore forms an approximate sample of size $k$ from $\g Y^{(m)}$, as long as $k/n$ is sufficiently small. After applying logarithms, this motivates the empirical extremal variogram $\hat\Gamma^{(m)}$ rooted at node $m$ \citep{EV22}, defined by
\[
	\hat \Gamma_{ij}^{(m)} := \widehat{\var}\Big(\log (\tfrac{n+1}{n} - \tilde F_i(X_{ti})) - \log(\tfrac{n+1}{n} - \tilde F_j(X_{tj})) : \tilde F_m(X_{tm}) > 1 - k/n  \Big),
\]
where $\widehat{\var}$ denotes the empirical variance with scaling equal to $(k+1)^{-1}$, the inverse of the effective sample size. A natural estimator for the combined extremal variogram in~\eqref{eq:avevario} is
\begin{equation} \label{eq:defempvario}
	\Gam := \frac{1}{d} \sum_{m \in V} \Gam^{(m)}.
\end{equation}
For the sake of brevity, the empirical extremal variograms $\hat\Gamma^{(m)}$ and $\hat\Gamma$ will simply be referred to as empirical variograms.

\begin{rem}[On the role of $k$]
For a given $m\in V$, by definition of the estimator $\widehat \Gamma^{(m)}$, $k$ is the number of marginal exceedances in the $m$th component that are used for estimation. It is therefore the effective sample size for this estimator. Choosing a different $m' \in V$, there will again be $k$ effective extreme samples for $\widehat \Gamma^{(m')}$. If there is strong dependence between the extremes of the $m$th and $m'$th variable, then these exceedances will have a large overlap. Consequently, the effective sample size that goes into the combined extremal variogram estimator $\widehat \Gamma$ is between $k$ and $dk$, depending on the strength of extremal dependence.

Choosing $k$ is an important question in extremes.
Indeed, a larger $k$ leads to a smaller variance. However, for a large $k$, $k/n$ can be far from zero and the observations entering $\widehat \Gamma$ can have a distribution that is far from $\bY$ leading to large bias. This bias-variance tradeoff is discussed more formally after \cref{thm:concentration}.
\end{rem}

\subsection{Concentration of empirical variograms} \label{sec:varioconc}

One of the most important contributions of this paper is a finite sample concentration bound that holds simultaneously for all the empirical variograms $\hat\Gamma^{(m)}$ under a general domain of attraction condition. To state the conditions for this result in a form that resembles typical assumptions made in the literature on asymptotics of extremes, we introduce additional notation. Define the tail copula function, or $R$-function
\begin{equation} \label{eq:Rlimit}
	R(\g x) := \lim_{q \downarrow 0} q^{-1} \Prob(F(\g X) > 1 - q\g x), \quad \g x \in [0,\infty)^d.
\end{equation}

The tail copula is a popular object in describing multivariate extreme value distributions and turns out to be convenient for our theoretical analysis. In what follows, for ordered sets $J \subset V$, let $R_J$ denote the tail copula corresponding to $\g X_J$. When $J$ is a pair or a triple, we write $R_{ij}$ and $R_{ijm}$ for $R_{(i, j)}$ and $R_{(i, j, m)}$.

The limit in \cref{eq:Rlimit} can be shown to exist under equation~\eqref{eq:mPdlimit}, and the function $R$ is related to $\g Y$ through the identity $R(\g x) = \Prob(\g Y > 1/\g x)/ \Prob(Y_1 > 1)$ whenever $1/\bx \in \cL$.
Provided that $\Prob(Y_j = 0) = 0$, $j \in V$, \eqref{eq:mPdlimit} further implies convergence of rescaled probabilities of joint extremes of every subset $J$ of the variables $X_1, \dots, X_d$ to the corresponding tail copula function $R_J$, which in turn can be interpreted as the limit of $R$ when components in $J^c$ are taken to infinity; see \cref{prop:subR} in the Supplementary Material.
Our main assumption is that this convergence holds for the subsets $J$ of size 3, uniformly and at a polynomial rate.

\begin{assum}[Extended second order condition] \label{assum:tail}
The marginal distribution functions $F_1,\dots,F_d$ of $\bX$ are continuous and there exist constants $\xi > 0$ and $K < \infty$ independent of $d$ such that for all triples of distinct indices $J = (i,j,m)$ and $q \in (0, 1]$,
\begin{equation} \label{eq:tail}
\sup_{\g x \in [0, q^{-1}]^2 \times [0, 1]} \Big| q^{-1} \Prob\big( F_J(\g X_J) > 1 - q\g x\big) - R_J(\bx) \Big| \leq K q^\xi,
\end{equation}
where $F_J(\g x_J) = (F_{i}(x_{i}),F_j(x_j),F_{m}(x_{m}))$.
\end{assum}

This formulation differs from classical second order conditions in that the supremum is taken over sets that grow with $q$. For extremal logistic and H\"usler--Reiss distributions, it is implied by a standard second order condition on bounded sets (Assumption~\ref{assum:RJstand} in the Supplementary Material) which is routinely imposed in theoretical developments for multivariate extreme value theory \citep{EKS2012, fougeres2015, EV22}; see Proposition~\ref{prop:equiv} in the Supplementary Material and the subsequent discussion for details.

\Cref{assum:tail} is sufficient to derive concentration bounds for the empirical extremal variogram, but a sharper result is obtained if $R$ has pairwise densities satisfying a certain~bound.

\begin{assum}[Bounds on densities]
\label{assum:r}
There exists an $\eps >0$ such that for every $\beta \in [-\eps, 1+\eps]$ there is a constant $K(\beta)$ with the following property: for each $i, j \in V$, $i \neq j$, the function $R_{ij}$ has a mixed partial derivative $r_{ij}$ satisfying
	\[
		r_{ij}(x, y) := \frac{\partial^2}{\partial x \partial y} R_{ij}(x, y) \leq \frac{K(\beta)}{x^\beta y^{1-\beta}}, \quad (x, y) \in (0, \infty)^2.
	\]
\end{assum}

\begin{rem}
\Cref{lemm:logisticdensity,lemm:HRdensity} in the Supplementary Material show that \cref{assum:r} is satisfied by any extremal logistic and H\"usler--Reiss distribution, with $\eps = \theta^{-1} - 1$ in the logistic model and with arbitrarily large $\eps$ in the H\"usler--Reiss model. See the lemmas for explicit constants $K(\beta)$. In addition, it can be checked that \cref{assum:r} holds if
\[
	r_{ij}(x, 1-x) \leq K_r (x(1-x))^\eps, \quad x \in (0, 1),
\]
for some positive constants $K_r$ and $\eps$.
\end{rem}

We are now ready to state the main theoretical result of this paper.
\begin{thm} \label{thm:concentration}
	
	Let \cref{assum:tail} hold and $\zeta \in (0, 1)$ be arbitrary. There exist positive constants $C$, $c$ and $M$ only depending on $K$, $\xi$ and $\zeta$ such that for any $n^\zeta \leq k \leq n/2$ and $\lambda \leq \sqrt{k}/(\log n)^4$,
	\[
		\Prob\bigg( \max_{m \in V} \big\| \Gam^{(m)} - \Gamma^{(m)} \big\|_\infty > C \Big\{ \Big(\frac{k}{n}\Big)^\xi (\log(n/k))^2 + \frac{(\log(n/k))^2(1 + \lambda)}{\sqrt{k}} \Big\} \bigg) \leq M d^3 e^{-c\lambda^2}.
	\]
	
	If in addition \cref{assum:r} holds, there exists a positive constant $\bar C$ only depending on $K$, $\xi$, $\zeta$, $\eps$ and $K(\beta)$ such that for any $k$ and $\lambda$ as above,
	\[
		\Prob\bigg( \max_{m \in V} \big\| \Gam^{(m)} - \Gamma^{(m)} \big\|_\infty > \bar C \Big\{ \Big(\frac{k}{n}\Big)^\xi (\log(n/k))^2 + \frac{1 + \lambda}{\sqrt{k}} \Big\} \bigg) \leq M d^3 e^{-c\lambda^2}.
	\]
The above statements continue to hold with $\max_{m \in V} \big\| \Gam^{(m)} - \Gamma^{(m)} \big\|_\infty$ replaced by $\big\| \Gam - \Gamma \big\|_\infty$.
\end{thm}

This theorem is the main technical result of our paper, and the proof turns out to be surprisingly involved; especially when establishing the sharper bound under \cref{assum:r}. A major difficulty stems from the use of empirical distribution functions to normalize the margins and the fact that some components of the vectors entering the estimation are not necessarily large. This is in contrast to most estimators in the extremes literature.

As mentioned previously, this result is of general interest in structure learning for extremes. It provides a crucial ingredient in the analysis of tree learning in \cite{EV22}. Specifically, the proof of consistent tree recovery based on extremal variograms in high dimensions (Theorem 4 in the latter paper) heavily utilizes our Theorem~\ref{thm:concentration}. Similarly, Proposition 4.1 in \cite{wan2023graphical} is essentially an adaptation of our concentration bounds to their setting and the proof mainly relies on our Theorem~\ref{thm:concentration}. Furthermore, the estimation procedures for H\"usler--Reiss distributions in \cite{rottger2023total} (Section~5), \cite{rottger2022local} (Section~4) and in \cite{rottger2023parametric} (Section~5) rely on a surrogate likelihood which only depends on the data through an empirical variogram. Theorem~\ref{thm:concentration} could therefore provide an important ingredient for theoretical guarantees associated to these methods, and for extending the consistency result Theorem 5.7 in \cite{rottger2023total} to increasing dimensions. Extending the consistency results in Section~5 of \cite{H21} from fixed to increasing dimensions is likely to utilize our Theorem~\ref{thm:concentration} as well.

\begin{rem}\label{rem:asympk}
The term $(k/n)^\xi (\log(n/k))^2$ appearing in the upper bound above results from an upper bound on the bias in estimating $\Gamma^{(m)}$ which is due to the fact that we only observe data from the domain of attraction of $\g Y$, rather than from the distribution of $\bY$ directly. This term is increasing in $k$ (for sufficiently small $k$) since larger values of $k$ mean more data that are far from the limiting model are used. The second term involving $\lambda$ in both cases arises from the stochastic error. We note that $k$ corresponds to the effective number of observations used for estimation of $\Gamma^{(m)}$, and in that sense the term $\lambda/\sqrt{k}$ appearing in the second tail bound corresponds to the typical $\sqrt{k}$ convergence rates in extreme value theory. This second term is decreasing in $k$ since using more data means less stochastic error.
To obtain estimators with good convergence rates, the bias and variance need to be balanced. For instance, assuming $\log d = o(k/(\log n)^8)$, the second bound in Theorem~\ref{thm:concentration} yields 
\[
	\max_{m \in V} \big\| \Gam^{(m)} - \Gamma^{(m)} \big\|_\infty = O_\Prob\Big( \Big(\frac{k}{n}\Big)^\xi (\log(n/k))^2 + (\log d)^{1/2}k^{-1/2}\Big),
\]
by choosing $\lambda = \sqrt{C\log d}$ for a suitable constant $C$. Then, provided that $\log d = O(n^\alpha)$ for some $\alpha < 1$, the optimal convergence rate is obtained with the choice $k \sim n^{2\xi/(2 \xi +1)}[\log n]^{-4/(2\xi+1)}[\log d]^{1/(2 \xi + 1)}$,
for which we obtain
\[
\max_{m \in V} \big\| \Gam^{(m)} - \Gamma^{(m)} \big\|_\infty = O_\Prob\Big( \Big( \frac{\log d}{n} \Big)^{\xi/(2\xi+1)} (\log n)^{2/(2\xi+1)} \Big).
\]
This choice depends on the unknown parameter $\xi$ from Assumption~\ref{assum:tail} quantifying the population level bias, which is typical in extreme value theory.
In particular, in a situation where the bias is nonexistent, $\xi$ could be taken as infinite in \cref{assum:tail}. Consequently, as $\xi \to \infty$ the above rate of convergence approaches $\sqrt{(\log d)/n}$, the typical uniform entry-wise rate of convergence of covariance estimators in high-dimensional statistics. 
\end{rem}
   
\begin{rem}\label{rem:norms}
As pointed out by the Associate Editor, convergence rates in norms other than the entry-wise max norm can also be of interest. Such bounds can be obtained from our results by utilizing general bounds on norm equivalence. For instance, the spectral norm admits the universal bound $\opnorm{\Gam - \Gamma}{2} \leq d \|\Gam - \Gamma\|_\infty$, which, together with our Theorem~\ref{thm:concentration}, is used in the theory of latent variable extremal graphical models in \cite{engelke2024extremal}. 
This approach ignores the dependence structure between entries of the estimated matrix.  Sharper results might hold under additional assumptions, but a challenge is that, due to marginal normalization, our estimators are not based on sums of independent random matrices. Hence none of the classical concentration results for sums and quadratic forms of independent random vectors are applicable in our setting. We leave this interesting question for future research.
\end{rem}

\section{Learning H\"usler--Reiss graphical models}
\label{sec:metho}

\subsection{Characterization of H\"usler--Reiss graphical models}

Recall the H\"usler--Reiss distribution from Example~\ref{ex:HR} and its parametrization by the variogram and precision matrices $\Gamma$ and $\Theta$. In what follows, suppose that $\bY$ is such a model and define its conditional independence graph $G = (V,E)$ by
\begin{align}\label{CI_graph}
	Y_i \perp_e Y_j \given \bY_{\backslash\{i, j\}} \quad \Longleftrightarrow \quad (i, j) \notin E.
\end{align}
In other words, $G$ is the smallest graph over which $\bY$ is an extremal graphical model. In view of \cref{ex:HRGM}, this means that for $i \neq j$, $\Theta_{ij} = 0$ if and only if $(i, j) \notin E$. Estimating $G$ is thus equivalent to estimating the sparsity pattern of $\Theta$. Note that, in contrast to the setting of Gaussian graphical models, the matrix $\Theta$ is not invertible and its sparsity pattern can therefore not be directly estimated by naively applying algorithms designed for those models.

The key observation to overcome this problem is that for any $m\in V$, the submatrix $\Theta_{\setminus m, \setminus m} \in \R^{(d-1) \times (d-1)}$ is invertible, and that $(\Theta_{\setminus m, \setminus m})^{-1} = \Sigma^{(m)}_{\setminus m, \setminus m}$ where
\begin{equation}\label{eq:GammatoSigma}
\Sigma_{ij}^{(m)} := \frac{1}{2} (\Gamma_{im} + \Gamma_{jm} - \Gamma_{ij}); \quad i, j \in V,
\end{equation}
see~\citet{H21} for details. In Algebraic Statistics, this set of matrices is also known as the Farris transform of $\Gamma$ \citep{FKE70}. Estimators of $\Sigma_{ij}^{(m)}$ can be obtained from $\hat \Gamma$ by simple plug-in. The sparsity pattern of $\Theta_{\setminus m, \setminus m}$ can in turn be estimated by leveraging algorithms for estimation of sparse inverse matrices that have been developed in the context of Gaussian graphical models. Note that for any single $m \in V$, this does not provide information on the complete graph, since by estimating the sparsity pattern of $\Theta_{\setminus m, \setminus m}$ we gain no information about potential edges that are connected to the $m$th node. At the same time, information about a potential edge $(i, j)$ is contained in the $d-2$ estimators corresponding to the sparsity pattern of $\Theta_{\setminus m, \setminus m}$ for $m \not= i,j$. This observation motivates the majority voting algorithm that we describe next.

\subsection{\eglearn{}: a majority voting algorithm}\label{sec:majority}

For a given $m \in V$ and estimator $\hat \Gamma$ of the variogram matrix $\Gamma$ (for instance the empirical variogram from~\eqref{eq:defempvario}), define the plug-in estimator $\hat\Sigma^{(m)}$ obtained by replacing $\Gamma$ by $\hat\Gamma$ in~\eqref{eq:GammatoSigma}. Consider an arbitrary algorithm $\cA$, called a base learner in what follows, that takes an estimated covariance matrix as input and returns an estimator of the sparsity pattern of its inverse. When given the sub-matrix $\hat\Sigma^{(m)}_{\setminus m, \setminus m}$ as input, the output $\hat Z^{(m)}$ of $\cA$ takes the form of a $(d-1) \times (d-1)$ symmetric adjacency matrix with entries $1$ in positions where $\Theta_{\setminus m, \setminus m}$ is estimated to be non-zero, and entries $0$ elsewhere.

Two examples of possible base learner algorithms are neighborhood selection \citep{MB06} and graphical lasso \citep{YL2007, FHT2008}; they are formally introduced in Section~\ref{sec:sparse} below. The base learner $\cA$ may require the choice of tuning parameters, as is the case for neighborhood selection and the graphical lasso, which can be fixed or data-dependent.

We note that neighborhood selection and graphical lasso are just two possible base learners and that a variety of alternative estimation methods exists. To keep the discussion concise, we focus on these two procedures since they are the easiest to describe, are simpler computationally due to the convexity of the penalty, and only require a single tuning parameter.

Augmenting the matrix $\hat Z^{(m)}$ with a row and column of zeros in the $m$th position, we obtain a $d \times d$ matrix $\tilde Z^{(m)}$. The entries of $\tilde Z^{(m)}$ outside its $m$th row and column are now considered as votes in favor or against certain edges in the graph $G$. Running the algorithm for each $m \in V$ results in $d$ such matrices. Those are then combined into a final graph estimator $\hat G = (V, \hat E)$ using majority voting: an edge $(i,j)$, $i\neq j$, is included in the final graph if and only if a $1$ appears in position $(i, j)$ of more than half of the $d-2$ matrices $\tilde Z^{(m)}$, $m \notin \{i, j\}$. The steps described above are formalized in the Algorithm~\ref{algo:majority}.

Figure~\ref{fig:majority} shows an illustration of the majority voting algorithm where the true underlying graph $G = (V,E)$ is the graph on the right-hand side of Figure~\ref{fig:intro}. In this example, the algorithm would output the true graph $G$ since exactly the true edges appear in the majority of the cases.

\begin{algorithm2e}[H]
	\KwIn{variogram estimate $\hat\Gamma$, base learner algorithm $\mathcal{A}$}
	\KwOut{extremal graph estimate $\hat G = (V, \hat E)$}
	initialize $\hat G := (V, \emptyset)$\;
	\For{$m \in V$}{
		obtain a $(d-1)\times(d-1)$ matrix $\hat Z^{(m)}$ from algorithm $\mathcal{A}$ with $\hat \Sigma^{(m)}_{\setminus m, \setminus m}$ as input\;
		obtain $\tilde Z^{(m)}$ by augmenting $\hat Z^{(m)}$ with a row and column of zeros in the $m$th position
	}
	\For{$i,j \in V$, $i < j$}{
		\If{$\tfrac{1}{d-2} \sum_{m=1}^d \tilde Z_{ij}^{(m)} > \tfrac{1}{2}$}{
			add an edge in $\hat E$ between nodes $i$ and $j$\;
		}
	}
	\caption{\eglearn{}: algorithm for learning general extremal graphical models.}
	\label{algo:majority}
\end{algorithm2e}

\begin{rem}\label{rem:wanzhou}
After the first version of the present paper appeared, an alternative proposal to estimating H\"usler--Reiss graphical models was put forward by \cite{wan2023graphical}. Specifically, those authors observe that the matrix $\Theta + c \bone\bone^\top$ is in fact invertible for any $c>0$ with inverse given by $\Sigma + \frac{1}{c d^2} \bone\bone^\top$ for a matrix $\Sigma$ that is a simple linear function of the extremal variogram matrix $\Gamma$. They then propose to estimate a sparse version of $\Theta$ by solving
\[
\mathop{\argmin}_{\Theta} \Big\{ - \log\det\Theta + \tr(\hat \Sigma 
\Theta) + \lambda \|\Theta - c\bone\bone^\top\|_{1,\mbox{off}} \Big\}
\]   
where $\|A\|_{1,\mbox{off}}$ is the sum of the off-diagonal entries of the matrix $A$. Utilizing our concentration results from Theorem~\ref{thm:concentration}, \cite{wan2023graphical} establish consistency of graph recovery and convergence rates for the resulting parameter estimates in their proposal.
\end{rem}

\begin{figure}[h]
  \centering
  \includegraphics[clip,height=3cm, page=5]{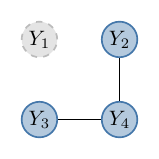}
  \includegraphics[clip,height=3cm, page=6]{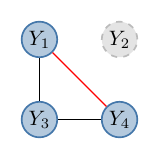}
  \includegraphics[clip,height=3cm, page=7]{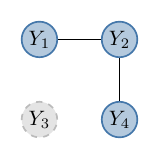}
  \includegraphics[clip,height=3cm, page=8]{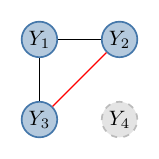}
  \caption{Illustration of the majority voting algorithm when the true underlying graph is the non-decomposable graph on the right-hand side of Figure~\ref{fig:intro}. Left to right: graphical representation of the estimated matrices $\tilde Z^{(m)}$, where the gray node $Y_m$ is not considered in the $m$th step, $m=1,\dots, 4$; black and red edges indicate correctly and incorrectly estimated edges, respectively.}
  \label{fig:majority}
\end{figure}

\subsection{Base learners for sparsity estimation} \label{sec:sparse}

Two classical methods from Gaussian graphical modeling to obtain sparse estimators of precision matrices are neighborhood selection \citep{MB06} and the graphical lasso \citep{YL2007}. The original theoretical guarantees for consistent structure recovery rely on Gaussian data and empirical covariances as input. Since the input estimator $\hat \Sigma^{(m)}_{\setminus m, \setminus m}$ for the base learner $\mathcal A$ in our \eglearn{} in Algorithm~\ref{algo:majority} uses neither Gaussian data nor the empirical covariance, we discuss in this section how the assumptions of sparse estimators for Gaussian distributions can be relaxed.  Related observations were made in \cite{liu2012high} for data that have a Gaussian copula but are not marginally Gaussian and \cite{loh2013} for discrete graphical models.

Throughout this section, we let $A \in \R^{p \times p}$ denote a symmetric, positive definite matrix and we are interested in the sparsity pattern of its inverse $B = A^{-1}$. We aim to use neighborhood selection and graphical lasso as the base learner algorithm $\mathcal{A}$ in the framework of our \eglearn{} in Algorithm~\ref{algo:majority}. As such, in the $m$th step of the algorithm, the matrix $A$ equals $\Sigma^{(m)}_{\setminus m, \setminus m}$ with $p=d-1$, and the interest is in the sparsity pattern of $B = \Theta_{\setminus m, \setminus m}$.

\subsubsection{Neighborhood selection}

Neighborhood selection was originally proposed by \cite{MB06} for estimating Gaussian graphical models. Although the motivation in \cite{MB06} relies on properties of multivariate normal distributions and their conditional independence, the underlying principle can be used to estimate the sparsity pattern of the inverse of a general symmetric matrix. Indeed, for the positive definite matrix $A \in \R^{p \times p}$  with inverse $B$, we have the representation
\[
\frac{-B_{\setminus\ell, \ell}}{B_{\ell\ell}} = (A_{\setminus\ell, \setminus\ell})^{-1}A_{\setminus\ell, \ell} =  \mathop{\arg\min}_{\beta \in \R^{p-1}} \big\{ -2A_{\ell, \setminus\ell} \beta + \beta^\top A_{\setminus\ell, \setminus\ell} \beta \big\}, \quad \ell = 1,\dots, p.
\]
The first equation follows from matrix computations using block inversion formulae \citep[see, for instance,][Equation~(C.4)]{L96} and the second from computing the gradient of the minimization problem. Hence, given access to an estimator $\hat A$, the sparsity pattern in $B_{\setminus\ell, \ell}$ can be estimated through the zero entries of
\begin{equation} \label{eq:deflasso}
	\hat\theta := \mathop{\arg\min}_{\theta \in \R^{p-1}} \big\{ -2\hat A_{\ell, \setminus\ell} \theta + \theta^\top \hat A_{\setminus\ell, \setminus\ell} \theta + \rho_\ell \|\theta\|_1 \big\},
\end{equation}
where $\rho_\ell$ denotes a penalty parameter and the $L^1$ penalty is used for enforcing sparse solutions. The set $\hat\nei(\ell)$ of indices of non-zero entries in $\hat\theta$ is then taken as an estimate of the non-zero pattern in the $\ell$th row of $B$. The procedure is repeated for each variable $\ell$. Since the matrix $B$ is symmetric, pairs $(i, j), (j,i)$ are added to the estimated set of non-zero entries if and only if $i \in \hat\nei(j)$ or $j \in \hat\nei(i)$; see Algorithm~\ref{algo:NS} below.

To link the above approach to neighborhood selection as proposed in \cite{MB06}, assume that $\hat \Sigma$ is the sample covariance matrix of $\g W_1,\dots, \g W_n$, a sample from a $p$-dimensional Gaussian distribution. \cite{MB06} regress $W_{1,\ell}, \dots, W_{n, \ell}$ on $\g W_{1, \setminus\ell}, \dots, \g W_{n, \setminus\ell}$ via the lasso, which amounts to solving the problem
\begin{align*}
\mathop{\arg\min}_{\beta \in \R^{p-1}} \Big\{ \frac{1}{n}\sum_{i=1}^n (W_{i,\ell} - \beta^\top \g W_{i, \setminus\ell})^2 + \rho_\ell \|\beta\|_1 \Big\}
= &\mathop{\arg\min}_{\beta \in \R^{p-1}} \big\{ -2\hat\Sigma_{\ell, \setminus\ell} \beta + \beta^\top \hat\Sigma_{\setminus\ell, \setminus\ell} \beta + \rho_\ell \|\beta\|_1 \big\}.
\end{align*}

\begin{algorithm2e}[H]
	\KwIn{Matrix $\hat A \in \R^{p \times p}$, penalty parameters $(\rho_\ell)_{\ell = 1,\dots,p}$}
	\KwOut{Estimate $\hat Z$ of sparsity pattern of $A^{-1}$}
initialize $\hat Z$ as a matrix of zeros\;
	\For{$\ell =1,\dots,p$}{
		$\hat\theta := \mathop{\arg\min}_{\theta \in \R^{p-1}} \big\{ -2\hat  A_{\ell, \setminus\ell} \theta + \theta^\top \hat A_{\setminus\ell, \setminus\ell} \theta + \rho_\ell \|\theta\|_1 \big\}$\;
		$\hat\nei(\ell) := \{j = 1,\dots,p: \theta_j \neq 0\}$\;
	}
	\For{$i = 1,\dots, p$, $j \neq i$}{
		\If{$i \in \hat\nei(j)$ or $j \in \hat\nei(i)$}{
			set $\hat Z_{ij} = \hat Z_{ji} = 1$\;
		}
	}
	\caption{Sparsity pattern estimation through neighborhood selection.}
	\label{algo:NS}
\end{algorithm2e}

\subsubsection{Graphical lasso}

As an alternative to running node-wise regressions as required for neighborhood selection, \cite{YL2007} suggested to estimate the precision matrix $B = A^{-1}$ via penalized maximum likelihood, with a penalty on the off-diagonal, element-wise $L^1$ norm of the matrix:
\begin{equation}\label{eq:GL}
\mathop{\arg\min}_{Q \in \mathcal{P}^{p-1}} \big\{ -\log\det Q + \tr(\hat A Q ) + \rho \| Q \|_{1, \rm off} \big\}, \quad \| Q \|_{1, \rm off} := \mathop{\sum\sum}_{i \neq j} |Q_{ij}|,
\end{equation}
where $\hat A$ denotes an estimator of the matrix $A$. The name ``graphical lasso'' was subsequently given to the algorithm that efficiently solves this problem \citep{FHT2008}.

In the original procedure of \cite{YL2007}, $\hat A$ denotes the sample covariance matrix of an i.i.d.~sample of $N(0,A)$ random vectors. Subsequently, \cite{RWRY11} proved that Gaussianity is not needed and the sparsity pattern of precision matrices can be recovered consistently as long as the data used to compute the empirical covariance matrix satisfy certain tail properties. A close look at the analysis of \cite{RWRY11} further reveals that there is nothing special about empirical covariance matrices. More precisely, under certain technical conditions, the optimization in~\eqref{eq:GL} yields an estimator of $A^{-1}$ with the correct sparsity pattern, provided that $\hat A$ is any estimator of $A$ that is close to $A$ in element-wise sup-norm. These claims will be made precise in Proposition~\ref{prop:GL} of the Supplementary Material.

\subsection{Consistency results} \label{sec:graphrecovery}

In this section, we provide theoretical guarantees for Algorithm~\ref{algo:majority} to correctly recover the true extremal graph $G$ when either neighborhood selection (Theorem~\ref{th:NS}) or graphical lasso (Theorem~\ref{th:GL}) are used as base learners. Complete proofs of the main results, Theorems~\ref{th:NS} and \ref{th:GL}, are given in the Supplementary Material.

The most important tool is \cref{thm:concentration}, which controls the probabilistic error of the input to \cref{algo:majority}, the empirical variogram. As a second ingredient, we establish sparsistency guarantees on the estimators provided by neighborhood selection and graphical lasso when the input matrices are not empirical covariances of i.i.d.~Gaussian data, but rather arbitrary estimators that are sufficiently close to their population versions in element-wise sup-norm. This is a  straightforward consequence of the analysis in \cite{RWRY11} for the graphical lasso but requires more work in the case of neighborhood selection, since classical arguments in~\cite{MB06} explicitly rely on properties of the multivariate normal distribution. These guarantees are the object of Propositions~\ref{prop:NS} and \ref{prop:GL}. They are stated and proved in the Supplementary Material.

Depending on the base learner used in Algorithm~\ref{algo:majority}, additional assumptions are needed on the matrices $\Sigma^{(m)}$ that parametrize the distribution of $\bY$.

\subsubsection{Neighborhood selection as base learner}

Let $s$ denote the maximal degree, that is, the largest number of edges connected to any node, of the graph $G$. For $m, \ell \in V$, $m \neq \ell$, define the set of all nodes, except for node $m$, that are connected to node $\ell$ as
\[
	S_{m, \ell} := \{i \in V \setminus \{m,\ell\} :  \Theta_{i\ell} \neq 0\} \subset V \setminus \{m\}
\]
and its complement $S_{m, \ell}^c$ taken in $V \setminus \{m\}$. Define the quantities $\theta_{\min}^\ns := \min_{i, \ell: \Theta_{i\ell} \neq 0} |\Theta_{i\ell}|/\Theta_{\ell\ell}$,
\begin{align*}
\mu &:= \min_{m,\ell \in V, m \neq \ell} \lambda_{\min} ( \Sigma^{(m)}_{S_{m,\ell}, S_{m,\ell}}),
\\
\kappa &:= \max_{m,\ell \in V, m \neq \ell} \normop{ \Sigma^{(m)}_{S_{m,\ell}^c, S_{m,\ell}}}_\infty,
\\
\vartheta & := \max_{m,\ell \in V, m \neq \ell} \normop{ ( \Sigma^{(m)}_{S_{m,\ell}, S_{m,\ell}})^{-1}}_\infty.
\end{align*}
Additionally, consider the neighborhood selection incoherence parameter
\begin{equation}
	\eta^{\text{ns}} := \min_{m,\ell \in V, m \neq \ell}\eta^{\text{ns}}_{m,\ell}, \quad \eta^{\text{ns}}_{m,\ell} := 1 - \normop{\Sigma^{(m)}_{S_{m,\ell}^c, S_{m,\ell}} (\Sigma^{(m)}_{S_{m,\ell}, S_{m,\ell}})^{-1}}_{\infty}.
\end{equation}
Incoherence parameters of this sort are known to be a crucial ingredient for support recovery via the lasso \citep{ZY06, MB06}. In our theory below we will assume that $\eta^{\text{ns}}$ is strictly positive; a technical relaxation to requiring sufficiently many $\eta^{\text{ns}}_{m,\ell}$ being positive is shown in the proof of \cref{th:NS}.

The majority voting in \Cref{algo:majority} applies the base learner algorithm to $d$ distinct problems, namely for every $m \in V$. Using neighborhood selection as the base learner in turn requires the choice of $d-1$ tuning parameters, resulting in a total of $d(d-1)$ tuning parameters $\rho^\ns_{m,\ell}, m \in V, \ell \in [d-1]$, where $\rho^\ns_{m,1},\dots,\rho^\ns_{m,d-1}$ correspond to the tuning parameters in the $m$th step of \cref{algo:majority}. Define
$
\rho^\ns_{\rm min} := \min_{m, \ell} \rho^\ns_{m,\ell}
$
and
$
\rho^\ns_{\rm max} := \max_{m, \ell} \rho^\ns_{m,\ell}
$.

\begin{thm} \label{th:NS}
Assume that $\bY$ is a H\"usler--Reiss distribution with parameter matrix $\Gamma$ and let $G$ be its conditional independence graph. Let $G^\ns(\hat\Gamma)$ denote the estimated graph obtained through neighborhood selection in Algorithm \ref{algo:majority} with penalty parameters $\rho_{m, \ell}^\ns$ and arbitrary input matrix $\Gam$. Then, as soon as $\eta^{\text{ns}}> 0$ and
\[
	\big\| \Gam - \Gamma \big\|_\infty < C^\ns := \frac{2}{3} \min\bigg\{
	\frac{\mu}{2s},
	\frac{\eta^{\text{\rm ns}}}{4\vartheta (1 + \kappa \vartheta) s},
	\frac{\theta_{\min}^{\text{\rm ns}} - \vartheta \rho^\ns_{\rm max}}{2 \vartheta (1 + \kappa \vartheta)},
	\frac{\rho^\ns_{\rm min} \eta^\text{\rm ns}}{8(1 + \kappa \vartheta)^2}
	\bigg\},
\]	
we have $G^\ns(\hat\Gamma) = G$.

Now suppose Assumption~\ref{assum:tail} holds with $R_J$ in there corresponding to $\bY$. Recall the estimator $\Gam$ in~\eqref{eq:defempvario} and the role of the parameter $k$ in this estimator.
Assuming that $n^\zeta \leq k \leq n/2$ for some $\zeta > 0$, $\log d = o(k/(\log k)^8)$ and the quantities $\mu, \kappa, \vartheta, \eta^{\text{\rm ns}}$ are bounded away from zero and infinity,
\[
	\Prob\big(G^\ns(\hat\Gamma) = G\big) \to 1, \quad n \to \infty,
\]
as long as $\rho^\ns_{\rm max} < \theta_{\min}^{\text{\rm ns}}/(2\vartheta)$ and
$
(k/n)^{\xi} (\log(n/k))^2 + \sqrt{\log d}/\sqrt{k} = o(\min(\rho^\ns_{\rm min}, s^{-1})).
$
\end{thm}

In the statement of the second part of the above theorem, we sacrifice some generality for the sake of simplicity. Combining the first statement with the general concentration bounds on $\|\Gam - \Gamma\|_\infty$ in \cref{thm:concentration}, one can obtain lower bounds on the probability of correct graph recovery that are explicit in all constants appearing above. We opt against providing such explicit expressions because the resulting terms are lengthy and do not add much in terms of interpretability. In the same vein, the quantities $\mu, \kappa, \vartheta, s, \eta^{\text{ns}}$ are for simplicity taken as the worst case over $m, \ell$. It is possible to introduce versions of $\mu, \kappa, \vartheta, s, \eta^{\text{ns}}$ that depend on $m,\ell$. This would allow for sharper but more complex results; in particular, the incoherence parameters $\eta^\ns_{m, 1}, \dots, \eta^\ns_{m, d-1}$ would need to be non-negative for only half of the values $m \in V$. The precise form of this statement is immediate from a close look at the proof in \cref{sec:proofeglearn} of the Supplementary Material.

For interpreting the second part of the above theorem, note that $r_{k,n} := (k/n)^{\xi} (\log(n/k))^2 + (k^{-1}\log d)^{1/2}$ is the order at which $\| \Gam - \Gamma \|_\infty$ concentrates with $(k/n)^{\xi} (\log(n/k))^2$ corresponding to the bias and $(k^{-1}\log d)^{1/2}$ to the stochastic part; see Theorem~\ref{thm:concentration} and the discussion right after for additional details. In particular, with the optimal choice of $k$ derived therein, we recall that $r_{k,n}$ is of the order of $(n^{-1}\log d)^{\xi/2\xi+1}$ up to a logarithmic factor in $n$. The quantity $\theta_{\min}^{\text{\rm ns}}$ can be interpreted as minimal signal strength among edges that are present in the graph. In order for such edges to be recovered with high probability, we need a minimal signal condition $\theta_{\min}^\ns \gg r_{k,n}$. Similarly, the maximal edge degree $s$ must satisfy $s^{-1} \gg r_{k,n}$. Some conditions on minimal signal strength and maximal node degree are imposed in all of the existing literature on sparsistency of estimators for graphical models.

Finally, \cref{prop:equiv} in the Supplementary Material and the discussion that follows it provide a way to verify \cref{assum:tail} from \cref{assum:RJstand} for data in the domain of attraction of H\"usler--Reiss distributions.

\subsubsection{Graphical lasso as base learner}

Define the set of edges in the graph excluding edges containing node $m$, and augmented by self loops, by
\[
	S_m := \{(i, j): i, j \in V \setminus \{m\}, \Theta_{ij} \neq 0\},
\]
as well as its complement $S_m^c$ taken in $(V \setminus \{m\})^2$. Let $\Omega^{(m)} := \Sigma^{(m)} \otimes \Sigma^{(m)}$ and define the quantities $\theta_{\min}^\gl :=  \min_{i \neq j: \Theta_{ij} \neq 0} |\Theta_{ij}|$, $\kappa_\Sigma := \max_m \opnorm{\Sigma^{(m)}}{\infty}$ and $\kappa_\Omega := \max_m \opnorm{(\Omega^{(m)}_{S_m, S_m})^{-1}}{\infty}$. The incoherence parameter for graphical lasso is defined as 
\begin{align}\label{ext_vario}
	\eta^\gl := \min_{m \in V} \eta_m^\gl, \quad \eta_m^\gl := 1 - \normop{\Omega^{(m)}_{S_m^c, S_m} (\Omega^{(m)}_{S_m, S_m})^{-1}}_{\infty}.
\end{align}
Similarly to neighborhood selection, such incoherence parameters play a crucial role in support recovery guarantees for the Gaussian graphical lasso \citep{RWRY11}.

Using graphical lasso as the base learner in \cref{algo:majority} requires the choice of $d$ tuning parameters $\rho^\gl_1, \dots, \rho^\gl_d$, one for each graphical lasso that is applied in the $m$th step of the loop. Define
$\rho^\gl_{\rm min} := \min_{m \in V} \rho^\gl_{m}$ and $\rho^\gl_{\rm max} := \max_{m \in V} \rho^\gl_{m}.$ 

\begin{thm} \label{th:GL}
Assume that $\bY$ is a H\"usler--Reiss distribution with parameter matrix $\Gamma$ and let $G$ be its conditional independence graph. Let $G^\gl(\hat\Gamma)$ denote the estimated graph obtained through the graphical lasso as base learner in Algorithm \ref{algo:majority} with penalty parameters $\rho_1^\gl, \dots, \rho_d^\gl$ and arbitrary input matrix $\Gam$. Then, as soon as $\eta^{\text{\rm gl}} > 0$ and
\[
\big\| \Gam - \Gamma \big\|_\infty < C^\gl := \frac{2}{3}\min\bigg\{\min_{i, m \in V, i \neq m} \Sigma_{ii}^{(m)},\frac{\eta^{\text{\rm gl}}\rho^\gl_{\rm min}}{8}, \frac{1}{\chi_0 s} - \rho^\gl_{\rm max}, \frac{\theta_{\min}^{\text{\rm gl}}}{4\kappa_\Omega} - \rho^\gl_{\rm max} \bigg\},
\]
for $\chi_0 := 6\kappa_\Sigma\kappa_\Omega(1\vee (9\kappa_\Sigma^2\kappa_\Omega/\eta^{\text{\rm gl}}))$, we have $G^\gl(\hat\Gamma) = G$.

Now suppose Assumption~\ref{assum:tail} holds with $R_J$ in there corresponding to $\bY$. Recall the estimator $\Gam$ in~\eqref{eq:defempvario} and the role of the parameter $k$ in this estimator. Assuming that $n^\zeta \leq k \leq n/2$ for some $\zeta > 0$, $\log d = o(k/(\log k)^8)$ and that the quantities $\min_{i \neq m} \Sigma_{ii}^{(m)}, \kappa_\Sigma, \kappa_\Omega, \eta^\gl$ are bounded away from zero and infinity,
\[
\Prob\big(G^\gl(\hat\Gamma) = G\big) \to 1, \quad n \to \infty,
\]
as long as $\rho^\gl_{\rm max} < (2\chi_0 s)^{-1} \wedge (\theta_{\min}^\gl/8\kappa_\Omega)$ and
$
(k/n)^{\xi} (\log(n/k))^2 + \sqrt{\log d}/\sqrt{k} = o(\rho^\gl_{\rm min}).
$  
\end{thm}

Similarly to the statements in \Cref{th:NS}, we opt for simplicity over generality. It is possible to obtain sharper statements for the second part by combining the statement in the first part with concentration results on the empirical variogram. Moreover, it suffices if a majority of the incoherence parameters $\eta^\gl_1,\dots,\eta^\gl_d$ are non-negative. 

Akin to $\theta_{\min}^{\text{\rm ns}}$, the quantity $\theta_{\min}^{\text{\rm gl}}$ corresponds to a minimal signal strength condition. Both quantities are of the same order provided that all diagonal entries of $\Theta^{(m)}$ are bounded away from zero and infinity for all $m$. In \cref{th:GL}, we find that assuming all other parameters fixed, $\theta_{\min}^\gl \gg r_{k,n}$ and $s^{-1} \gg r_{k,n}$ are required for consistent graph recovery with graphical lasso as the base learner. This matches the requirements when neighborhood selection is the base learner; see the discussion following Theorem~\ref{th:NS}.

Finally, note that the required \cref{assum:tail} in \cref{th:GL} can be verified from \cref{assum:RJstand}, as discussed after~\cref{th:NS}.

\begin{rem}
One important difference between the assumptions for consistent graph estimation via neighborhood selection and graphical lasso lies in the definition of the corresponding incoherence parameters. While there are no general results stating that one parameter is always smaller than the other, simulations indicate that in the models we considered, the neighborhood selection incoherence parameter is more likely to be positive; see Section~\ref{sec:block_graph} in the Supplementary Material for additional details. This is in line with the discussion in \citet[][Sections~3.1.1, 3.1.2]{RWRY11} who show that in two examples, conditions required for consistent graph recovery via neighborhood selection are weaker than those required for graphical lasso. As pointed out by a referee, incoherence assumptions can be avoided by using non-convex penalties \citep{fan2001variable,loh2017support}. It is also possible to avoid incoherence assumptions through the use of thresholding \citep{wang2023thresholded} or by adaptive reweighting similar to the adaptive lasso \citep{zou2006adaptive}. 

The main reason we decide to stick with the $L^1$ penalty is that it is a simple procedure with a single tuning parameter that does not require non-convex optimization. Thresholding or adaptive re-weighting (similar to the adaptive lasso) requires the choice of a second tuning parameter. Non-convex optimization, as required by SCAD or MCP, can be tricky in high dimensions and many of the available theoretical guarantees concern the existence of a sequence of minimizers with desirable statistical properties. When working on a single regression problem, non-uniqueness of the solution can be diagnosed and different solutions can be compared. This is a process, however, that can hardly be automated inside our algorithm.
\end{rem}

\section{Simulations}\label{sec:sims}

\subsection{Simulation setup}

We conduct several simulation studies to compare the performance and properties of different structure learning methods for extremal graphs. Two classes of H\"usler--Reiss distributions $\g Y$ are chosen as the true extremal graphical model. They are described below by first sampling a random graph structure $G=(V,E)$ and then generating a random parameter matrix $\Gamma$ that factorizes on that graph. Using the exact method of \citet{DEO2016}, we then simulate $n$ samples of a max-stable random vector $\g X$ associated to $\g Y$  \citep[cf.,][]{roo2018}, whose copula is
\begin{align}\label{max_stable}
  \Prob(F(\bX) \leq \bx) = \exp\{-\Prob(\g Y \not \leq - 1/\log \g x)/\Prob(Y_1 > 1)\}, \quad \bx \in [0, 1]^d.
\end{align}
It is shown in \cref{sec:ms} of the Supplementary Material that this distribution satisfies \cref{assum:RJstand} with $\xi' = 1$. Hence by \cref{prop:equiv} and \cref{lemm:HRdensity}, it satisfies \cref{assum:tail} with any $\xi<1$. In particular, it is in the domain of attraction of $\g Y$.

\newcommand{\BA}{\mathrm{BA}}

As the first random graph $G = (V,E)$ we consider the Barabasi--Albert model denoted by $\BA(d, q)$, which is a preferential attachment model with $d$ nodes and a degree parameter $q \in \mathbb N$ \citep{alb2002}.
We consider $d \in \{20, 50, 100\}$ and the parameters $q=1$ (which is a tree) and $q=2$. Figure~\ref{fig:BA} in the Supplementary Material shows realizations of the $\BA(100, 1)$ and $\BA(100, 2)$ models. In each simulation setup, we use $k = \lfloor n^{0.7} \rfloor$ to compute the empirical variogram $\hat\Gamma$, which satisfies the assumptions of \cref{thm:concentration}, where each time the sample size $n$ is chosen so that $k/d \in \{0.5, 1, 2.5\}$ (when $q=1$) and $k/d \in \{0.5, 1, 5\}$ (when $q=2$). The case of $k/d=0.5$ is a high-dimensional setting since there are fewer effective samples than variables.

In order to randomly generate a valid H\"usler--Reiss parameter matrix $\Gamma$ on $G$, we use the scheme in \cite{ying2021minimax} to sample a weighted graph Laplacian matrix.
The latter can be used as a H\"usler--Reiss precision matrix $\Theta$, which then corresponds uniquely to a parameter matrix $\Gamma$; see Example~\ref{ex:HR}. Specifically, we sample for every undirected edge $\{i,j\} \in E$ of $G$ an independent random variable $U_{ij} \sim \text{Unif}[2,5]$, and define the matrix $W \in \R^{d \times d}$ by
\begin{align*}
W_{ij} = W_{ji} :=
\begin{cases}
  U_{ij} \text{ if } \{i,j\} \in E, \\
  0 \text{ otherwise.}
\end{cases}
\end{align*}
Let $D$ be the diagonal degree matrix with entry $D_{ii}$ given by the $i$th row sum of $W$, $i \in V$. The matrix $\Theta = D - W$ is called the weighted Laplacian matrix over the graph $G$ and is a valid H\"usler--Reiss precision matrix \citep{rottger2023total}.

We note that this construction always results in a precision matrix satisfying
\begin{align}\label{emtp2}
  \Theta_{ij} \leq 0, \quad i,j \in V, i\neq j.
\end{align}
By \citet[][Lemma 4.5]{rottger2023total} this implies that the corresponding H\"usler--Reiss distribution is $\text{EMTP}_2$, a notion of positive dependence for multivariate Pareto distributions. While encountered frequently in multivariate extreme value models \citep[see][Section 4]{rottger2023total}, such positive dependence is not present in all H\"usler--Reiss distributions.

As a second model for $G$ and $\Gamma$ we therefore consider a setup where $\Theta_{ij} > 0$ for some $i,j \in V$. 
The construction method is described in Section~\ref{sec:block_graph} in the Supplementary Material.  The resulting block graph $G$ induces a block-like structure on the precision matrix $\Theta$ \citep{H21}; see the right-hand side of Figure~\ref{fig:BA}.

\subsection{Competing methods and evaluation}
\label{sim:compete}

We apply several methods for structure estimation to the simulated data. All methods are based on the empirical extremal variogram $\hat \Gamma$ defined in~\eqref{eq:defempvario}. For a given sample size $n$, this estimator uses only the $k$ largest exceedances in each variable. Throughout the simulation study, we choose $k = \lfloor n^{0.7}\rfloor$, which satisfies the assumptions for our theory. We refer to \citet[][Section 5]{EV22} for a detailed study of the choice of $k$ in the framework of structure learning for extremal trees, which applies in the same manner here.
In practice, $k$ can be chosen based on stability plots of the estimated entries of $\hat \Gamma$, similar to the Hill plot \citep[e.g.,][]{dre2000}. Note that this strategy separates the choice of $k$ from the selection of the penalty parameter used in \eglearn{}, thus avoiding the simultaneous tuning of two hyperparameters.

The first estimator of the extremal graph structure $G$ is the extremal minimum spanning tree $T_{\MST}$ introduced in \cite{EV22}. For a given estimate $\hat \Gamma$ of the extremal variogram, it is defined as
\begin{align}\label{Tmin}
  T_{\MST} = \mathop{\argmin}_{T = (V,E)} \sum_{(i,j) \in E} \hat \Gamma_{ij},
\end{align}
where the minimum is taken over all tree structures $T$. In \cite{EV22} it is shown to consistently recover an underlying tree even in high dimensions. By construction, this always results in a tree and hence cannot be consistent for more general graphs.

A second method is from \cite{rottger2023total} who obtain an $\text{EMTP}_2$ estimator of $\Theta$ as the maximizer of the convex objective
\begin{align}\label{nllk}
	{\log\det}^* \Theta + \frac12 \tr(\hat \Gamma \Theta)
\end{align}
over all H\"usler--Reiss precision matrices satisfying the $\text{EMTP}_2$ constraint~\eqref{emtp2}. Here, the pseudo-determinant ${\det}^*$ of a matrix is defined as the product of its non-zero eigenvalues. With $\hat\Gamma$ the empirical variogram defined in \cref{eq:defempvario}, \cref{nllk} can be seen as a pseudo H\"usler--Reiss log-likelihood. This method sometimes introduces sparsity, but it is important to note that it is not designed for structure estimation. While \cite{rottger2023total} show consistency of the estimates of $\Gamma$ and $\Theta$ in a fixed--dimensional setting, there are no guarantees for consistent graph recovery of $G$, even under the $\text{EMTP}_2$ assumption. The method should thus not be considered as a direct competitor but is included for comparison.

In Section~\ref{sec:majority} we introduced our \eglearn{} algorithm that uses majority voting for structure estimation of general extremal graphical models. It can either be combined with neighborhood selection or graphical lasso as the base learning method. Both methods depend on collections of tuning parameters, which are denoted as $\rho_{m, \ell}^\ns$ and $\rho_m^\gl$, respectively, $m \in V$, $\ell \in [d-1]$.
For simplicity, we set them all to the same value $\rho$. We thus obtain for each of the two base learners a path of estimated graphs indexed by $\rho$, ranging from dense to sparse graphs as $\rho$ increases; see Figure~\ref{fig:lasso_path} in the Supplementary Material for examples of paths resulting from the two base learners, where the true graph was generated from the $\BA(20, 2)$ model. As a benchmark, we consider an oracle version of our estimator by selecting $\rho$ to minimize the evaluation metric~\cref{Fscore} below among all connected graphs in the estimated paths.

In practice, it is often necessary to select the amount of sparsity of the graph structure (or equivalently, the hyperparameter $\rho$) in a data driven way. For this, we propose three tuning strategies termed AIC, BIC and MBIC. They are based on the traditional Akaike (AIC) and Bayesian (BIC) information criteria, and on a modified BIC (MBIC) tailored to high-dimensional model selection \citep{WLL09}. In order to implement those, we need to estimate the parameter matrix of the H\"usler--Reiss distribution constrained to each graph in the obtained path. \cite{H21} show that for a connected graph with edge set $\hat E$, the maximizer of \cref{nllk} under the constraint that $\Theta_{ij} = 0$ for all $(i, j) \notin \hat E$ can be obtained by solving a matrix completion problem on $\hat\Gamma$. Using their implementation, we compute constrained estimators $\hat\Theta^{\hat E}$. The AIC, BIC and MBIC graphs then maximize
\[
	L(\hat\Theta^{\hat E}) - 2|\hat E|, \quad L(\hat\Theta^{\hat E}) - \log(n_e) |\hat E|, \quad L(\hat\Theta^{\hat E}) - \log(n_e)\log\log(d) |\hat E|,
\]
respectively, where $\hat E$ ranges over the path of estimated graphs and $L$ denotes the H\"usler--Reiss log-likelihood based on the extreme pseudo-observations
\[
	\bigg\{ \frac{k/n}{1 - \tilde F(\g X_t) + 1/n} : \max_{i \in V} \tilde F_i(X_{ti}) > 1 - k/n \bigg\},
\]
of which there are $n_e \in [k, n]$; recall the definition of the function $\tilde F$ in \cref{sec:empvario} and the H\"usler--Reiss density in \cref{ex:HR}.

In every simulation, the estimated graph $\hat G = (V, \hat E)$ is compared with the true underlying graph $G = (V, E)$, based on the $F$-score. It is defined as
\begin{align}\label{Fscore}
  F = \frac{|E \cap \hat E| }{|E \cap \hat E| + \frac12 (|E^c \cap \hat E| + |E \cap \hat E^c|)},
\end{align}
where $E^c$ denotes all possible undirected edges on $V\times V$ except for those in the edge set $E$.

\subsection{Results}

The results for the $\BA(d, q)$ models in dimensions $20$ and $100$ are shown in \cref{fig:d20,fig:d100} as box-plots of the $F$-scores for 100 repetitions of each experiment; see \cref{sec:d50} of the Supplementary Material for similar results in dimension $50$. As expected, the minimum spanning tree tends to outperform the other methods when the underlying model is indeed a tree ($q=1$). In the case where $q=2$, however, this method is limited by the tree constraint and it is no longer competitive, especially as the sample size grows. The $\text{EMTP}_2$ estimator does not recover the graph well, especially in high dimensions, which is not surprising given that it is not designed for structure learning.

Turning to our methods, we first observe that \eglearn{} with graphical lasso as the base learner does not perform well on any of the simulations, even with the oracle value for the penalty parameter~$\rho$. This is surprising since the graphical lasso for Gaussian distributions is a well established method. Neighborhood selection outperforms graphical lasso as the base learner in every scenario; with oracle tuning, it outperforms every other method except the minimum spanning tree when the underlying graph is a tree. 

As pointed out by the Associate Editor, thresholding can improve the performance of graphical lasso in terms of graph recovery \citep[see, e.g.,][]{wang2023thresholded}. We tried the thresholded graphical lasso as base learner in simulations and observed that it indeed improves graph recovery performance but is still not better than neighborhood selection. Since the latter requires only one tuning parameter, we omit additional results on thresholded graphical~lasso.

For the reasons above, we only consider data driven hyperparameter tuning with neighborhood selection as the base learner. The filled box-plots in each scenario in \cref{fig:d20,fig:d100} show the performance of the AIC, BIC and MBIC estimated graphs as defined in Section~\ref{sim:compete}. We first observe that the MBIC is consistently at least as good as the other two, and that it is generally competitive with the graph estimated via oracle tuning, with the exception of the high-dimensional tree model $\BA(100, 1)$. It appears that in this model, the oracle graph does not fit the data particularly well, lowering its likelihood and hence its MBIC value, but is still close to the underlying tree in terms of $F$-score due to the sparsity of that graph.

In Section~\ref{sec:block_graph} we discuss additional simulations with H\"usler--Reiss distributions generated according to the block model. The results provide insight into the difference between graphical lasso and neighborhood selection base learners and the connection to incoherence~parameters.

\begin{figure}[h]
  \includegraphics[scale=.85]{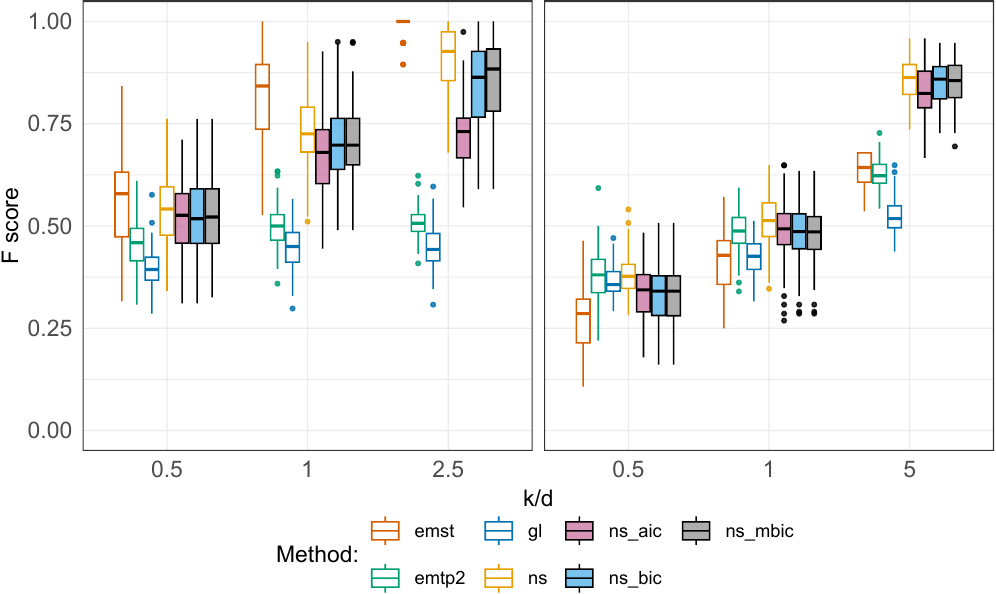}
  \caption{Boxplots of 100 repetitions of the $F$-scores of different methods fitted to data from the model $\BA(d, q)$ of degree $q=1$ (left) and $q=2$ (right) and in dimension $d = 20$.}
  \label{fig:d20}
\end{figure}

\begin{figure}[h]
  \includegraphics[scale = 0.85]{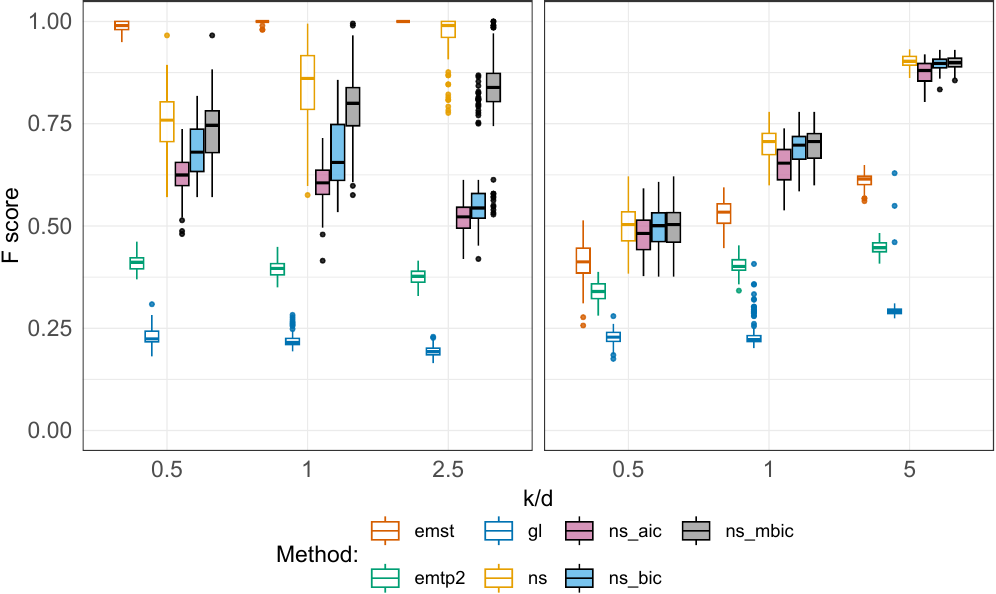}
  \caption{Boxplots of 100 repetitions of the $F$-scores of different methods fitted to data from the model $\BA(d, q)$ of degree $q=1$ (left) and $q=2$ (right) and in dimension $d = 100$.}
  \label{fig:d100}
\end{figure}

\section{Applications} \label{sec:appl}

We apply our \eglearn\ algorithm to a hydrological data set related to flood risk assessment. We focus on learning the extremal graph structures; for comments on how the estimated graphs can be used for subsequent inference in the hydrological example, we refer to the papers cited below. We conducted a similar analysis on a financial data set related to systemic risk; the results are found in \cref{sec:currency} of the Supplementary~Material.

The hydrological data set was introduced in \cite{eng2014b} and has been used as a benchmark in multivariate extremes for dependence modeling \citep{EH2020, rottger2023total} and causal inference \citep{mha2020,gne2021,tra2021}. It contains daily discharge measurements at $d=31$ stations in the upper Danube basin over 50 years. We follow the original paper and use the pre-processed data containing $n=428$ observations after a declustering of the $d$ time series.
The $n$ samples can be considered as independent observations from a $d$-dimensional random vector $\bX$, and we are interested in understanding the extremal dependence structure of $\bX$.
A certain amount of information can be deduced from the physical properties of the river network; the flow connections are shown in the left panel of Figure~\ref{fig:danube}. It is reasonable to expect that extremal dependence is strong along these connections. Due to other meteorological processes such as precipitation, there may be additional dependence between flood events that is not entirely explained by flow connections.

We first compute the empirical variogram $\hat \Gamma$ with $k = \lfloor n^{0.7} \rfloor = 69$, which corresponds to how $k$ is chosen in our simulation studies. We then use \eglearn\ with neighborhood selection as the base learner, a suitable grid of penalty parameters $\rho$ and input variogram estimate $\hat \Gamma$.
Preliminary simulations, as well as prior work \citep{mha2020}, suggest that the stations 23--27 are asymptotically independent of the remaining locations. This can be expected: considering the geographical map of the river basin \citep[][Figure 1]{eng2014b}, these are the only stations in the North of the Danube river, which are not fed by rainfall over the Alps. We therefore choose to remove those five stations and study the remaining 26 variables corresponding to stations along the river flow graph shown in the middle panel of \cref{fig:danube}.
The graph structure obtained by \eglearn{} and selected by the MBIC is shown in the right panel of Figure~\ref{fig:danube}. We first observe that it is fairly sparse and close to the flow connection tree. In particular, it contains almost all the flow connections as edges, confirming the water flow as main source of extremal dependence. However, as conjectured, there is evidence of a strong dependence between several alpine stations at the bottom of the graph that are not directly flow connected, which is potentially explained by extreme rainfall.

\begin{figure}[h]
\centering
\includegraphics[width=0.3\textwidth, trim={5em 5em 5em 5em}]{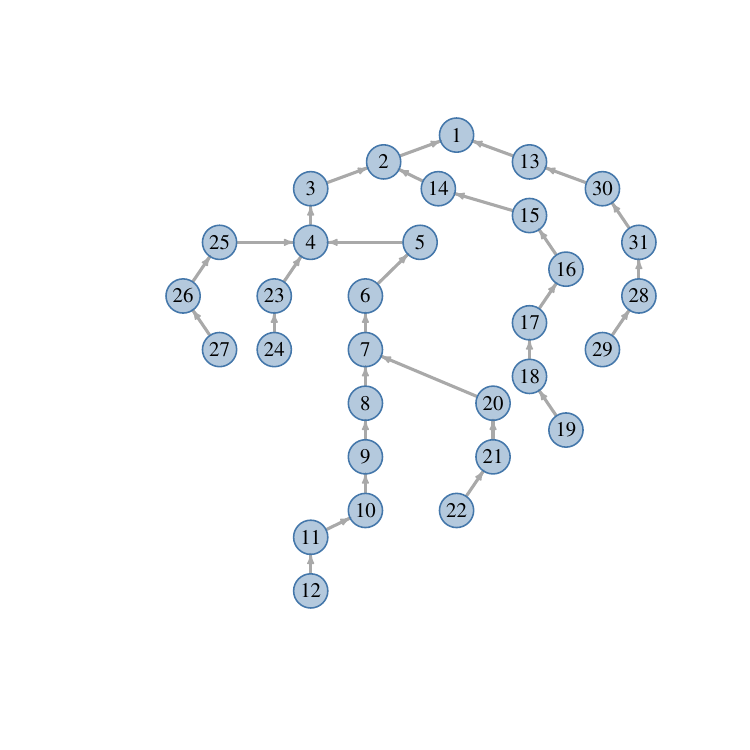}
\hspace{1em}\includegraphics[width=0.3\textwidth, trim={5em 5em 5em 5em}]{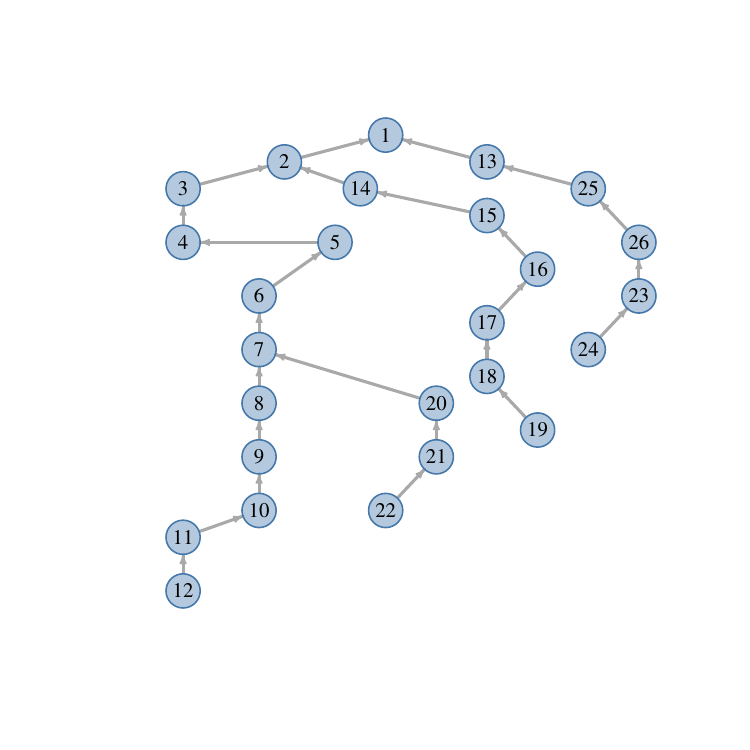}
\hspace{1em}\includegraphics[width=0.3\textwidth, trim={5em 5em 5em 5em}]{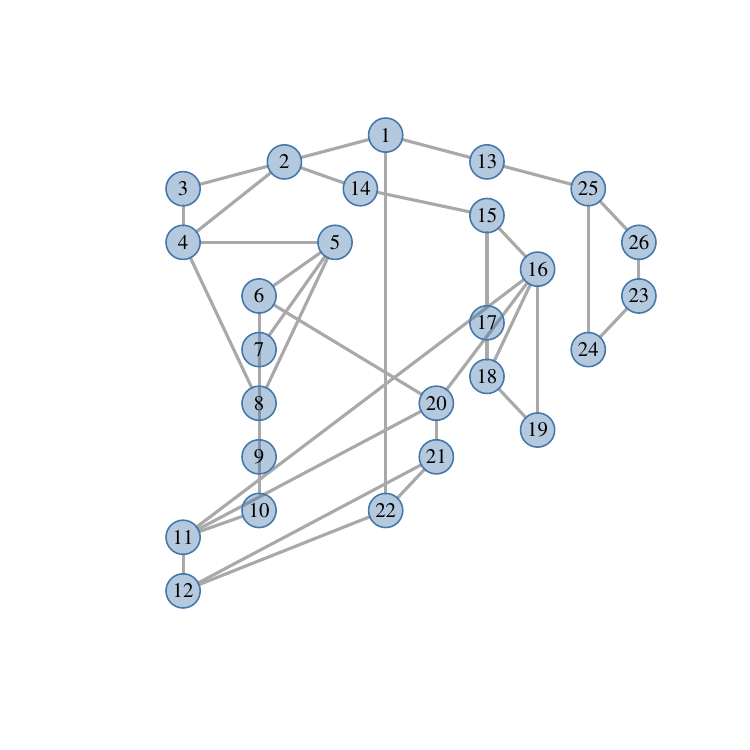}
\caption{The physical flow connection tree corresponding to the full Danube data (left) and to the data without stations 23--27 (middle), and the selected graph by MBIC estimated by \eglearn\ (right).}
\label{fig:danube}
\end{figure}

While the information criteria provide a good heuristic for selecting a unique and connected graphical model, the whole path of estimated graphs corresponding to all penalty parameters $\rho$ offers a more complete view on the structure of the data. Given a target sparsity level, one can choose a suitable model along this path. We report in \cref{sec:graphs_ic} of the Supplementary Material three estimated graphs corresponding to different $\rho$ parameters. The first is much denser than the MBIC graph, whereas the second is slightly sparser. The third, which is the sparsest connected graph that was output by \eglearn{}, is very similar to the flow connection~tree.

Further model evaluation, based on splitting the data sets into training and test subsets, is presented in \cref{sec:graphs_ic} of the Supplementary Material. We find that the test H\"usler--Reiss likelihood on the graphs learned by \eglearn{} compares favorably with that based on the (very dense) full model or on (sparse) tree models. We also provide estimated graphs obtained by applying standard tools for Gaussian graphical models to the full data. Interestingly, \eglearn{} uncovers different graphs than the classical methods. This can be explained by different conditional dependence structures in the bulk and tail and highlights the importance of tools that are specialized for the analysis of tail dependence. Similar conclusions are drawn regarding the financial data; see \cref{fig:likelihoods,fig:vanilla} and the related discussions in \cref{sec:graphs_ic}.

\section{Extensions and future work} \label{sec:future}

We have introduced a general methodology for estimating H\"usler--Reiss graphical models through \texttt{EGlearn} and provided a thorough theoretical analysis of the resulting procedure. This is the first principled approach for estimating extremal graphical models with arbitrary structure, and many questions warrant further~investigation.

A first direction is a systematic exploration of alternative base learners in \texttt{EGlearn}. We have focused on the two most popular and classical approaches, but many more possible choices exist; see \citet[Chapter 3]{DM17} and the references cited therein for a partial overview of the recent literature with a focus on graphical modeling. 

Similarly, in this paper we have used the empirical extremal variogram, but different estimators of the variogram matrix $\Gamma$ could be considered. For instance, one could consider method of moments or M-estimators \citep{EKS2008,EKS2012,LEV21}.

In view of the characterization in~\eqref{eq:HRtheta}, one may consider an $L^1$-penalized version of the pseudo log-likelihood in \cref{nllk} and estimate the H\"usler--Reiss precision matrix $\Theta$ by the extremal graphical lasso problem
\begin{equation} \label{eq:penalTheta}
	\mathop{\arg\min}_{\Theta \in \cS_{\bone}^d} \Big\{ -{\log\det}^*\Theta - \frac12 \tr(\hat \Gamma \Theta) + \rho \mathop{\sum\sum}_{i \neq j} |\Theta_{ij}| \Big\};
\end{equation}
recall that ${\det}^*$ denotes the pseudo-determinant of a matrix (the product of its non-zero eigenvalues), and $\cS_{\bone}^d$ is the cone of symmetric, positive semi-definite matrices with rank $d-1$ and row sums equal to zero. An advantage of this approach over all existing proposals is that it automatically provides an estimator which is a valid H\"usler--Reiss precision matrix. The main difficulties are the semi-definiteness of $\Theta$ and the zero row sum constraint. This renders all available theory for estimation of sparse inverses of precision matrices inapplicable. Indeed, preliminary simulations and theoretical investigations indicate that the estimator in~\eqref{eq:penalTheta} in general does not lead to consistent recovery of the sparsity structure of $\Theta$. Replacing the vanilla $L^1$ penalties by adaptively weighted or non-convex versions might provide a way out, but a full investigation of this approach is beyond the scope of the present paper.

Another highly promising direction for future research is uncertainty quantification for the estimated graph. This could include testing the null hypothesis that an edge between two given nodes is not present or testing for specific structures in a graph. Uncertainty quantification would also be an important step on the way to estimating directed graphs \citep{engelke2025extremesstructuralcausalmodels}, or at lest some aspects such as the moralized graph.

Finally, all of our results on graph recovery rely on the H\"usler--Reiss domain of attraction. We are not aware of any principled tests for the hypothesis that this model indeed provides a good approximation for the tail. Developing such tests is an important future direction.

\section*{Acknowledgments}
The authors would like to thank Nicola Gnecco and Yanbo Tang for helpful comments, and Yan Zhang for comments on the R implementation. We are also grateful to the Associate Editor and three anonymous referees whose constructive comments resulted in a better manuscript. SE was supported by an Eccellenza grant of the Swiss National Science Foundation, ML was supported by a postdoctoral fellowship from NSERC of Canada, and SV was supported by a discovery grant from NSERC of~Canada.

\newpage

\setcounter{page}{1}

\setcounter{section}{0}
\renewcommand*{\theHsection}{chX.\the\value{section}}
\setcounter{lemm}{0}
\renewcommand*{\theHlemm}{chX.\the\value{lemm}}
\setcounter{rem}{0}
\renewcommand*{\theHrem}{chX.\the\value{rem}}
\setcounter{coro}{0}
\renewcommand*{\theHcoro}{chX.\the\value{coro}}
\setcounter{prop}{0}
\renewcommand*{\theHprop}{chX.\the\value{prop}}
\setcounter{assum}{0}
\renewcommand*{\theHassum}{chX.\the\value{assum}}
\setcounter{figure}{0}
\renewcommand*{\theHfigure}{chX.\the\value{figure}}
\setcounter{table}{0}
\renewcommand*{\theHtable}{chX.\the\value{table}}

\renewcommand{\thesection}{S\arabic{section}}
\renewcommand{\thelemm}{S\arabic{lemm}}
\renewcommand{\therem}{S\arabic{rem}}
\renewcommand{\thecoro}{S\arabic{coro}}
\renewcommand{\theprop}{S\arabic{prop}}
\renewcommand{\theassum}{S\arabic{assum}}
\renewcommand{\thefigure}{S\arabic{figure}}
\renewcommand{\thetable}{S\arabic{table}}

\begin{center}
SUPPLEMENTARY MATERIAL
\end{center}

This Supplementary Material is divided into six sections. \Cref{sec:connectionscalssicalCI} contains some additional discussion on extremal graphical models, and \cref{sec:simsupp} presents numerical results that complement \cref{sec:sims,sec:appl}. The rest of the Supplementary Material is dedicated to the proofs of all the results in the paper. We start with the theory related to graph recovery and precision matrix estimation (\cref{sec:proofeglearn,sec:proofgraph}), followed by the proof of \cref{thm:concentration} and related auxiliary results (\cref{sec:proof,sec:auxiliary}). All references to sections, results, equations, etc. not starting with the letter ``S'' are pointing to the main paper.

\section{Additional comments on extremal graphical models}
\label{sec:connectionscalssicalCI}

\subsection{Utility and interpretation of extremal graphical models}

Although the notion of extremal conditional independence in \eqref{ext_CI} is different from classical conditional independence, there are close connections. Specifically, if a multivariate Pareto distribution $\g Y$ satisfying $\g Y_A \perp_e \g Y_B \mid \g Y_C$ is restricted to any rectangle $R \subset \mathcal L = [0,\infty)^d \setminus [0,1)^d$, resulting in the restricted random vector $\g Y^{(R)}$ with support $R$ and law defined by $P(\g Y^{(R)} \in A) := P(\g Y \in A|\g Y \in R)$, then it satisfies the conditional independence in the usual sense $\g Y_A^{(R)} \indep \g Y_B^{(R)} \mid \g Y_C^{(R)}$ \citep{eng_iva_kir}. Restricting to the rectangle $[1,\infty)^d$ leads to a particularly nice interpretation since this corresponds to considering vectors where all entries are large. In other words, if one were to threshold their data by only considering the observations where all variables are simultaneously large, then these rare observations could be assumed to follow a graphical model in the usual sense. Our results show that the structure of this graphical model can be learned by using the considerably more common observations where at least one variable is large.
		
There is also a second way to interpret extremal conditional independence is in terms of classical conditional independence: a line of research \citep{segers2020, AS21} shows that extremal graphical models naturally arise as limits in distribution of certain usual graphical models.

In addition to the above connection to classical conditional independence, there are several other reasons why \eqref{ext_CI} is a natural and useful approach to graphical modeling. 

\begin{enumerate}
\item One classical interpretation of graphical models is in terms of optimal prediction of one variable in the graph given the remaining variables. Similarly as in usual graphical models, in our setting it is sufficient to use the neighborhood of a variable (in the graph we define through \eqref{ext_CI}) to obtain optimal prediction; see the discussion in \citet[][page 12]{engelke2024review}. This provides a natural interpretation of extremal graphs in terms of prediction properties.
\item Multivariate Pareto distributions, and specifically H\"usler--Reiss distributions with conditional independencies as in \eqref{ext_CI} can be seen as a class of sparse models for tail dependence. Varying the degree of sparsity provides a collection of models with different complexities. This allows to balance sample size and model complexity in a principled and systematic way. In high-dimensional settings, enforcing sparsity through penalization also provides an important regularization technique. This is especially crucial for modeling extremes where there are fewer data in the tail than in the bulk of the distribution. 
\item In some fields such as sparse spatial modeling based on SPDEs \citep[e.g.,][]{lin2011}, the  main point of graphical models is computational efficiency through density factorizations. While this is not the primary focus of this paper, it should be kept in mind that extremal graphical models offer the same computational gains, which can be used for cases where computational time is the bottleneck. 
\end{enumerate}

\subsection{Example of non-faithful H\"usler--Reiss distribution}

 The conditional independence graph $G$ defined in~\eqref{CI_graph} encodes the extremal conditional independencies that correspond to zeros in the precision matrix $\Theta$. There might however be additional conditional independence statements, in which case we say that the H\"usler--Reiss distribution is not faithful to the graph $G$; note that this is not a problem for the structure considered in this paper.

To provide a concrete example, consider the H\"usler--Reiss distribution with variogram and corresponding precision matrices given by	
\[
\Gamma =
\begin{bmatrix}
0\phantom{0} & 3\phantom{0} & 4\phantom{0} & 7\phantom{0} & 6 \\
3\phantom{0} & 0\phantom{0} & 3\phantom{0} & 8\phantom{0} & 5 \\
4\phantom{0} & 3\phantom{0} & 0\phantom{0} & 7\phantom{0} & 8 \\
7\phantom{0} & 8\phantom{0} & 7\phantom{0} & 0\phantom{0} & 11 \\
6\phantom{0} & 5\phantom{0} & 8\phantom{0} & 11\phantom{0} & 0
\end{bmatrix} \qquad 
\Theta =
\begin{bmatrix}
1.52\phantom{0} & -0.08\phantom{0} & -0.28\phantom{0} & -0.68\phantom{0} & -0.48 \\
-0.08\phantom{0} & 1.32\phantom{0} & 0.12\phantom{0} & -1.28\phantom{0} & -0.08 \\
-0.28\phantom{0} & 0.12\phantom{0} & 1.92\phantom{0} & -0.48\phantom{0} & -1.28 \\
-0.68\phantom{0} & -1.28\phantom{0} & -0.48\phantom{0} & 4.12\phantom{0} & -1.68 \\
-0.48\phantom{0} & -0.08\phantom{0} & -1.28\phantom{0} & -1.68\phantom{0} & 3.52
\end{bmatrix}.
\]
Since there are no zeros in the precision matrix, the only possible conditional independence graph is the fully connected graph. However, there are other conditional independencies in this model. Indeed, the marginal distribution for the first four nodes is H\"usler--Reiss with variogram and precision
\[
\Gamma_{\setminus 5,\setminus 5} =
\begin{bmatrix}
0\phantom{0} & 3\phantom{0} & 4\phantom{0} & 7 \\
3\phantom{0} & 0\phantom{0} & 3\phantom{0} & 8 \\
4\phantom{0} & 3\phantom{0} & 0\phantom{0} & 7 \\
7\phantom{0} & 8\phantom{0} & 7\phantom{0} & 0
\end{bmatrix}
\qquad
\Theta^{\setminus 5,\setminus 5} =
\begin{bmatrix}
0.42\phantom{0} & -0.25\phantom{0} & -0.08\phantom{0} & -0.08 \\
-0.25\phantom{0} & 0.50\phantom{0} & -0.25\phantom{0} & 0 \\
-0.08\phantom{0} & -0.25\phantom{0} & 0.42\phantom{0} & -0.08 \\
-0.08\phantom{0} & 0\phantom{0} & -0.08\phantom{0} & 0.17
\end{bmatrix}.
\]
Note that the precision matrix $\Theta^{\setminus 5,\setminus 5}$ of the model with variogram $\Gamma_{\setminus 5,\setminus 5}$ is not simply $\Theta_{\setminus 5,\setminus 5}$ (thus the superscript notation); see also \cite{engelke2024extremal}. We can see that we have the extremal conditional independence $Y_2 \perp_e Y_4 \mid \mathbf{Y}_{\{1,3\}}$. Thus, the H\"usler--Reiss distribution with variogram $\Gamma$ is not faithful to the fully connected graph (or any other graph). This example is inspired by Example 1 in \cite{NIPS2014_2bd7f907} in the Gaussian case.

\section{Additional numerical results}
\label{sec:simsupp}

\subsection{Additional figures for \cref{sec:sims}}

\Cref{fig:BA} illustrates example graphs simulated from the Barabasi--Albert and block graph models used throughout the simulation study. \Cref{fig:lasso_path} presents regularization paths obtained by applying \eglearn\ with different values of the penalty parameter. For more details on how these figures are produced, we refer to \cref{sec:sims}.

\begin{figure}[H]
  \includegraphics[width=.32\textwidth, trim= 5em 5em 5em 5em, clip]{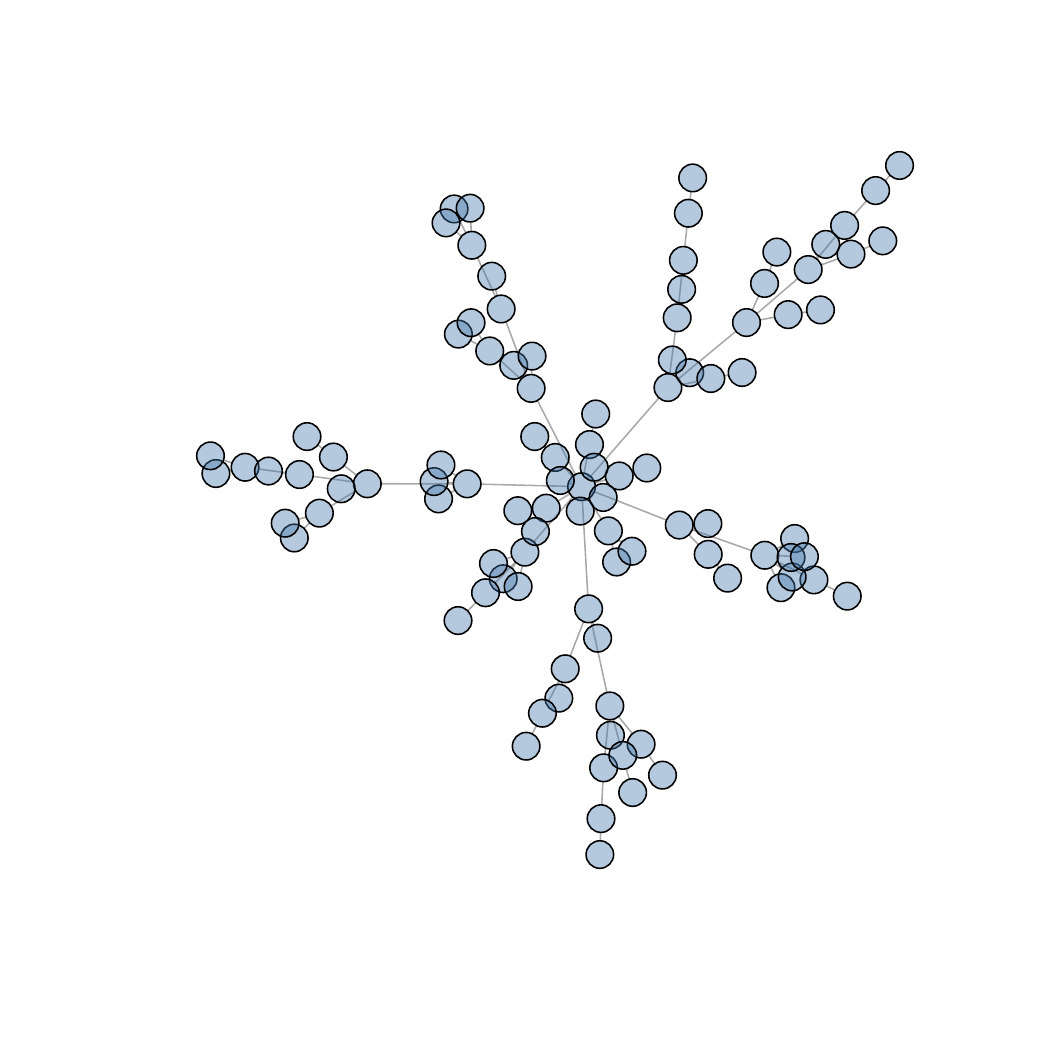}
  \includegraphics[width=.32\textwidth, trim= 5em 5em 5em 5em, clip]{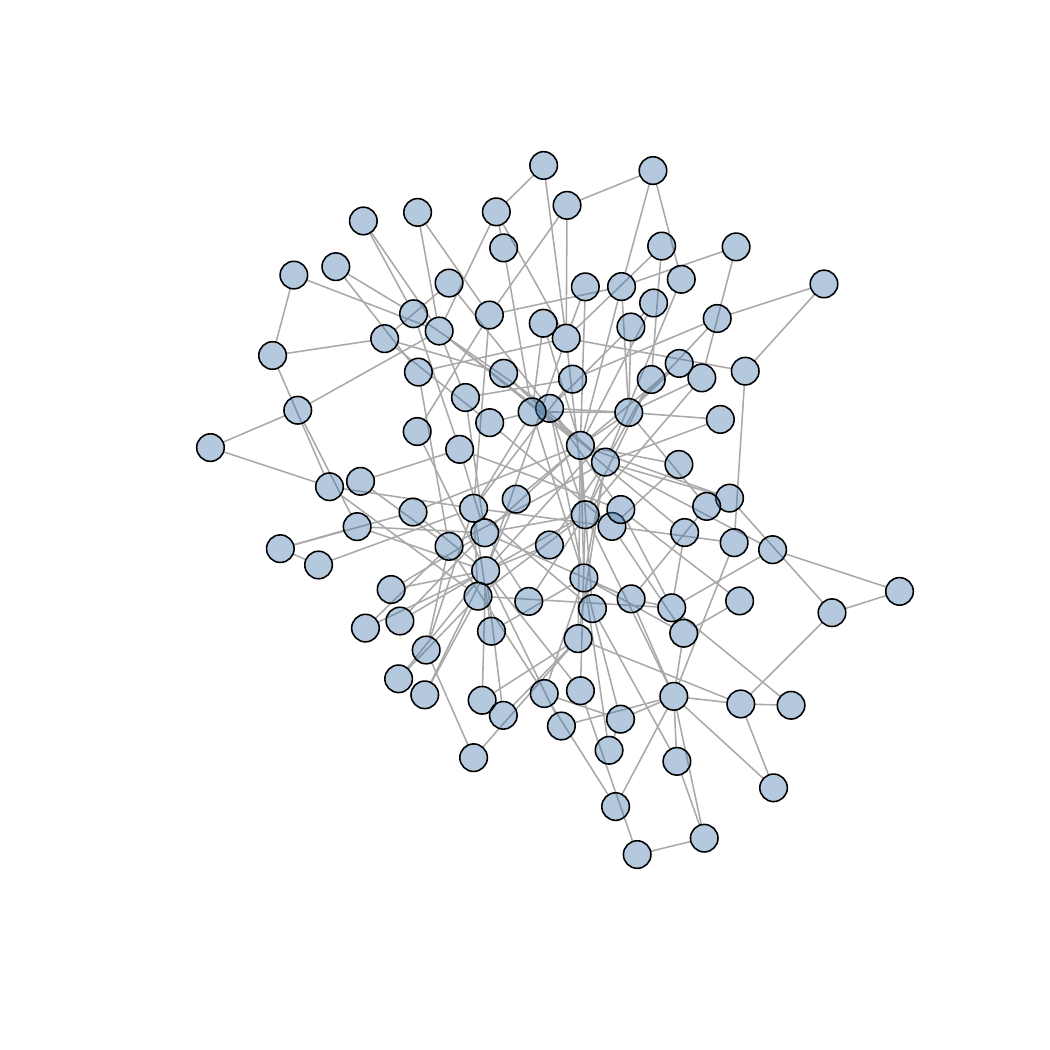}
  \includegraphics[width=.32\textwidth, trim= 5em 5em 5em 5em, clip]{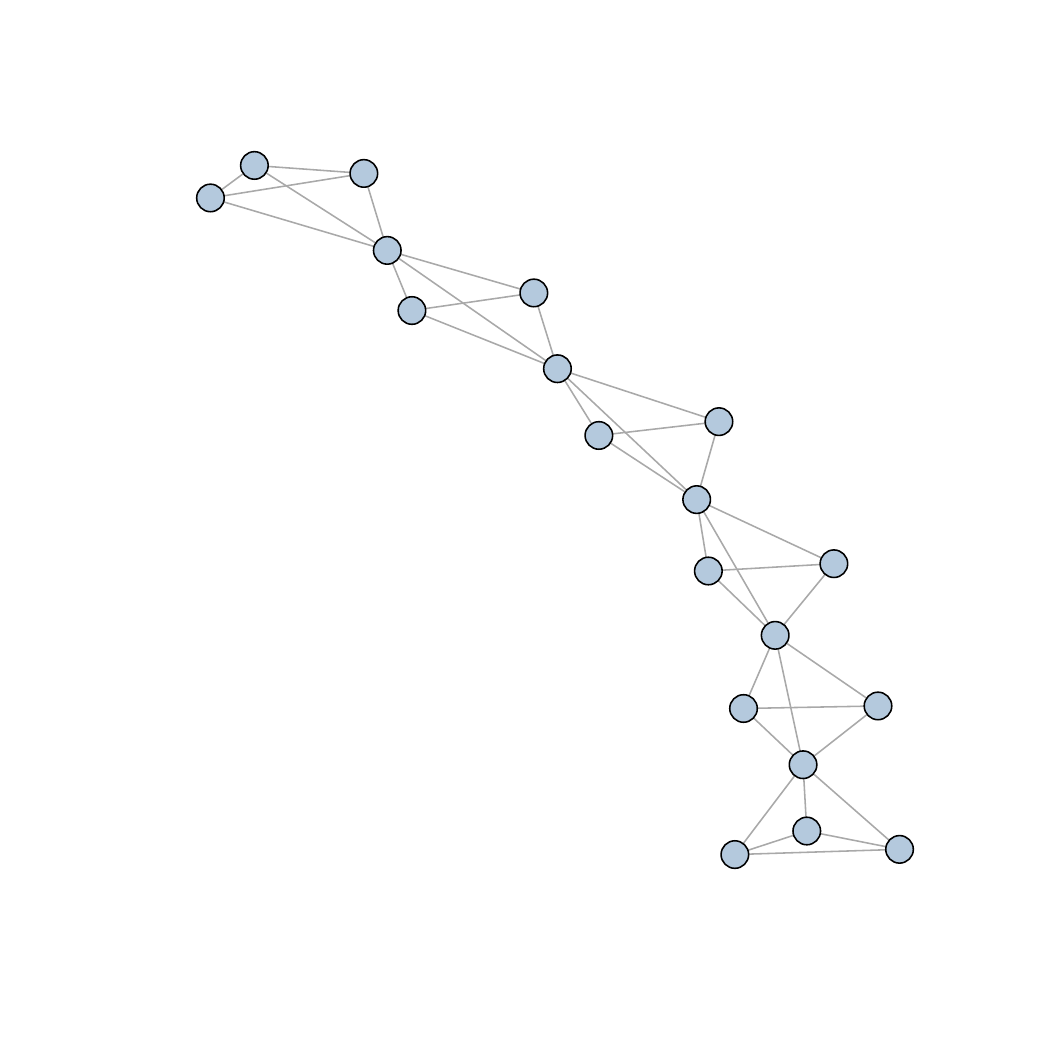}
  \caption{Realizations of the Barabasi--Albert models $\BA(100, 1)$ (left) and $\BA(100, 2)$ (center) in dimensions $d=100$, and of the block graph model $\BM(6, 4, \alpha)$ in dimension $d=19$.}
  \label{fig:BA}
\end{figure}

\begin{figure}[H]
  \includegraphics[scale = 0.85]{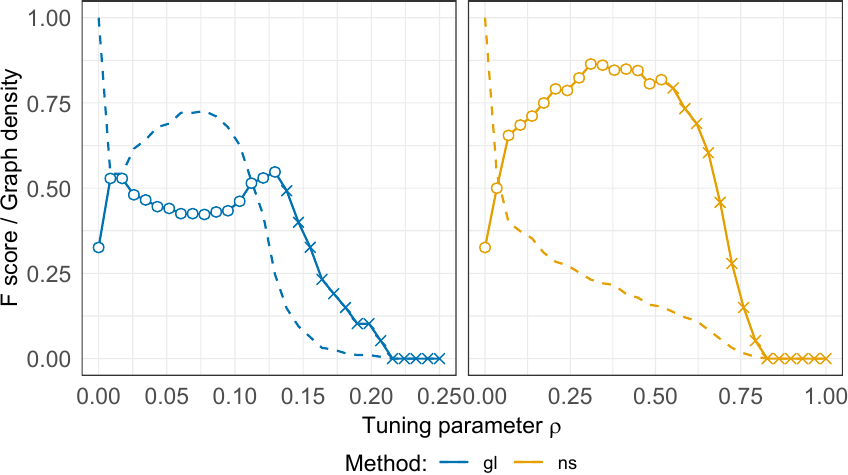}
  \caption{Paths of \eglearn{} estimated graphs with graphical lasso (left) and neighborhood selection (right) as a function of the common tuning parameter $\rho$, fitted to data from the $\BA(20, 2)$ model with $k=100$. Circles indicate that the estimated graph $\hat G$ is connected, and crosses correspond to disconnected graphs. Dashed lines show the density of the graph, that is, the proportion of existing edges in $\hat G$ out of all $d(d-1)/2$ possible edges.}
  \label{fig:lasso_path}
\end{figure}

\subsection{Simulation results for the Barabasi--Albert model in dimension $d=50$}
\label{sec:d50}

The results of our experiments on the Barabasi--Albert model of dimension $d=50$ are presented in \cref{fig:d50}. The conclusions are similar to the 100-dimensional case: with neighborhood selection as base learner, \eglearn{} seems to have a consistent behavior as $k$ increases, as opposed to the graphical lasso case. With either oracle or MBIC tuning, this method moreover outperforms the others except for the minimum spanning tree when the underlying graph is indeed tree structured.

\begin{figure}[H]
  \includegraphics[scale = 0.85]{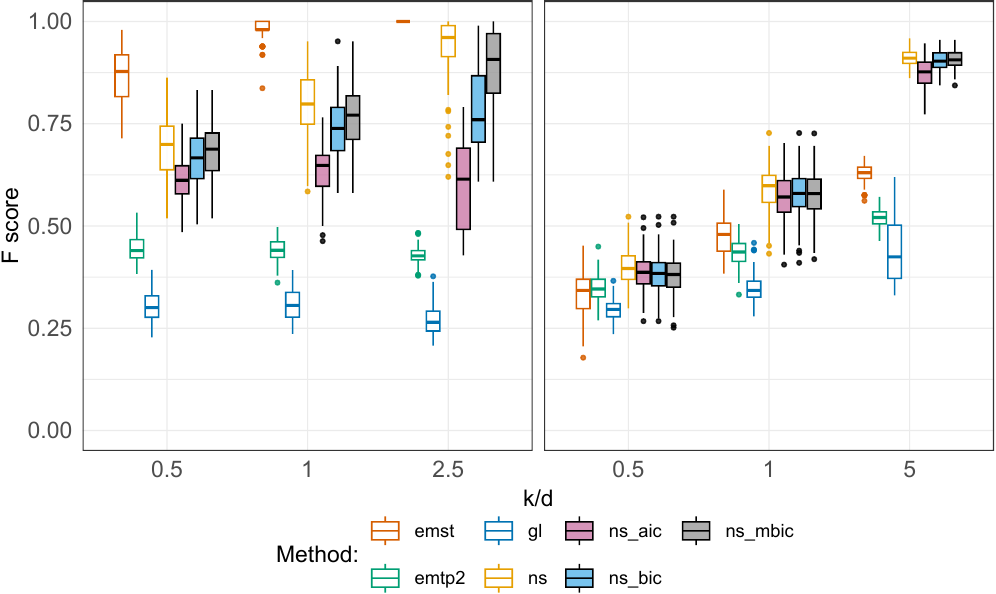}
  \caption{Boxplots of 100 repetitions of the $F$-scores of different methods fitted to data from the model $\BA(d, q)$ of degree $q=1$ (left) and $q=2$ (right) and in dimension $d = 50$.}
  \label{fig:d50}
\end{figure}

\subsection{Simulation setup and results for the block graph model}
\label{sec:block_graph}

As opposed to the Barabasi--Albert model, as a second model for $G$ and $\Gamma$ we consider a setup where $\Theta_{ij} > 0$ for some $i,j \in V$. Note that in this case there is no easy construction method similar to the Laplacian above.
Instead, we consider a graph with $n_C$ fully connected cliques $C_1,\dots ,C_{n_C}$, each consisting of $n_N$ nodes.
We assume that the only intersections of these cliques are between $C_{j}$ and $C_{j+1}$, $j=1,\dots, n_C - 1$, and that each intersection consists of a single node. This results in a block graph $G$ and a block-like structure of the precision matrix $\Theta$ \citep{H21}; see the right-hand side of Figure~\ref{fig:BA} for a block graph with $n_C = 10$ and $n_N = 4$.
On this extremal graph structure, it suffices to specify $\Gamma_{C_j,C_j}$ on each clique $C_j$, and the remaining entries are implied by the conditional independence structure \citep[][Proposition 4]{EH2020}.
Following \cite{H21}, we can construct a valid $\Gamma_{C_j,C_j}$ matrix by taking any $(n_N\times n_N)$-dimensional covariance matrix $S$ and projecting it by $PSP$, where $P = I_d - \g 1 \g 1^\top / d$. For each clique we generate independently a correlation matrix $S$ following the method in \cite{joe2006}, whose off-diagonal entries have marginal $\text{Beta}(\alpha -1 + n_N/2, \alpha -1 + n_N/2)$ distributions rescaled to $(-1,1)$, where $\alpha>0$ is a parameter. We denote this block model for $\Gamma$ by $\BM(n_C, n_N, \alpha)$. It has dimension $d = n_N + (n_C-1)(n_N-1)$ and is parametrized by the number of cliques $n_C$, the number of nodes $n_N$ per clique, and the parameter $\alpha$ governing the dependence inside the cliques. Figure~\ref{fig:alpha} shows the proportion of positive off-diagonal entries of the H\"usler--Reiss precision matrix $\Theta$ corresponding to the block graph model $\BM(6, 4, \alpha)$ for different $\alpha$ values. It can be seen that for increasing $\alpha$, less positive values appear and the model becomes closer to $\text{EMTP}_2$.

We run simulations with H\"usler--Reiss distributions generated according to the block model $\BM(6, 4, \alpha)$, which results in $d =19$ nodes in the graph. For the dependence parameter we choose a sequence of values $\alpha \in \{0.1,1,2,10,20\}$. As before we use $k = \lfloor n^{0.7}\rfloor$ exceedances and we simulate two different sample sizes $n$ such that $k/d \in \{2, 10\}$. The results for 100 repetitions are shown in Figure~\ref{fig:alpha}.
The top right panel shows boxplots of the $F$-scores for the oracle \eglearn{} with graphical lasso and neighborhood selection as a function of $\alpha$.
We observe that, again, \eglearn{} with graphical lasso base learner does not seem to be consistent since even with the larger sample size the $F$-scores do not improve much. On the contrary, \eglearn{} with neighborhood selection performs well especially for larger sample sizes, suggesting consistency of the method.
We further observe that in general, smaller values of $\alpha$ correspond to more difficult estimation problems. Note that this corresponds to the case of higher proportions of positive entries $\Theta_{ij}$ in the H\"usler--Reiss precision matrix (top left panel of Figure~\ref{fig:alpha}).

To understand this behavior and the related phenomenon that the graphical lasso as base learner does not seem to work well, we take a closer look at the assumptions for consistent structure recovery by \eglearn\ in \cref{th:NS,th:GL}. For a given parameter matrix $\Gamma$, a crucial requirement for consistency with both base learners is the positivity of the incoherence parameters $\eta^\ns$ and $\eta^\gl$, respectively. The bottom panels of Figure~\ref{fig:alpha} show boxplots of these parameters for the generated block models. All incoherence parameters $\eta^\ns$ for neighborhood selection are positive and thus, Theorem~\ref{th:NS} guarantees consistent graph recovery. We also note that as $\alpha$ increases, so does $\eta^\ns$, and the graph recovery performance improve. This is in line with our theory; the expression of $C^\ns$ in \cref{th:NS} suggests that a higher $\eta^\ns$ increases the probability of graph recovery. On the other hand, all incoherence parameters $\eta^\gl$ are negative and Theorem~\ref{th:GL} is not applicable. More generally, for all the simulation settings we have considered, the neighborhood selection incoherence parameter $\eta^\ns$ is much more likely to be positive than its graphical lasso equivalent $\eta^\gl$. It thus appears that the assumption of \cref{th:NS} is significantly weaker than that of \cref{th:GL}. This is also consistent with results in the literature of Gaussian structure learning \citep[Sections~3.1.1, 3.1.2]{RWRY11}.

\begin{figure}[H]
	\centering
	\includegraphics[height=.9\textwidth]{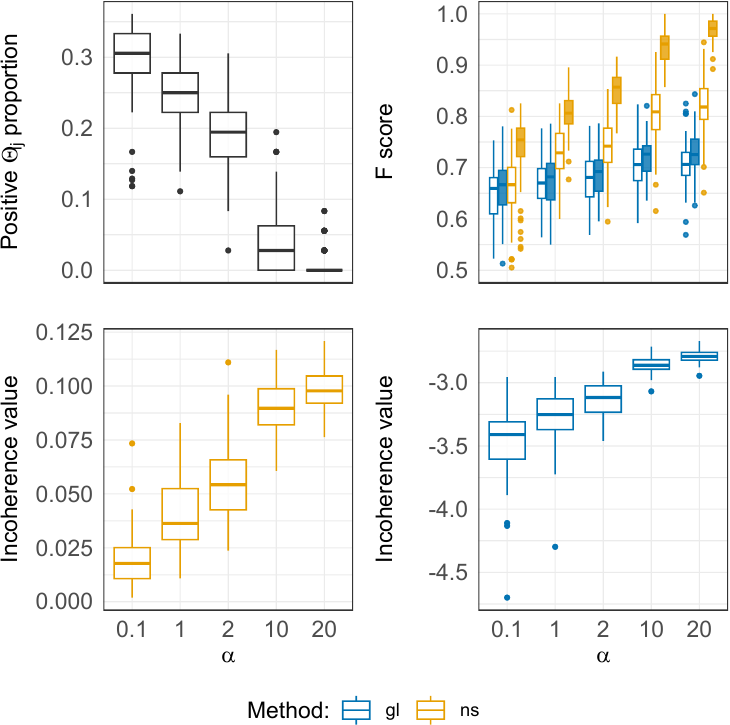}
	\caption{Top left: proportion of positive off-diagonal entries of the H\"usler--Reiss precision matrix $\Theta$ of the model $\BM(6, 4, \alpha)$; top right: performance of \eglearn{} with neighborhood selection and graphical lasso with $k = 2d$ and $k=10d$ for empty and filled boxes, respectively; bottom: incoherence parameters $\eta^\ns$ for neighborhood selection (left) and $\eta^\gl$ for graphical lasso (right).}
	\label{fig:alpha}
\end{figure}

\subsection{A second data application}
\label{sec:currency}

\cite{EV22} consider $n=3790$ daily observations of spot foreign exchange rates between 2005 and 2020 of $d=26$ currencies expressed in terms of the British Pound sterling; the currencies along with their abbreviations are listed in \cref{tab:countries}. As in the latter paper, we pre-process this data to obtain the absolute values of de-garched daily log-returns, and thus consider the extremal dependence structure of the log-returns in both directions.

\begin{table}[H]
	\centering
	\begin{tabular}{l|l|l|l}
		Code & Foreign Exchange Rate (into GBP) & Code & Foreign Exchange Rate (into GBP)\\
		\hline 
		AUS & Australian Dollar & NOR & Norwegian Krone \\
		CAN & Canadian Dollar  & POL &Polish Zloty \\
		CHN & Chinese Yuan & RUS & Russian Ruble\\
		CZE & Czech Koruna & SAU &Saudi Riyal\\
		DNK & Danish Krone & SGP & Singapore Dollar\\
		EUR & Euro & ZAF & South African Rand\\
		HKG & Hong Kong Dollar &  KOR &South Korean Won\\
		HUN & Hungarian & SWE & Swedish Krona\\
		IND & Indian Rupee & CHE & Swiss Franc\\
		ISR & Israeli Shekel & TWN & Taiwan Dollar\\
		JPN &  Japanese Yen & THA & Thai Baht\\
		MYS & Malaysian ringgit & TUR & Turkish Lira\\
		NZL & New Zealand Dollar & USA &  US Dollar
	\end{tabular}
	\caption{\label{tab:countries}Three-letter country codes.}
\end{table}

Figure~\ref{fig:currency} shows the minimum spanning tree obtained in \cite{EV22}. As in the previous application, we compute the empirical variogram $\hat \Gamma$ using $k = \lfloor n^{0.7} \rfloor = 319$. In order to allow for more general graphs than trees, we apply \eglearn\ with neighborhood selection as the base learner, a suitable grid of penalty parameters $\rho$ and input variogram estimate $\hat \Gamma$. The right panel of Figure~\ref{fig:currency} shows the graph selected by the MBIC. The estimated graph is much denser than a tree, suggesting that the dependencies in the extremes are considerably stronger than could have been explained by the minimum spanning tree approach.

\begin{figure}[H]
\centering
\includegraphics[width=0.3\textwidth, trim={5em 5em 5em 5em}]{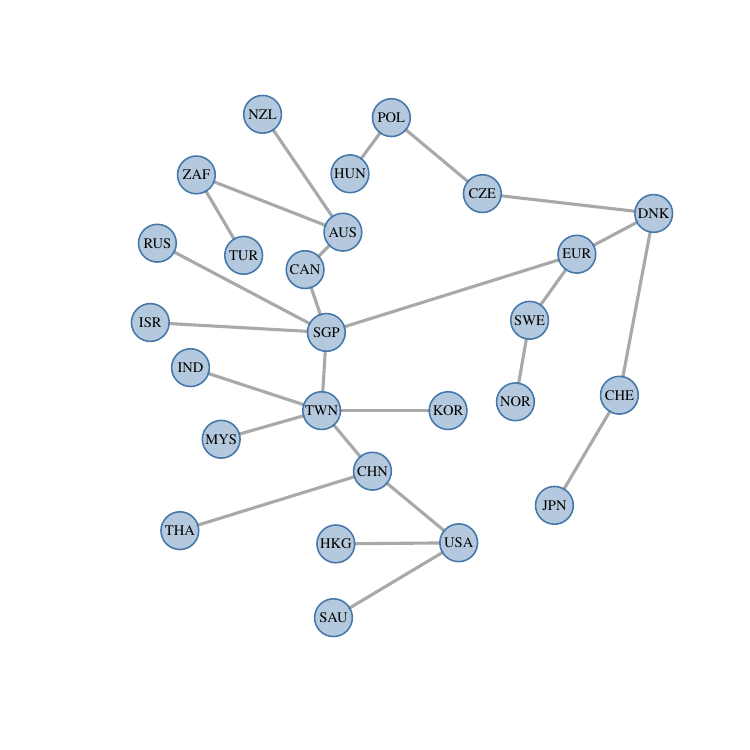}
\hspace{1em}\includegraphics[width=0.3\textwidth, trim={5em 5em 5em 5em}]{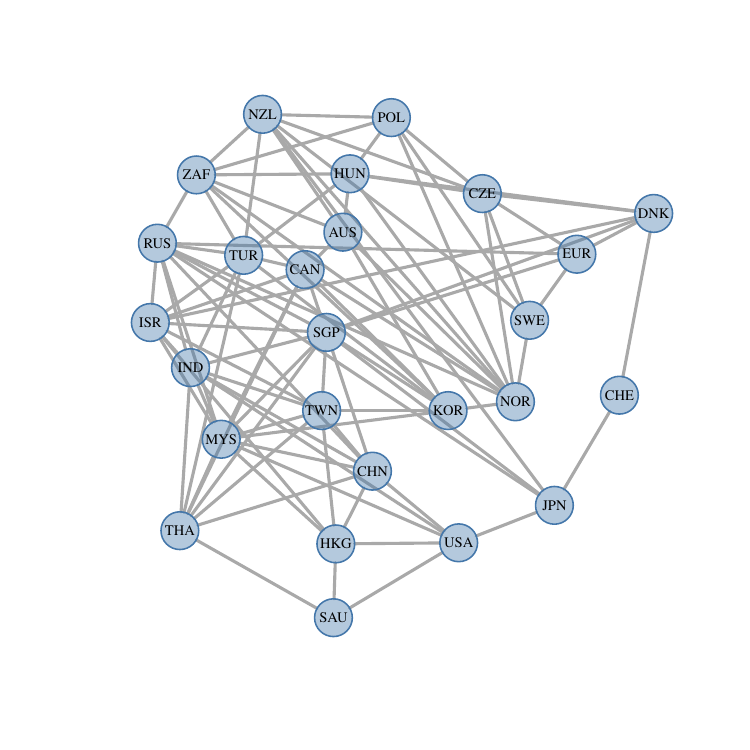}
\caption{The minimum spanning tree (left) and the selected graph by MBIC estimated by \eglearn\ (right) for the exchange rates data set; the list of currencies with their respective abbreviations are shown in \cref{tab:countries}.}
\label{fig:currency}
\end{figure}

As discussed in \cref{sec:appl}, an alternative way to select the graph structure is to look for a desired sparsity level in the estimated path. In \cref{sec:graphs_ic}, we report three such estimated graphs that are considerably sparser than the MBIC graph. An advantage of sparse graphs is that they allow for better interpretation in the context of the application. For instance, in our case, a notable feature of all estimated models is that there are many connections within the cluster of South Asian currencies, suggesting a strong degree of dependence between the currencies of these countries.

\subsection{Additional details on the real data analyses}
\label{sec:graphs_ic}

\Cref{fig:danube_ic} contains graphs obtained by running \eglearn\ on the Danube data set (without stations 23--27) with neighborhood selection as the base learner. The empirical variogram is calculated with $k = 69$, and the three graphs correspond to penalty parameters $\rho$ equal to $0.075$, $0.2$ and $0.405$, respectively. \Cref{fig:currency_ic} contains the estimated graphs calculated similarly for the currency exchange data set with $k = 319$ and $\rho$ equal to $0.35$, $0.45$ and $0.53$, respectively. In both cases, the third graph corresponds to the most sparse connected graph in the estimated path.

\begin{figure}[H]
\centering
\includegraphics[width=0.3\textwidth, trim={5em 5em 5em 5em}]{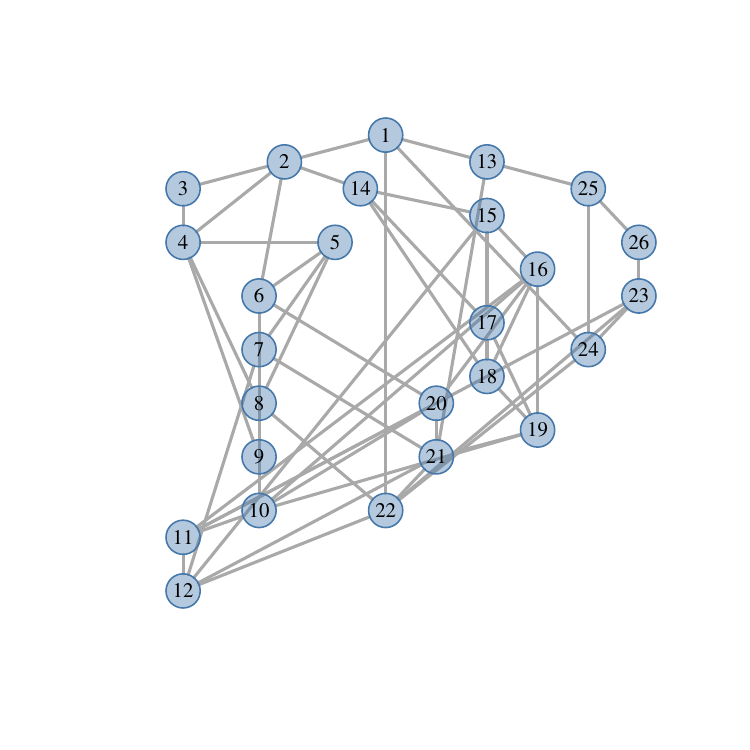}
\includegraphics[width=0.3\textwidth, trim={5em 5em 5em 5em}]{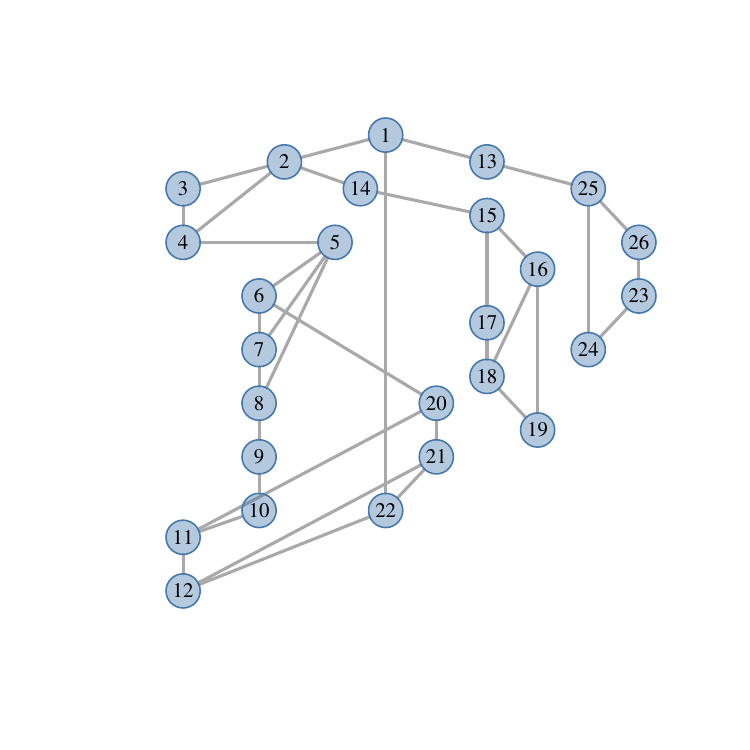}
\includegraphics[width=0.3\textwidth, trim={5em 5em 5em 5em}]{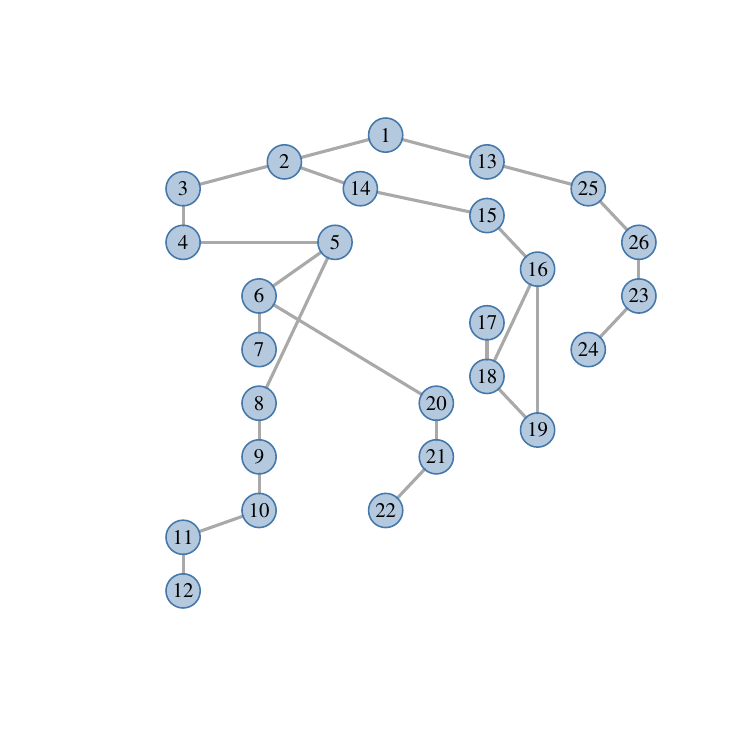}
\caption{Estimated graphs for the Danube data obtained from \eglearn\ with neighborhood selection and penalty parameter $\rho$ equal to $0.06$, $0.09$ and $0.12$, respectively.}
\label{fig:danube_ic}
\end{figure}

\begin{figure}[H]
\centering
\includegraphics[scale=0.45, trim={5em 5em 5em 5em}]{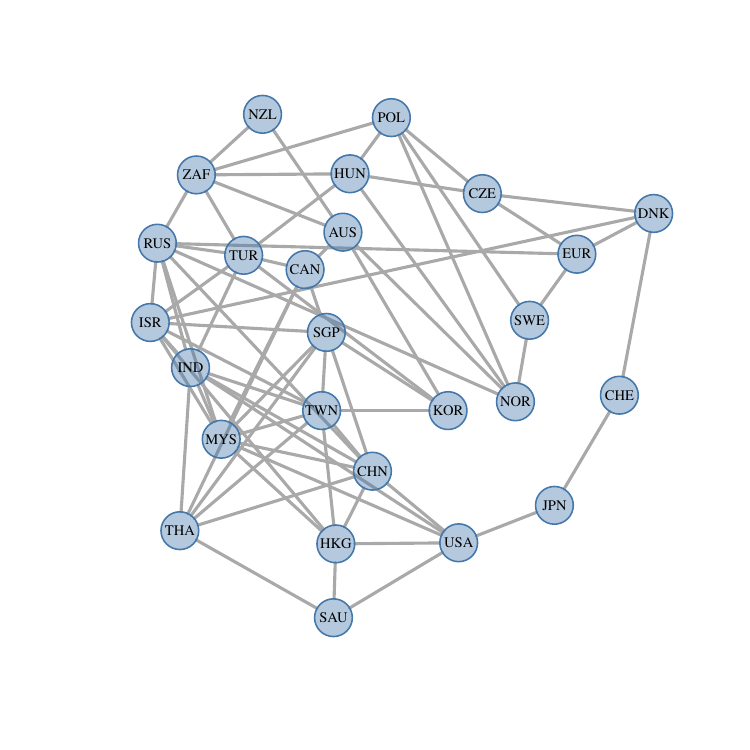}
\includegraphics[scale=0.45, trim={5em 5em 5em 5em}]{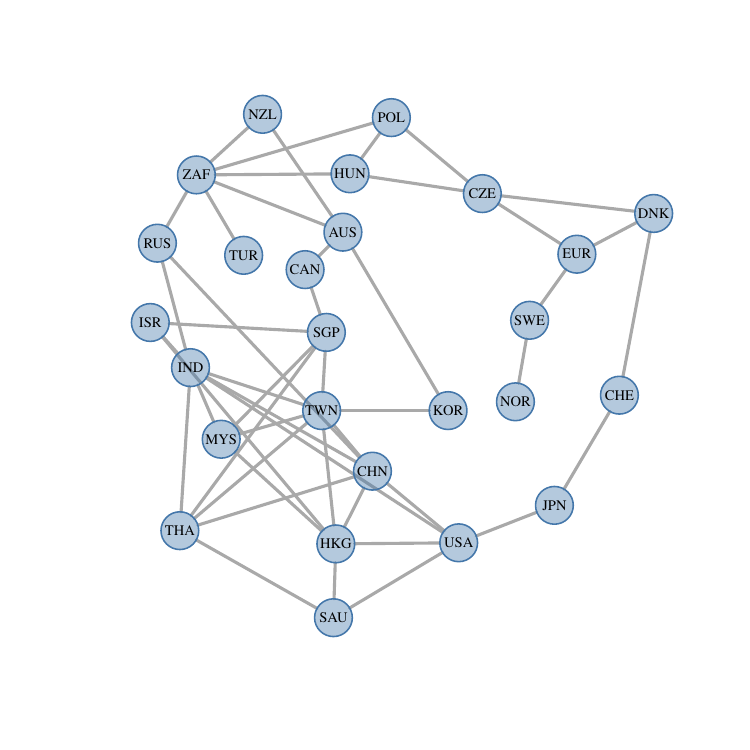}
\includegraphics[scale=0.45, trim={5em 5em 5em 5em}]{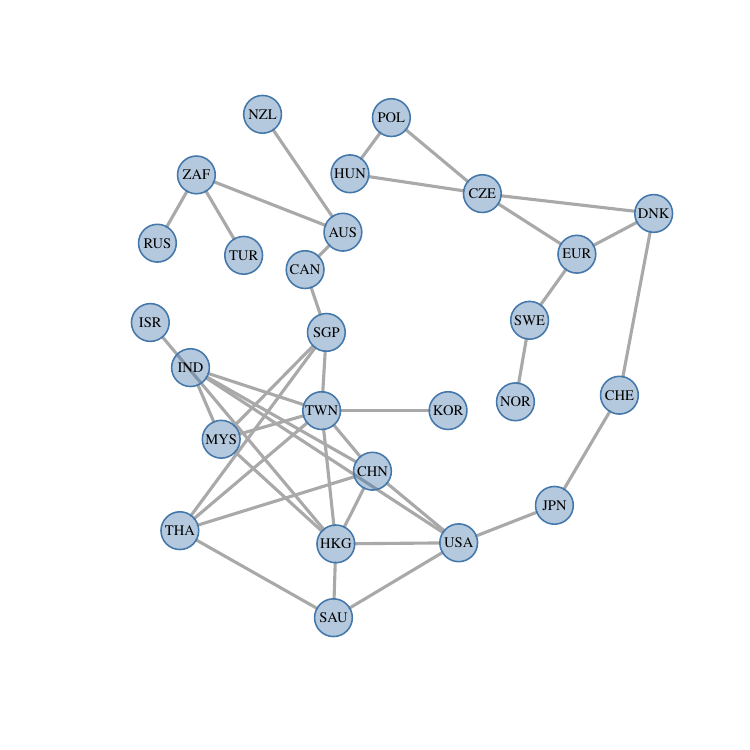}
\caption{Estimated graphs for the currency exchange data obtained from \eglearn\ with neighborhood selection and penalty parameter $\rho$ equal to $0.35$, $0.45$ and $0.53$, respectively.}
\label{fig:currency_ic}
\end{figure}

In order to evaluate the modeling performance of the different estimated graphs, we look at the test H\"usler--Reiss log-likelihood $L$, as defined in \cref{sec:sims}, of the fitted model evaluated on test data. To this end, we split both of our data sets into two parts and use the first half for fitting and the second half for model evaluation.

Figure~\ref{fig:likelihoods} shows the test log-likelihoods for the \eglearn{} estimated graphs along a path of penalty parameters $\rho$, as well as the log-likelihood of the AIC (pink horizontal line) and BIC (gray horizontal line) selected models. In both examples, the MBIC and BIC model selection criteria agree with each other. On the Danube river data (left panel), the fully connected graph ($\rho = 0$) has a much worse performance than sparser graphs, confirming that the enforced sparsity leads to a better model fit. The test likelihood of the extremal minimum spanning tree and the tree induced by flow connections are represented by the dotted and dashed horizontal lines, respectively. While these tree models show a reasonable performance in this case, even more can be gained by allowing denser models such as our AIC and BIC/MBIC graphs. In our analysis, the currency exchange data (right panel) was found to exhibit a much denser extremal dependence structure. As such, too sparse models ($\rho \geq 0.3$) perform poorly as well as the fully connected model ($\rho=0$). The AIC selected model, albeit too conservative when it comes to support recovery in the simulations study, attains a considerably better log-likelihood value than the BIC/MBIC model, which however surpasses the fully connected model. Here, the extremal minimum spanning tree model has a test log-likelihood value that is below the plotting range by thousands of points, suggesting that such a model is not appropriate for this data.

\begin{figure}[H]
	\includegraphics[height=.49 \textwidth]{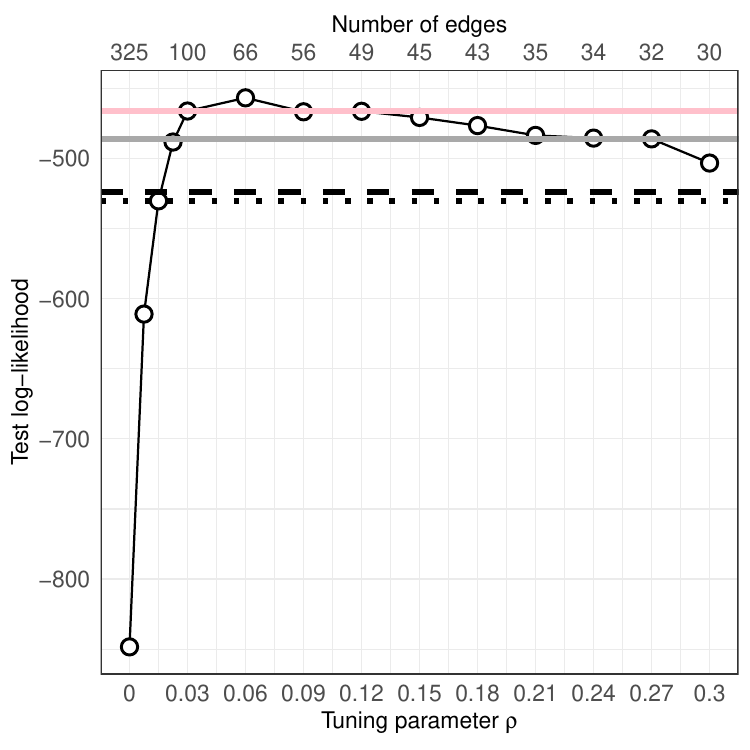}
	\includegraphics[height=.49 \textwidth]{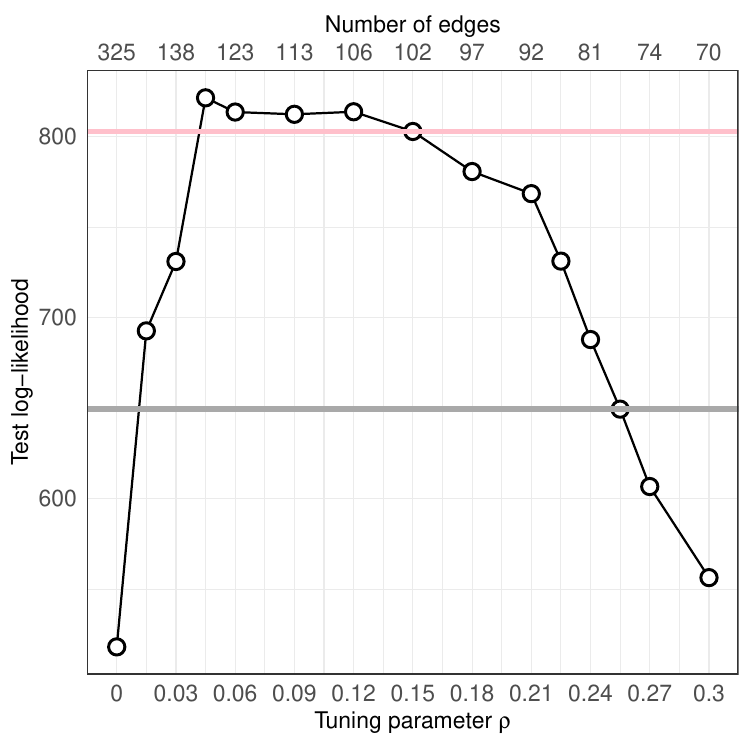}
	\caption{Test log-likelihoods of H\"usler--Reiss models on the estimated graphs by \eglearn{} with neighborhood selection as the base learner, as a function of the penalty parameter $\rho$, on the Danube river (left panel) and currency exchange (right panel) data sets. Pink and gray horizontal lines represent the test log-likelihood of the graphs selected by the AIC and BIC/MBIC criteria. Left panel: dotted and dashed lines represent the test log-likelihood of the extremal minimum spanning tree and the tree induced by flow connections. Right panel: the test log-likelihood of the extremal minimum spanning tree is below the plotting range.}
	\label{fig:likelihoods}
\end{figure}

As a means of comparison, we also apply naive graphical lasso and neighborhood selection to both of our data sets (after suitable marginal normalization) with a range of different penalty parameters. The top row of \cref{fig:vanilla} shows three estimated graphs on the Danube river data by the \eglearn{} (left), graphical lasso (middle), and neighborhood selection (right) methods, each of which contains the same number (27) of edges. Overall, the graphs obtained by the ``vanilla'' methods uncover quite different dependencies. In particular, the graphical lasso method seems to focus on the densely connected group of variables 2 to 7, instead of favoring a more tree-like, connected graph. The graphical lasso and neighborhood selection graphs tend to be highly disconnected, since these methods are completely agnostic to the connectedness requirement of extremal graphical models. Similar conclusions can be drawn from the three graphs on the bottom row of \cref{fig:vanilla}, which are obtained by applying the same methodology to the currency exchange data, each of which contains the same number (49) of edges: the graphical lasso graph contains a densely connected group comprised mostly of the South Asian currencies, but is divided into a large number of connected components.

\begin{figure}[H]
	\centering
	\includegraphics[width=.3\textwidth, page=1, trim={5em 3em 5em 5em}]{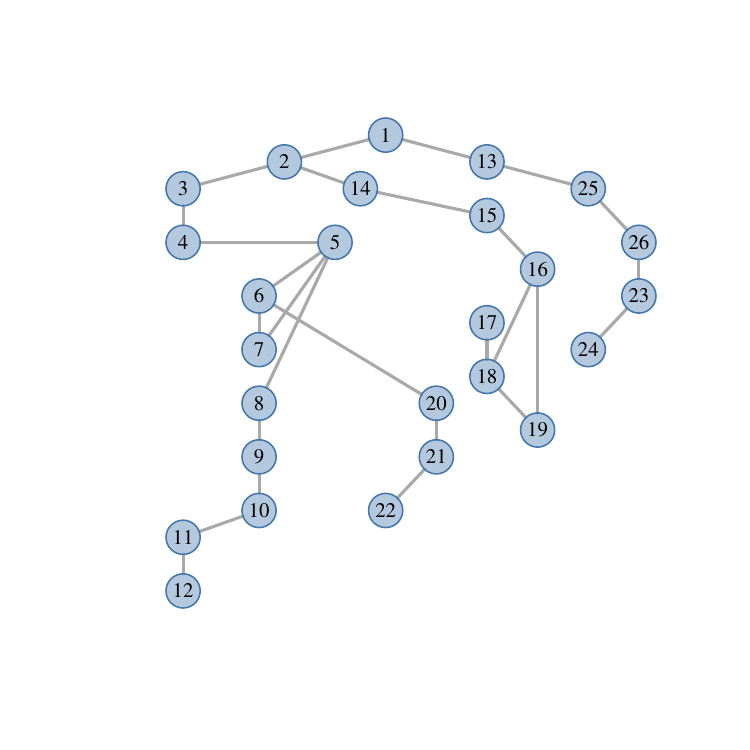}
	\hspace{1em}
	\includegraphics[width=.3\textwidth, page=2, trim={5em 3em 5em 5em}]{img/comparison_danube26_27.pdf}
	\hspace{1em}
	\includegraphics[width=.3\textwidth, page=3, trim={5em 3em 5em 5em}]{img/comparison_danube26_27.pdf}
	\\
	\includegraphics[width=.3\textwidth, page=1, trim={5em 5em 5em 5em}]{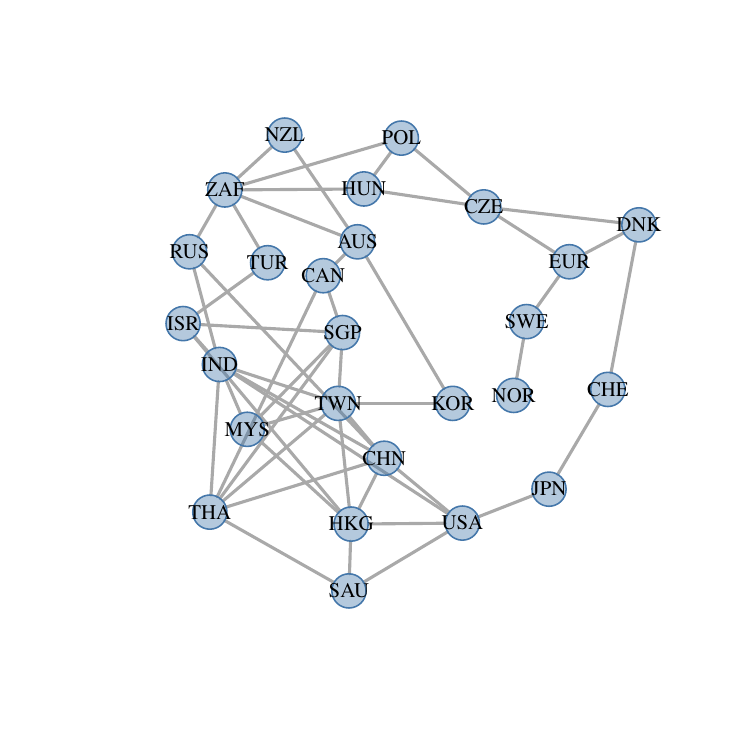}
	\hspace{1em}
	\includegraphics[width=.3\textwidth, page=2, trim={5em 5em 5em 5em}]{img/comparison_exchange_49.pdf}
	\hspace{1em}
	\includegraphics[width=.3\textwidth, page=3, trim={5em 5em 5em 5em}]{img/comparison_exchange_49.pdf}
	\caption{Top: estimated graphs of the same sparsity level (27 edges) obtained by applying \eglearn{} (left), the standard graphical lasso (middle) and standard neighborhood selection (right) to the Danube river data. Bottom: estimated graphs of the same sparsity (49 edges) obtained by applying \eglearn{} (left), the standard graphical lasso (middle) and standard neighborhood selection (right) to the currency exchange data.}
	\label{fig:vanilla}
\end{figure}

\section{Proofs of extremal graph recovery results}
\label{sec:proofeglearn}

\subsection{Proof of \Cref{th:NS}} \label{sec:proofNS}

When running the neighborhood selection algorithm on $\Sigma^{(m)}_{\setminus m, \setminus m}$, $m \in V$, recall that $\rho_{m, 1}^\ns, \dots, \rho_{m, d-1}^\ns$ are used as penalty parameters. Let
\begin{align*}
	\mu_{m, \ell} &:= \lambda_{\min} (\Sigma^{(m)}_{S_{m, \ell}, S_{m, \ell}}),
	\\
	\kappa_{m, \ell} &:= \normop{\Sigma^{(m)}_{S_{m, \ell}^c, S_{m, \ell}}}_\infty,
	\\
	\vartheta_{m, \ell} &:= \normop{ (\Sigma^{(m)}_{S_{m, \ell}, S_{m, \ell}})^{-1}}_\infty,
	\\
	s_{m, \ell} &:= |S_{m, \ell}|,
	\\
	C_{m, \ell}^\ns &:= \frac{2}{3} \min\bigg\{
		\frac{\mu_{m, \ell}}{2s_{m, \ell}},
		\frac{\eta_{m, \ell}^\ns}{4\vartheta_{m, \ell} (1 + \kappa_{m, \ell} \vartheta_{m, \ell}) s_{m, \ell}},
	\\
	&\phantom{:= \min\bigg\{\ }
		\frac{\min_{i \in S_{m, \ell}} \frac{|{B}_{i\ell}|}{{B}_{\ell\ell}} - \vartheta_{m, \ell} \rho_{m, \ell}^\ns}{2 \vartheta_{m, \ell} (1 + \kappa_{m, \ell} \vartheta_{m, \ell})},
		\frac{\rho_{m, \ell}^\ns \eta_{m, \ell}^\ns}{8(1 + \kappa_{m, \ell} \vartheta_{m, \ell})^2}
		\bigg\}.
\end{align*}
By \cref{prop:NS}, we have
\[
	\{(i, j) \in (V\setminus\{m\})^2: \tilde Z^{(m)} = 1\} = \{(i, j) \in (V\setminus\{m\})^2 \cap E\}
\]
provided that
\[
	\|\hat \Sigma^{(m)} - \Sigma^{(m)} \|_\infty < \frac{3}{2} \min_{\ell \neq m} C_{m, \ell}^\ns.
\]
Note that the use of \cref{prop:NS} depends on the assumption that $\hat\Sigma^{(m)}_{\setminus m, \setminus m}$ is symmetric positive semi-definite. However, this is sure to hold by definition: the Farris transform in~\eqref{eq:GammatoSigma} maps strictly conditionally negative definite variogram matrices to positive definite covariance matrices. By continuity, the empirical variogram will then be mapped to a matrix $\hat\Sigma^{(m)}_{\setminus m, \setminus m}$ that lies in the closure of the cone of positive definite covariance matrices, i.e., that is symmetric positive semi-definite.

From the definition of $\Sigma^{m}, \hat \Sigma^{(m)}$ through $\Gamma, \hat\Gamma$,
\[
	\max_{m \in V} \|\hat \Sigma^{(m)} - \Sigma^{(m)} \|_\infty \leq \frac{3}{2} \|\hat \Gamma - \Gamma \|_\infty.
\]
The \eglearn\ algorithm correctly learns the presence (resp. absence) of an edge $(i, j)$ if a 1 (resp. a 0) rightfully appears in position $(i, j)$ of $\tilde Z^{(m)}$ for at least $\lfloor d/2 \rfloor$ values of $m \notin \{i, j\}$. This is clearly the case for each pair $(i, j)$ if at least $\lfloor d/2 \rfloor + 2$ of the neighborhood selections perfectly succeed, which is guaranteed if $\|\hat \Gamma - \Gamma \|_\infty < \min_{\ell \neq m} C_{m, \ell}^\ns$ for at least $\lfloor d/2 \rfloor + 2$ values of $m$. Noting that the assumed bound $C^\ns$ lower bounds each $C_{m, \ell}^\ns$, the first statement of the theorem is proved.

For a proof of the second part, let
\[
	\lambda := \sqrt{\tfrac{3}{c} \log d + \sigma}
\]
with $c$ as in \cref{thm:concentration}, and $\sigma \to \infty$ but $\sigma = o(k/(\log k)^8)$. Then, by assumption on $d$ and $k$, we have eventually $\lambda \leq \sqrt{k}/(\log n)^4$, so that \cref{thm:concentration} applies. It states that with probability at least $1 - M\exp\{ 3\log d - c\lambda^2\} = 1 - o(1)$,
\[
	\|\Gam - \Gamma\|_\infty \lesssim \Big(\frac{k}{n}\Big)^\xi (\log(n/k))^2 + \frac{\sqrt{\log d + \sigma}}{\sqrt{k}}.
\]
By assumption,
\[
	\Big(\frac{k}{n}\Big)^\xi (\log(n/k))^2 + \frac{\sqrt{\log d}}{\sqrt{k}} = o(C^\ns),
\]
so if $\sigma$ is chosen to diverge slowly enough, we have $\Prob(\|\Gam - \Gamma\|_\infty < C^\ns) \to 1$.
\hfill $\Box$

\subsection{Proof of \Cref{th:GL}} \label{sec:proofGL}

When running the graphical lasso algorithm on $\Sigma_{\setminus m, \setminus m}^{(m)}$, $m \in V$, recall that $\rho_m^\gl$ is used as penalty parameter. Let
\begin{align*}
	\kappa_{\Sigma, m} &:= \opnorm{\Sigma^{(m)}}{\infty}, \quad \kappa_{\Omega, m} := \opnorm{(\Omega_{S_m, S_m}^{(m)})^{-1}}{\infty}, \quad \chi_m := 6\kappa_{\Sigma, m} \kappa_{\Omega, m} \Big( 1 \vee \frac{9\kappa_{\Sigma, m}^2 \kappa_{\Omega, m}}{\eta_m^\gl} \Big),
	\\
	s_m &:= \max_{\ell \in [d-1]} s_{m, \ell},
	\\
	C_m^\gl &:= \frac{2}{3} \min \bigg\{
		\min_{i \in [d-1]} \Sigma_{ii}^{(m)},
		\frac{\eta_m^\gl \rho_m^\gl}{8},
		\frac{1}{\chi_m s_m} - \rho_m^\gl,
		\frac{\theta_{\min}^\gl}{4\kappa_{\Omega, m}} - \rho_m^\gl
		\bigg\}.
\end{align*}
By \cref{prop:GL}, we have
\[
	\{(i, j) \in (V\setminus\{m\})^2: \tilde Z^{(m)} = 1\} = \{(i, j) \in (V\setminus\{m\})^2 \cap E\}
\]
provided that
\[
	\|\hat \Sigma^{(m)} - \Sigma^{(m)} \|_\infty < \frac{3}{2} C_m^\gl.
\]
Similarly to the proof of \cref{th:NS}, deduce that whenever $\|\hat\Gamma - \Gamma\|_\infty < C_m^\gl$ for at least $\lfloor d/2 \rfloor + 2$ values of $m$, \eglearn\ correctly recovers the extremal graph $G$. Moreover, $C^\gl \leq C_m^\gl$ for each $m \in V$, which completes the proof of the first part. The proof of the second statement is identical to that of the second statement in \cref{th:NS}.
\hfill $\Box$

\subsection{Sparsistency results for neighborhood selection and graphical lasso}

Let ${A}$ be a $p$-dimensional covariance matrix, ${B} := {A}^{-1}$, $\hat{A}$ be an estimator and $\eps := \|\hat{A} - {A}\|_\infty$. Recall the definition of the graph $G({B})$ associated to ${B}$, i.e., the graph with edge set $\{(i, j): i \neq j, {B}_{ij} \neq 0\}$.

We start by discussing the neighborhood selection algorithm. Define the neighborhood of a node $\ell$ in the graph $G(B)$ by $\nei(\ell) := \{i \in [p]\setminus\ell: B_{i\ell} \neq 0\}$. Let
\begin{align*}
	\mu_\ell &:= \lambda_{\min} ({A}_{\nei(\ell), \nei(\ell)}),
	\\
	\kappa_\ell &:= \normop{{A}_{\nei(\ell)^c, \nei(\ell)}}_\infty,
	\\
	\vartheta_\ell &:= \normop{ ({A}_{\nei(\ell), \nei(\ell)})^{-1}}_\infty,
	\\
	s_\ell &:= |\nei(\ell)|,
	\\
	\eta_\ell &:= 1 - \normop{{A}_{\nei(\ell)^c, \nei(\ell)} ({A}_{\nei(\ell), \nei(\ell)})^{-1}}_{\infty}.
\end{align*}

\begin{prop} \label{prop:NS}
	
	Assume that $\hat A$ is symmetric positive semi-definite, that $\min_{\ell \in [p]} \eta_\ell > 0$, and that
	\[
		\eps < \min_{\ell \in [p]} \min\bigg\{
		\frac{\mu_\ell}{2s_\ell},
		\frac{\eta_\ell}{4\vartheta_\ell (1 + \kappa_\ell \vartheta_\ell) s_\ell},
		\frac{\min_{i \in \nei(\ell)} \frac{|{B}_{i\ell}|}{{B}_{\ell\ell}} - \vartheta_\ell \rho_\ell}{2 \vartheta_\ell (1 + \kappa_\ell \vartheta_\ell)},
		\frac{\rho_\ell \eta_\ell}{8(1 + \kappa_\ell \vartheta_\ell)^2}
		\bigg\}.
	\]
	Then the graph $G_\text{NS}(\hat{A})$ obtained through the neighborhood selection in \cref{algo:NS} with penalty parameters $\rho_1, \dots, \rho_p$ is equal to $G({B})$.
	
\end{prop}

We now discuss the graphical lasso. Define the maximal edge degree $s := \max_{\ell \in [p]} s_\ell$, with $s_\ell$ as earlier, $S$ to be the augmented edge set $\{(i, j): {B}_{ij} \neq 0\}$, including the self loops $(i, i)$, and $S^c$ to be its complement in $V^2$. Let $\Omega := {A} \otimes {A}$,
\begin{align*}
	\kappa_{A} &:= \opnorm{{A}}{\infty}, \quad \kappa_\Omega := \opnorm{(\Omega_{SS})^{-1}}{\infty}, \quad \chi := 6\kappa_{A} \kappa_\Omega \Big( 1 \vee \frac{9\kappa_{A}^2 \kappa_\Omega}{\alpha} \Big),
	\\
	\alpha &:= 1 - \opnorm{\Omega_{S^c S} (\Omega_{SS})^{-1}}{\infty}.
\end{align*}

\begin{prop} \label{prop:GL}
	
	Assume that $\alpha>0$ and that
	\[
		\eps < \min\bigg\{ \min_{i \in [p]} {A}_{ii}, \frac{\alpha\rho}{8}, \frac{1}{\chi s} - \rho, \frac{1}{4\kappa_\Omega} \min_{i \neq j, {B}_{ij} \neq 0} |{B}_{ij}| - \rho \bigg\}.
	\]
	Then the graph $G_\text{GL}(\hat{A})$ obtained through the graphical lasso with penalty parameter $\rho$ is equal to $G({B})$.
	
\end{prop}

\begin{rem}
	
	For both algorithms, $s\eps \to 0$ is sufficient for model selection consistency as the sample size increases, if everything else is constant and the relevant incoherence condition is satisfied.
	
\end{rem}

\section{Sparsistency results for neighborhood selection and graphical lasso: proofs}
\label{sec:proofgraph}

\subsection{Proof of \cref{prop:NS}}

First note that if each of the lasso regressions therein succeeds in recovering the corresponding neighborhood, then clearly \cref{algo:NS} recovers the right graph.

We hereby fix an arbitrary index $\ell \in [p]$. Now, note that the assumed bound on $\eps$ implies
\begin{align}
	s_\ell \eps &\leq \mu_\ell/2, \label{eq:epscond1}
	\\
	\vartheta_\ell s_\ell \eps &\leq 1/2, \label{eq:epscond2}
	\\
	2\vartheta_\ell (1 + \kappa_\ell \vartheta_\ell) s_\ell \eps &\leq \eta_\ell/2, \label{eq:epscond3}
	\\
	2\vartheta_\ell (1 + \kappa_\ell \vartheta_\ell) \eps + \vartheta_\ell \rho_\ell &< \min_{i \in \nei(\ell)} \frac{|{B}_{i\ell}|}{{B}_{\ell\ell}}, \label{eq:epscond4}
	\\
	2(1 + \kappa_\ell \vartheta_\ell)^2 \eps &< \frac{\rho_\ell \eta_\ell}{4}. \label{eq:epscond5}
\end{align}
For the remainder of the proof, $\ell$ will be kept fix and hence will be partially removed from the notation. In particular, the subscripts of $\mu_\ell, \kappa_\ell, \vartheta_\ell, s_\ell, \eta_\ell, \rho_\ell$ will be omitted. Our goal is to show then that
\begin{equation} \label{eq:lasso}
	\hat\theta := \mathop{\arg\min}_{\theta \in \R^{p-1}} \big\{ -2\hat{A}_{\ell, \setminus\ell} \theta + \theta^\top \hat{A}_{\setminus\ell, \setminus\ell} \theta + \rho \|\theta\|_1 \big\}
\end{equation}
has the same support (i.e., the set of indices where it is nonzero) as ${B}_{\setminus\ell, \ell}$.

Partition $[p]$ into the three subsets $\{\ell\}$, $\nei(\ell)$ and $\nei(\ell)^c \setminus \{\ell\}$. We shall index elements of vectors and rows/columns of matrices by $\ell$, 1 and 2, respectively, to denote those subsets, e.g.,
\begin{align*}
	\hat {A}_{11} &:= (\hat {A}_{ij})_{i \in \nei(\ell), j \in \nei(\ell)}
	\\
	\hat {A}_{21} &:= (\hat {A}_{ij})_{i \in \nei(\ell)^c \backslash \{\ell\}, j \in \nei(\ell)}
	\\
	\hat {A}_{1\ell} &:= (\hat {A}_{i\ell})_{i \in \nei(\ell)}
	\\
	\hat {A}_{2\ell} &:= (\hat {A}_{i\ell})_{i \in \nei(\ell)^c \backslash \{\ell\}};
\end{align*}
similarly define the population versions ${A}_{11}$, ${A}_{21}$, ${A}_{1\ell}$ and ${A}_{2\ell}$. We use the same notation for partitioning the matrix ${B}$. We now show that \cref{eq:epscond1,eq:epscond2,eq:epscond3,eq:epscond4,eq:epscond5} imply the bounds
\begin{align}
	\normop{\hat {A}_{21}  (\hat{A}_{11})^{-1}}_{\infty} &\leq 1 - \eta/2 \label{eq:h1}
	\\
	\Big\| (\hat{A}_{11})^{-1} \hat{A}_{1\ell} + \frac{{B}_{1\ell}}{{B}_{\ell\ell}} \Big\|_\infty + \frac{\rho}{2} \normop{(\hat{A}_{11})^{-1}}_{\infty} &< \min_{i \in \nei(\ell)} \frac{|{B}_{i\ell}|}{{B}_{\ell\ell}} \label{eq:h2}
	\\
	\Big\| \hat{A}_{21} (\hat{A}_{11})^{-1} \hat{A}_{1\ell}  - \hat{A}_{2\ell} \Big\|_{\infty} &< \frac{\rho\eta}{4} \label{eq:h3}.
\end{align}
Subsequently, by adapting the arguments of \cite{ZY06}, we will show that the invertibility of $\hat{A}_{11}$, along with \cref{eq:h1,eq:h2,eq:h3}, imply the result.

\noindent\textbf{Preliminaries to the proofs of \cref{eq:h1,eq:h2,eq:h3}:} Let us start by obtaining a few useful bounds and identities.

We first prove that under the assumptions made, the matrix $\hat{A}_{11}$ is invertible. Observe that
\begin{align*}
\lambda_{\min}(\hat{A}_{11}) &= \min_{b: \|b\|_2=1} b^\top \hat{A}_{11} b \geq \mu - \sup_{b: \|b\|_2=1} b^\top (\hat{A}_{11} - {A}_{11}) b \geq \mu - s \|\hat{A}_{11} - {A}_{11}\|_\infty 
\\
& = \mu - s\eps \geq \mu/2,
\end{align*}
by \cref{eq:epscond1}. In the second inequality, we have used the known result that the spectral norm of any square matrix is upper bounded by its dimension times its maximum norm. Most importantly, we have established that $\hat{A}_{11}$ is invertible.

Now, by sub-multiplicativity of operator norms,
\begin{align*}
	\normop{(\hat{A}_{11})^{-1} - ({A}_{11})^{-1}}_\infty &= \normop{(\hat{A}_{11})^{-1} ({A}_{11} - \hat{A}_{11}) ({A}_{11})^{-1}}_\infty
	\\
	&\leq \normop{({A}_{11})^{-1}}_\infty \normop{(\hat{A}_{11})^{-1}}_\infty \normop{\hat{A}_{11} - {A}_{11} }_\infty
	\\
	&\leq \vartheta \Big(\vartheta + \normop{(\hat{A}_{11})^{-1} - ({A}_{11})^{-1}}_\infty \Big) s\eps.
\end{align*}
Rearranging yields
\begin{equation} \label{eq:sigmainvbound}
	\normop{(\hat{A}_{11})^{-1} - ({A}_{11})^{-1}}_\infty \leq \frac{\vartheta^2 s\eps}{1 - \vartheta s\eps} \leq 2\vartheta^2 s\eps,
\end{equation}
since the $L^\infty/L^\infty$-operator norm of a square matrix is also upper bounded by its dimension times its maximum norm. The last inequality is due to \cref{eq:epscond2}. For further use, note that this implies
\begin{equation} \label{eq:sigmainvbound2}
	\normop{(\hat{A}_{11})^{-1}}_\infty \leq \vartheta + \normop{(\hat{A}_{11})^{-1} - ({A}_{11})^{-1}}_\infty \leq \vartheta + 2\vartheta^2 s\eps \leq 2\vartheta,
\end{equation}
where in the last inequality we applied \cref{eq:epscond2} again.

Using \cref{eq:sigmainvbound}, it is possible to obtain a sharper bound on the maximum norm difference between $({A}_{11})^{-1}$ and $(\hat{A}_{11})^{-1}$. Indeed, note that for any matrices $T_1$, $T_2$, we have
\begin{equation} \label{eq:matrixproductnorm}
	\|T_1 T_2\|_\infty \leq \|T_1\|_\infty \normop{T_2}_1,
\end{equation}
which reduces to $\|T_1\|_\infty \|T_2\|_1$ if $T_2$ is a column vector. Similarly,
\begin{equation} \label{eq:matrixproductnorm2}
	\|T_1 T_2\|_\infty \leq \normop{T_1}_\infty \|T_2\|_\infty.
\end{equation}
Repeatedly using those facts, along with symmetry,
\begin{align}
	\Big\| (\hat{A}_{11})^{-1} - ({A}_{11})^{-1} \Big\|_\infty &= \Big\| (\hat{A}_{11})^{-1} ({A}_{11} - \hat{A}_{11}) ({A}_{11})^{-1} \Big\|_\infty \notag
	\\
	&\leq \Big\| (\hat{A}_{11})^{-1} ({A}_{11} - \hat{A}_{11}) \Big\|_\infty \normop{({A}_{11})^{-1}}_1 \notag
	\\
	&\leq \normop{(\hat{A}_{11})^{-1}}_\infty \Big\| {A}_{11} - \hat{A}_{11} \Big\|_\infty \normop{({A}_{11})^{-1}}_1 \notag
	\\
	&\leq 2\vartheta^2 \eps, \label{eq:sigmainvboundew}
\end{align}
where in the last step we used the fact that the $L^1/L^1$- and $L^\infty/L^\infty$-operator norms of a symmetric matrix (in this case, $({A}_{11})^{-1}$) are equal, along with \cref{eq:sigmainvbound2}.

We now prove that
\begin{equation} \label{eq:linpred}
	({A}_{11})^{-1} {A}_{1\ell} = -\frac{{B}_{1\ell}}{{B}_{\ell\ell}}.
\end{equation}
If $\bW \sim N(0, A)$, it is known that $-{B}_{1\ell}/{B}_{\ell\ell}$ is the vector of non-zero coefficients for optimal linear prediction of $W_\ell$ using $\bW_{\setminus\ell}$. Since removing the non-predictor variables does not change the prediction, we have
\[
	-\frac{{B}_{1\ell}}{{B}_{\ell\ell}} = -\frac{{B}_{1\ell}^*}{{B}_{\ell\ell}^*}
\]
where
\[
	{B}^* := \begin{bmatrix}
								{A}_{11} & {A}_{1\ell}
								\\
								{A}_{\ell 1} & {A}_{\ell\ell}
							\end{bmatrix}^{-1}.
\]
By the inversion formula for block matrices,
\[
	-{B}_{1\ell}^* = \frac{1}{{A}_{\ell\ell}} \Big( {A}_{11} - \frac{{A}_{1\ell}{A}_{\ell 1}}{{A}_{\ell\ell}} \Big)^{-1} {A}_{1\ell}, \quad {B}_{\ell\ell}^* = \frac{1}{{A}_{\ell\ell}} \Big( 1 - \frac{{A}_{\ell 1}({A}_{11})^{-1} {A}_{1\ell}}{{A}_{\ell\ell}} \Big)^{-1},
\]
so letting $\lambda := \frac{{A}_{\ell 1}({A}_{11})^{-1} {A}_{1\ell}}{{A}_{\ell\ell}} \in (0, 1)$,
\[
	-\frac{{B}_{1\ell}^*}{{B}_{\ell\ell}^*} = (1 - \lambda) \Big( {A}_{11} - \frac{{A}_{1\ell}{A}_{\ell 1}}{{A}_{\ell\ell}} \Big)^{-1} {A}_{1\ell}.
\]
Applying Woodbury's matrix inversion formula, we find
\begin{align*}
	\Big( {A}_{11} - \frac{{A}_{1\ell}{A}_{\ell 1}}{{A}_{\ell\ell}} \Big)^{-1} &= ({A}_{11})^{-1} - ({A}_{11})^{-1} {A}_{1\ell} \big( -{A}_{\ell\ell} + {A}_{\ell 1} ({A}_{11})^{-1} {A}_{1\ell} \big)^{-1} {A}_{\ell 1} ({A}_{11})^{-1}
	\\
	&= ({A}_{11})^{-1} - ({A}_{11})^{-1} {A}_{1\ell} \frac{1}{{A}_{\ell\ell} (\lambda-1)} {A}_{\ell 1} ({A}_{11})^{-1}.
\end{align*}
Simple matrix algebra yields
\begin{multline*}
	-\frac{{B}_{1\ell}^*}{{B}_{\ell\ell}^*} = (1-\lambda) ({A}_{11})^{-1} {A}_{1\ell} \Big( 1 - \frac{1}{{A}_{\ell\ell} (\lambda-1)} {A}_{\ell 1} ({A}_{11})^{-1} {A}_{1\ell} \Big)
    \\
    = (1-\lambda) ({A}_{11})^{-1} {A}_{1\ell} \Big( 1 - \frac{\lambda}{\lambda-1} \Big) = ({A}_{11})^{-1} {A}_{1\ell},
\end{multline*}
which finally establishes \cref{eq:linpred}.

Similarly, we prove that
\begin{equation} \label{eq:condcov}
	{A}_{21} ({A}_{11})^{-1} {A}_{1\ell} = {A}_{2\ell}.
\end{equation}
Indeed, notice that the Schur complement of ${A}_{11}$ in ${A}$,
\[
	{A}_{(2, \ell), (2, \ell) \given 1}
	:= \begin{bmatrix}
				{A}_{22} & {A}_{2\ell}
				\\
				{A}_{\ell 2} & {A}_{\ell\ell}
			\end{bmatrix}
		- \begin{bmatrix}
				{A}_{21}
				\\
				{A}_{\ell 1}
			\end{bmatrix}
		({A}_{11})^{-1}
		\begin{bmatrix}
			{A}_{12} & {A}_{1\ell}
		\end{bmatrix},
\]
is the conditional covariance matrix of the random vector $\bW_{\nei(\ell)^c}$ given $\bW_{\nei(\ell)}$. The off-diagonal block of ${A}_{(2, \ell), (2, \ell) \given 1}$, that is, ${A}_{2\ell} - {A}_{21} ({A}_{11})^{-1} {A}_{1\ell}$, is therefore the conditional covariance between $W_\ell$ and the other variables not in its neighborhood, given the variables in its neighborhood. By definition of the neighborhood, this is zero, hence \cref{eq:condcov} holds.

\noindent\textbf{Proof of \cref{eq:h1}:} Using \cref{eq:sigmainvbound}, we have
\begin{align*} \label{eq:sigmabound2}
	&\normop{\hat {A}_{21} (\hat{A}_{11})^{-1} - {A}_{21} ({A}_{11})^{-1}}_{\infty}
    \\
    &\leq \normop{\hat {A}_{21} (\hat{A}_{11})^{-1} - {A}_{21} (\hat{A}_{11})^{-1}}_{\infty} + \normop{{A}_{21} (\hat{A}_{11})^{-1} - {A}_{21} ({A}_{11})^{-1}}_{\infty}
	\\
	&\leq \normop{(\hat{A}_{11})^{-1}}_\infty \normop{\hat {A}_{21} - {A}_{21}}_\infty + \normop{{A}_{21}}_\infty \normop{(\hat{A}_{11})^{-1} - ({A}_{11})^{-1}}_\infty
	\\
	&\leq 2\vartheta s\eps + \kappa 2\vartheta^2 s\eps
	\\
	&= 2\vartheta (1 + \kappa\vartheta) s\eps,
\end{align*}
where the third inequality follows by applying \cref{eq:sigmainvbound,eq:sigmainvbound2}. Now by the reverse triangle inequality,
\begin{align*}
\normop{\hat {A}_{21}  (\hat{A}_{11})^{-1}}_{\infty} &\leq 
    \normop{{A}_{21}  ({A}_{11})^{-1}}_{\infty} + \normop{\hat {A}_{21} (\hat{A}_{11})^{-1} - {A}_{21} ({A}_{11})^{-1}}_{\infty}
    \\
    &\leq 1- \eta + 2\vartheta (1 + \kappa\vartheta) s\eps.
\end{align*}
\Cref{eq:h1} then follows from \cref{eq:epscond3}.

\noindent\textbf{Proof of \cref{eq:h2}:} First, by \cref{eq:linpred},
\begin{align*}
	\Big\| (\hat{A}_{11})^{-1} \hat{A}_{1\ell} + \frac{{B}_{1\ell}}{{B}_{\ell\ell}} \Big\|_\infty &= \Big\| (\hat{A}_{11})^{-1} \hat{A}_{1\ell} - ({A}_{11})^{-1} {A}_{1\ell} \Big\|_\infty
	\\
	&\leq \Big\| (\hat{A}_{11})^{-1} \hat{A}_{1\ell} - (\hat{A}_{11})^{-1} {A}_{1\ell} \Big\|_\infty + \Big\| (\hat{A}_{11})^{-1} {A}_{1\ell} - ({A}_{11})^{-1} {A}_{1\ell} \Big\|_\infty
	\\
	&\leq \normop{(\hat{A}_{11})^{-1}}_\infty \Big\| \hat{A}_{1\ell} - {A}_{1\ell} \Big\|_\infty + \Big\| \Big( (\hat{A}_{11})^{-1} - ({A}_{11})^{-1} \Big) \Big\|_\infty \|{A}_{1\ell}\|_1
	\\
	&\leq 2\vartheta \eps + 2\vartheta^2 \eps \kappa
	\\
	&= 2\vartheta (1 + \kappa\vartheta) \eps,
\end{align*}
where we first used \cref{eq:matrixproductnorm,eq:matrixproductnorm2}, then \cref{eq:sigmainvbound2,eq:sigmainvboundew}. Noting that
\[
	\frac{\rho}{2} \normop{(\hat{A}_{11})^{-1}}_{\infty} \leq \frac{\rho}{2} 2\vartheta = \vartheta\rho,
\]
\cref{eq:h2} follows from \cref{eq:epscond4}.

\noindent\textbf{Proof of \cref{eq:h3}:} First, by \cref{eq:condcov},
\begin{align*}
	\Big\| \hat{A}_{21} (\hat{A}_{11})^{-1} \hat{A}_{1\ell}  - \hat{A}_{2\ell} \Big\|_{\infty} &\leq \Big\| \hat{A}_{21} (\hat{A}_{11})^{-1} \hat{A}_{1\ell} - {A}_{21} ({A}_{11})^{-1} {A}_{1\ell} \Big\|_{\infty} + \Big\| \hat{A}_{2\ell} - {A}_{2\ell} \Big\|_{\infty},
\end{align*}
the second term of which is clearly upper bounded by $\eps$. The first term above is upper bounded by
\begin{align*}
	&\Big\| \hat{A}_{21} (\hat{A}_{11})^{-1} \hat{A}_{1\ell} - \hat{A}_{21} ({A}_{11})^{-1} {A}_{1\ell} \Big\|_{\infty} + \Big\| \hat{A}_{21} ({A}_{11})^{-1} {A}_{1\ell} - {A}_{21} ({A}_{11})^{-1} {A}_{1\ell} \Big\|_{\infty}
	\\
	&\quad \leq \normop{\hat{A}_{21}}_\infty \Big\| (\hat{A}_{11})^{-1} \hat{A}_{1\ell} - ({A}_{11})^{-1} {A}_{1\ell} \Big\|_\infty + \Big\| \hat{A}_{21} - {A}_{21} \Big\|_\infty \normop{({A}_{11})^{-1} {A}_{1\ell}}_1
	\\
	&\quad \leq (\kappa + s\eps) 2\vartheta (1 + \kappa\vartheta) \eps + \kappa\vartheta \eps
	\\
	&\quad \leq 2(1 + \kappa\vartheta)^2 \eps - \eps,
\end{align*}
where we used \cref{eq:matrixproductnorm,eq:matrixproductnorm2} and the result in the proof of \cref{eq:h2} above where we bound $\| (\hat{A}_{11})^{-1} \hat{A}_{1\ell} - ({A}_{11})^{-1} {A}_{1\ell} \|_\infty$. The final bound was obtained by applying \cref{eq:epscond2} and rearranging. Hence \cref{eq:h3} follows from \cref{eq:epscond5}.

\noindent\textbf{Proof that \cref{eq:lasso} recovers the support of ${B}_{\setminus\ell, \ell}$:} Note that
\[
	\nabla_\theta \Big\{ -2\hat{A}_{\ell, \setminus\ell}\theta + \theta^\top \hat{A}_{\setminus\ell, \setminus\ell} \theta \Big\} = 2\hat{A}_{\setminus\ell, \setminus\ell} \theta - 2\hat{A}_{\setminus\ell, \ell}
\]
and that the sub-differential of the 1-norm at a point $\theta \in \R^{p-1}$ is given by the set of all $x \in [-1, 1]^{p-1}$ such that
\[
	\theta_j \neq 0 \Longrightarrow x_j = \sign(\theta_j).
\]
Considering the optimization problem in \cref{eq:lasso}, the KKT conditions state that any point $\hat\theta$ satisfying
\begin{align}
	\big( 2\hat{A}_{\setminus\ell, \setminus\ell} \hat\theta - 2\hat{A}_{\setminus\ell, \ell} \big)_{\hat J} &= -\rho\sign(\hat\theta_{\hat J}) \label{eq:KKT1}
	\\
	\big\| \big( 2\hat{A}_{\setminus\ell, \setminus\ell} \hat\theta - 2\hat{A}_{\setminus\ell, \ell} \big)_{\hat J^c} \big\|_\infty &\leq \rho \label{eq:KKT2}
\end{align}
where $\hat J := \{j \in [p-1]: \hat\theta_j \neq 0\}$, is a solution. Following the arguments in the proof of Proposition 1 in \cite{ZY06}, we shall identify one such solution that has the right support. We will subsequently show that it is unique, utilizing arguments similar to (but not implied by) what is found in Section 2.1 of \cite{T13}.

Let $\theta^* := -{B}_{\setminus\ell, \ell}/{B}_{\ell\ell}$ and $\theta_1^*$, $\theta_2^*$ denote its subvectors indexed by $\nei(\ell)$ and $\nei(\ell)^c \backslash \{\ell\}$, respectively\footnote{For simplicity, we slightly abuse the notation here. In fact, $\theta_1^*$ and $\theta_2^*$ are the subvectors that, in the product ${A}_{\setminus\ell, \setminus\ell} \theta^*$, are multiplied by the columns of $A$ corresponding to variables in $\nei(\ell)$ and in $\nei(\ell)^c \backslash \{\ell\}$, respectively.}; we will use the same notation for $\hat\theta$ defined below. The candidate solution $\hat\theta$ is defined by $\hat\theta_2 = 0$ and
\[
	\hat\theta_1 := (\hat{A}_{11})^{-1} \Big( \hat{A}_{1\ell} - \frac{\rho}{2} \sign(\theta_1^*) \Big),
\]
where the $\sign$ function is applied to $\theta_1^*$ coordinate-wise. First note that by \cref{eq:h2},
\[
	\|\hat\theta_1 - \theta_1^*\|_\infty < \min_{i \in \nei(\ell)} |\theta_i^*|,
\]
hence $\sign(\theta_1^*) = \sign(\hat\theta_1)$. Thus, we find
\[
	2\hat{A}_{\nei(\ell), \setminus\ell} \hat\theta - 2\hat{A}_{1\ell} = 2\hat{A}_{11} \hat\theta_1 - 2\hat{A}_{1\ell} = -\rho\sign(\theta_1^*) = -\rho\sign(\hat\theta_1),
\]
i.e., \cref{eq:KKT1} is satisfied. Next, by \cref{eq:h1,eq:h3} we have
\begin{align*}
	\|2\hat{A}_{\nei(\ell)^c \backslash \{\ell\}, \setminus\ell} \hat\theta - 2\hat{A}_{2\ell}\|_\infty &= \|2\hat{A}_{21} \hat\theta_1 - 2\hat{A}_{2\ell}\|_\infty
	\\
	&= \Big\| 2\hat{A}_{21} (\hat{A}_{11})^{-1} \hat{A}_{1\ell} - 2\hat{A}_{2\ell} - \rho \hat{A}_{21} (\hat{A}_{11})^{-1} \sign(\theta_1^*) \Big\|_\infty
	\\
	&\leq 2 \Big\| \hat{A}_{21} (\hat{A}_{11})^{-1} \hat{A}_{1\ell} - \hat{A}_{2\ell} \Big\|_\infty + \rho \normop{\hat{A}_{21} (\hat{A}_{11})^{-1}}_\infty
	\\
	&< \frac{\rho\eta}{2} + \rho\Big( 1 - \frac{\eta}{2} \Big)
	\\
	&= \rho,
\end{align*}
i.e., \cref{eq:KKT2} is satisfied. Therefore, we have proved the existence of a solution $\hat\theta$ to \cref{eq:lasso} which has the same sign pattern (hence the same support) as $\theta^*$.

It remains to show that this solution is unique. We first prove a weaker statement: $\hat{A}_{\setminus\ell, \setminus\ell} \tilde\theta$ is unique across all solutions $\tilde\theta$ to \cref{eq:lasso}. Indeed, suppose that $\tilde\theta^{(1)}$ and $\tilde\theta^{(2)}$ are two distinct solutions such that $\hat{A}_{\setminus\ell, \setminus\ell} \tilde\theta^{(1)} \neq \hat{A}_{\setminus\ell, \setminus\ell} \tilde\theta^{(2)}$. Let $W$ be a (possibly rectangular) matrix such that $\hat{A}_{\setminus\ell, \setminus\ell} = W^\top W$. For instance, $W$ can be obtained from a Cholesky decomposition; this is possible by the assumption that $\hat A$ (and hence $\hat{A}_{\setminus\ell, \setminus\ell}$) is symmetric positive semi-definite. By assumption, $W\tilde\theta^{(1)} \neq W\tilde\theta{(2)}$. Both points being a solution means that the objective function in \cref{eq:lasso} attains its minimum value at both points: if
\[
	Q(\theta) := -2\hat{A}_{\ell, \setminus\ell} \theta + \|W\theta\|_2^2 + \rho\|\theta\|_1,
\]
then $Q(\tilde\theta^{(1)}) = Q(\tilde\theta^{(2)}) = \min_{\theta} Q(\theta) =: v_{\min}$, say. Then, evaluating $Q$ at a point $\alpha\tilde\theta^{(1)} + (1-\alpha)\tilde\theta^{(2)}$, for some $\alpha \in (0, 1)$, yields
\begin{align*}
	&-2\hat{A}_{\ell, \setminus\ell} \big( \alpha\tilde\theta^{(1)} + (1-\alpha)\tilde\theta^{(2)} \big) + \big\| \alpha W\tilde\theta^{(1)} + (1-\alpha) W\tilde\theta^{(2)} \big\|_2^2 + \rho\big\| \alpha\tilde\theta^{(1)} + (1-\alpha)\tilde\theta^{(2)} \big\|_1
	\\
	&\quad < \alpha \big\{ -2\hat{A}_{\ell, \setminus\ell} \tilde\theta^{(1)} + \|W\tilde\theta^{(1)}\|_2^2 + \rho\|\tilde\theta^{(1)}\|_1 \big\}
    \\
    &\quad \quad + (1-\alpha) \big\{ -2\hat{A}_{\ell, \setminus\ell} \tilde\theta^{(2)} + \|W\tilde\theta^{(2)}\|_2^2 + \rho\|\tilde\theta^{(2)}\|_1 \big\}
	\\
	&\quad = \alpha Q(\tilde\theta^{(1)}) + (1-\alpha) Q(\tilde\theta^{(2)})
	\\
	&\quad = v_{\min},
\end{align*}
where the strict inequality is a consequence of the strict convexity of the squared Euclidean norm and the convexity of the 1-norm. This is a contradiction since $v_{\min}$ was assumed to be the minimum value. Hence $\hat{A}_{\setminus\ell, \setminus\ell} \tilde\theta$ must be unique across all solutions $\tilde\theta$ of \cref{eq:lasso}.

Now, using this preliminary uniqueness result, we notice that for any solution $\tilde\theta$, we have
\[
	\Big\| 2\hat{A}_{\nei(\ell)^c \backslash \{\ell\}, \setminus\ell} \tilde\theta - 2\hat{A}_{2\ell} \Big\|_\infty = \Big\| 2\hat{A}_{\nei(\ell)^c \backslash \{\ell\}, \setminus\ell} \hat\theta - 2\hat{A}_{2\ell} \Big\|_\infty < \rho,
\]
hence $\tilde\theta_2 = 0$. Therefore,
\[
	2\hat{A}_{11} \tilde\theta_1 - 2\hat{A}_{1\ell} = 2\hat{A}_{\nei(\ell), \setminus\ell} \tilde\theta - 2\hat{A}_{1\ell} = 2\hat{A}_{\nei(\ell), \setminus\ell} \hat\theta - 2\hat{A}_{1\ell} = -\rho\sign(\theta_1^*),
\]
which uniquely defines $\tilde\theta_1$ by the invertibility of $\hat{A}_{11}$. Deduce that $\tilde\theta = \hat\theta$.
\hfill $\Box$

\subsection{Proof of \cref{prop:GL}}

First note that the assumed bound on $\eps$ implies
\begin{align}
	\eps &\leq \frac{\alpha\rho}{8}, \label{eq:glassocond1}
	\\
	2\kappa_\Omega (\eps + \rho) &\leq \min\bigg\{ \frac{1}{3\kappa_A s}, \frac{1}{3\kappa_A^3 \kappa_\Omega s} \bigg\}, \label{eq:glassocond2}
	\\
	6 \kappa_A^3 \kappa_\Omega^2 s (\eps + \rho) &\leq \frac{\alpha}{9} \leq \frac{1}{1 + 8/\alpha}, \label{eq:glassocond3}
	\\
	2\kappa_\Omega (\eps + \rho) &\leq \frac{1}{2} \min_{i \neq j, B_{ij} \neq 0} |B_{ij}|. \label{eq:glassocond4}
\end{align}
Without loss of generality, assume that $\rho>0$. Otherwise, \cref{eq:glassocond1} implies that $\hat A = A$, in which case the result is trivial.

The proof is heavily based on that of Theorems 1 and 2 of \cite{RWRY11}. Note that by assumption, each diagonal element of $\hat A$ satisfies $\hat A_{ii} \geq A_{ii} - \eps > 0$, $i \in [p]$. Then by Lemma 3 of that paper, the positivity of $\rho$ ensures that the solution $\hat B$ exists, is unique and satisfies
\[
	-\hat B^{-1} + \hat A + \rho Z = 0,
\]
for some matrix $Z$ in the sub-differential of the off-diagonal norm at the point $\hat B$, as defined in \cite{RWRY11}. The strategy is now to consider the solution $\tilde B$ of the graphical lasso optimization problem with the additional constraint that $B_{S^c} = 0$, which is also guaranteed to exist and to be unique by \cref{eq:glassocond1}. Define $\Delta := \|\tilde B - B\|_\infty$. By \cref{eq:glassocond2}, the condition of Lemma 6 of \cite{RWRY11} is satisfied. It then follows from that result that $\|\Delta\|_\infty \leq 2\kappa_\Omega (\eps + \rho)$. Lemma 5 from that paper now implies that the matrix $R(\Delta)$, as defined therein, satisfies
\begin{align*}
	\|R(\Delta)\|_\infty \leq \frac{3}{2} \kappa_A^3 s \|\Delta\|_\infty^2
	\leq 6 \kappa_A^3 \kappa_\Omega^2 s (\eps + \rho)^2
	\leq \frac{\eps + \rho}{1 + 8/\alpha}
	\leq \frac{(\alpha/8 + 1)\rho}{1 + 8/\alpha} = \frac{\alpha\rho}{8},
\end{align*}
where the last three inequalities are due to the previously obtained bound on $\|\Delta\|_\infty$, to \cref{eq:glassocond3} and to \cref{eq:glassocond1}, respectively. Now that $\eps \vee \|R(\Delta)\|_\infty \leq \alpha\rho/8$, we may apply Lemma 4 of \cite{RWRY11} and find that in fact, $\hat B = \tilde B$. It follows, by definition of $\tilde B$, that $\hat B_{ij} = 0$ for all $(i, j) \in S^c$ and that for $(i, j) \in S$, $i \neq j$,
\[
	|\hat B_{ij} - B_{ij}| \leq \|\Delta\|_\infty \leq \frac{|B_{ij}|}{2}
\]
by \cref{eq:glassocond4}. That is, the element-wise error is too small for $\hat B_{ij}$ to reach 0. We have therefore guaranteed that the sparsity pattern of $\hat B$ is the same as that of $B$.
\hfill $\Box$

\section{Proof of Theorem~\ref{thm:concentration}}
\label{sec:proof}

We start by introducing useful additional notation and auxiliary variables that will be useful throughout this proof. For $\ell \in \{1, 2\}$, let $e_i^{(m),\ell} := \E[(\log Y_i^{(m)})^\ell]$ and $e_{ij}^{(m)} := \E[(\log Y_i^{(m)})(\log Y_j^{(m)})]$. Then we have
\begin{equation}
\label{eq:repGamma}
\Gamma_{ij}^{(m)} = e_i^{(m),2} + e_j^{(m),2} - 2e_{ij}^{(m)} - \big( e_i^{(m),1} - e_j^{(m),1} \big)^2 , \quad i \neq j, m \in V.
\end{equation} 
Similarly
\begin{align*}
	\hat\Gamma_{ij}^{(m)} &= \hat e_i^{(m),2} + \hat e_j^{(m),2} - 2\hat e_{ij}^{(m)} - \big( \hat e_i^{(m),1} - \hat e_j^{(m),1} \big)^2, \quad i \neq j, m \in V,
\end{align*}
where
\begin{align*}
	\hat e_i^{(m),\ell} &:= \frac{1}{k} \sum_{t=1}^n \Big\{\log \Big(\frac{k}{n \hat F_{i}(U_{ti})}\Big)\Big\}^\ell \Ind{\hat F_m(U_{tm}) \leq k/n},
	\\
	\hat e_{ij}^{(m)} &:= \frac{1}{k} \sum_{t=1}^n \log \Big(\frac{k}{n \hat F_{j}(U_{tj})}\Big) \log \Big(\frac{k}{n \hat F_{i}(U_{ti})}\Big) \Ind{\hat  F_m(U_{tm}) \leq k/n},
\end{align*}
$U_{ti} := 1 - F_i(X_{ti})$, $1 \leq t \leq n$, are independent and uniformly distributed, and $\hat F_i$ is the empirical distribution function of $(U_{ti})_{1 \leq t \leq n}$, satisfying $\hat F_i(U_{ti}) = \tfrac{n+1}{n} - \tilde F_i(X_{ti})$.
\footnote{Strictly speaking, in the implementation the quantities $\hat e$ above are scaled by $\tfrac{1}{k+1}$ rather than $\tfrac{1}{k}$. By a boundedness argument (see the proof of \cref{lemm:boundGamma}), one easily shows that the scaling difference does not affect the result of \cref{thm:concentration}.}

Let $i \neq j$ and $m$ be arbitrary. An expression for the estimation error is given by
\begin{align}
	\hat\Gamma_{ij}^{(m)} - \Gamma_{ij}^{(m)} &= \big( \hat e_i^{(m),2} - e_i^{(m),2} \big) + \big( \hat e_j^{(m),2} - e_j^{(m),2} \big) - 2\big( \hat e_{ij}^{(m)} - e_{ij}^{(m)} \big) \notag
	\\
	&\quad -2\big( e_i^{(m),1} - e_j^{(m),1} \big) \big( \big( \hat e_i^{(m),1} - e_i^{(m),1} \big) - \big( \hat e_j^{(m),1} - e_j^{(m),1} \big) \big) \notag
	\\
	&\quad - \big( \big( \hat e_i^{(m),1} - e_i^{(m),1} \big) - \big( \hat e_j^{(m),1} - e_j^{(m),1} \big) \big)^2, \label{eq:estimerror}
\end{align}
the last two terms stemming from the identity $y^2 - x^2 = 2x(y-x) + (y-x)^2$. In order to prove the result, it is sufficient to bound the differences
\[
	\hat e_i^{(m), \ell} - e_i^{(m), \ell}, \quad \hat e_m^{(m), \ell} - e_m^{(m), \ell}, \quad \hat e_{im}^{(m)} - e_{im}^{(m)}, \quad \hat e_{ij}^{(m)} - e_{ij}^{(m)}
\]
for all distinct triples $(i, j, m)$ and $\ell \in \{1, 2\}$. The terms $\hat e_m^{(m), \ell} - e_m^{(m), \ell}$ are entirely deterministic, since it is known that the observations $X_{tm}$ that are used for the estimator $\Gam^{(m)}$ have ranks $n-k+1, \dots, n$ (by continuity, it can be assumed that there are no ties in the data). They are on the order of $(\log k)^\ell/k$, as is proved in \cref{sec:emml}. The rest of the proof thus focuses on the other three differences.

\subsection{Preliminaries, additional notation {and structure of the proof}}

Recall that $\hat F_i$ is the empirical distribution function of $(U_{ti})_{1 \leq t \leq n}$ and denote its left-continuous inverse by $\hat F_i^-$, where $f^-(t) := \inf\{x: f(x) \geq t\}$. Consider the rescaled tail quantile processes
\begin{equation} \label{eq:defuni}
	u_n^{(i)}(x) := \frac{n}{k} \hat F_i^-(kx/n).
\end{equation}
Similarly to $R$ and its margins $R_J$, $J \subset V$, let
\begin{equation} \label{eq:def}
	\hat R_J^0(\bx_J) := \frac{1}{k} \sum_{t=1}^n \Ind{U_{ti} \leq \frac{k}{n} x_i, i \in J}, \quad \hat R_J(\bx_J) := \hat R_J^0(\hat \bx_J), \quad \bx_J \in [0, \infty)^{|J|},
\end{equation}
where $\hat\bx_J := (u_n^{(i)}(x_i))_{i \in J}$. The function $\hat R_J$ can be seen as the tail empirical copula of the random vector $\bU_J$. As an intermediate between $R_J$ and $\hat R_J$, let
\[
	R_{J, n}(\bx_J) := \frac{n}{k} \Prob\Big( \bU_J \leq \frac{k}{n} \bx_J \Big) = \frac{n}{k} \Prob\Big( F_J(\bX_J) > 1 - \frac{k}{n} \bx_J \Big),
\]
the pre-asymptotic version of $R_J$.

The function $R_J$ can be seen as a measure on $[0, \infty)^{|J|}$ and for measurable sets $A_i \subset \R$, we will write $R_J((A_i)_{i \in J})$ to denote $R_J(\otimes_{i \in J} A_i)$. If in place of one of the $A_i$ there is a number $a_i$, it will be understood that $A_i = [0, a_i]$. For example, $R_{ij}([x, \infty), y) = R_{ij}([x, \infty) \times [0, y])$. We use the same conventions for the functions $R_{J, n}$, $\hat R_J^0$ and $\hat R_J$, as well as $G_{J, n}$ and $\bar R_J$ to be defined later.

The functions $R_J$, $R_{J, n}$, $\hat R_J^0$ and $\hat R_J$ enjoy certain properties which will be used throughout the proof, as well as throughout \cref{sec:auxiliary}: they are non-decreasing in each component and are upper bounded by their minimum argument. Beyond component-wise monotonicity, the induced measures, as described above, are non-negative. The functions $R_J$ moreover inherit the homogeneity property of multivariate Pareto distributions ($R_J(q\bx_J) = q R_J(\bx_J)$, $q \geq 0$).

From now on, fix a number $a \in (0, 1)$. It is proved in \cref{lemm:intrep} that for all distinct triples $(i, j, m)$ and $\ell \in \{1, 2\}$, we have
\begin{align}
	& \hat e_i^{(m), \ell} - e_i^{(m), \ell}\notag
    \\
    &= \int_a^1 \frac{\big( \hat R_{im}(x, 1) - R_{im}(x, 1) \big) (-2\log x)^{\ell-1}}{x} \d x \notag
	\\
	&\quad - \int_1^{n/k} \frac{\big( \hat R_{im}([x, \infty), 1) - R_{im}([x, \infty), 1) \big) (-2\log x)^{\ell-1}}{x} \d x \notag
	\\
	&\quad + {O\bigg( \Big( \frac{k}{n} \Big)^\xi \log(n/k) + \frac{(\log(n/k) + \log(1/a))^2}{k} + a(\log(n/k) + \log(1/a)) \bigg)}, \label{eq:erriml}
	\\ \notag
	\\
	\hat e_{im}^{(m)} - e_{im}^{(m)} &= \int_a^1 \int_a^1 \frac{\hat R_{im}(x, y) - R_{im}(x, y)}{xy} \d x \d y \notag
	\\
	&\quad - \int_a^1 \int_1^{n/k} \frac{\hat R_{im}([x, \infty), y) - R_{im}([x, \infty), y)}{xy} \d x \d y \notag
	\\
	&\quad + {O\bigg( \Big( \frac{k}{n} \Big)^\xi \log(n/k) + \frac{(\log(n/k) + \log(1/a))^2}{k} + a(\log(n/k) + \log(1/a)) \bigg)}, \label{eq:errimm}
	\\ \notag
	\\
	\hat e_{ij}^{(m)} - e_{ij}^{(m)} &= \int_a^1 \int_a^1 \frac{\hat R_{ijm}(x, y, 1) - R_{ijm}(x, y, 1)}{xy} \d x \d y \notag
	\\
	&\quad - \int_a^1 \int_1^{n/k} \frac{\hat R_{ijm}([x, \infty), y, 1) - R_{ijm}([x, \infty), y, 1)}{xy} \d x \d y \notag
	\\
	&\quad - \int_1^{n/k} \int_a^1 \frac{\hat R_{ijm}(x, [y, \infty), 1) - R_{ijm}(x, [y, \infty), 1)}{xy} \d x \d y \notag
	\\
	&\quad + \int_1^{n/k} \int_1^{n/k} \frac{\hat R_{ijm}([x, \infty), [y, \infty), 1) - R_{ijm}([x, \infty), [y, \infty), 1)}{xy} \d x \d y \notag
	\\
	&\quad + {O\bigg( \Big( \frac{k}{n} \Big)^\xi \log(n/k) + \frac{(\log(n/k) + \log(1/a))^2}{k} + a(\log(n/k) + \log(1/a)) \bigg)} \label{eq:errijm}
\end{align}
almost surely, where the error terms are not stochastic. We shall separately bound each of the eight integrals above. Denote these integrals, in order of appearance in \cref{eq:erriml,eq:errimm,eq:errijm}, by $\cI_i^{(m), \ell, -}$, $\cI_i^{(m), \ell, +}$, $\cI_{im}^{(m), --}$, $\cI_{im}^{(m), +-}$, $\cI_{ij}^{(m), --}$, $\cI_{ij}^{(m), +-}$, $\cI_{ij}^{(m), -+}$ and $\cI_{ij}^{(m), ++}$.

The processes $\hat R_J - R_J$ can be decomposed into the stochastic error $\hat R_J - R_{J, n}$ and the difference $R_{J, n} - R_J$ between the tail at finite and infinite levels. Replacing $\hat R_J - R_J$ by $(\hat R_J - R_{J, n}) + (R_{J, n} - R_J)$, each integral $\cI_\cdot^{(m), \cdot}$ is written as
\[
	\cI_\cdot^{(m), \cdot} =: A_\cdot^{(m), \cdot} + B_\cdot^{(m), \cdot},
\]
where the $A$ terms are stochastic and the $B$ terms represent deterministic bias.

We proceed as follows. In \cref{sec:B}, it is shown that \cref{assum:tail} is sufficient to bound all the bias terms, up to a constant, by $(k/n)^\xi (\log(n/k) + \log(1/a))^2$. Subsequently, we prove in \cref{sec:A} concentration results for the stochastic terms which are then leveraged in \cref{subsec:proof} to complete the proof.

Before moving on, we highlight some consequences of \cref{assum:tail,prop:equiv} in terms of the functions $R_{J, n}$ and $R_J$. For any distinct triple $(i, j, m)$ and $q \in (0, 1]$,
\begin{align}
	\sup_{x \leq n/k, y \leq 1} |R_{ij, n}(x, y) - R_{ij}(x, y)| &\leq 2K \Big( \frac{k}{n} \Big)^\xi, \label{eq:B}
	\\
	\sup_{x \leq n/k, y \leq n/k, z \leq 1} |R_{ijm, n}(x, y, z) - R_{ijm}(x, y, z)| &\leq K \Big( \frac{k}{n} \Big)^\xi, \label{eq:B3}
	\\
	1 - R_{ijm}(q^{-1}, q^{-1}, 1) \leq 1 - R_{im}(q^{-1}, 1) + 1 - R_{jm}(q^{-1}, 1) &\leq 2Kq^\xi; \label{eq:T3}
\end{align}
to obtain the first inequality in \cref{eq:T3}, apply the inequality
\[
	P([0, q^{-1}] \times [0, q^{-1}]) \geq P([0, q^{-1}] \times [0, \infty)) + P([0, \infty) \times [0, q^{-1}]) - 1,
\]
valid for probability measures $P$ on $[0, \infty)^2$, with $P = R_{ijm}(\cdot \times [0, 1])$.

\subsection{The bias terms \texorpdfstring{$B$}{B}}
\label{sec:B}

\Cref{eq:B} directly implies
\begin{align*}
	\big| B_i^{(m), \ell, -} \big| &\leq \int_a^1 \frac{\big| R_{im, n}(x, 1) - R_{im}(x, 1) \big| (-2\log x)^{\ell-1}}{x} dx
	\\
	&\leq 2K \Big( \frac{k}{n} \Big)^\xi \int_a^1 \frac{(-2\log x)^{\ell-1}}{x} dx
	\\
	&\lesssim \Big( \frac{k}{n} \Big)^\xi (\log(1/a))^\ell,
	\\
	\\
	\big| B_i^{(m), \ell, +} \big| &\leq \int_1^{n/k} \frac{\big| R_{im, n}([x, \infty), 1) - R_{im}([x, \infty), 1) \big| (2\log x)^{\ell-1}}{x} dx
	\\
	&\leq 2K \Big( \frac{k}{n} \Big)^\xi \int_1^{n/k} \frac{(2\log x)^{\ell-1}}{x} dx
	\\
	&\lesssim \Big( \frac{k}{n} \Big)^\xi (\log(n/k))^\ell,
	\\
	\\
	\big| B_{im}^{(m), --} \big| &\leq \int_a^1 \int_a^1 \frac{\big| R_{im, n}(x, y) - R_{im}(x, y) \big|}{xy} dxdy
	\\
	&\leq 2K \Big( \frac{k}{n} \Big)^\xi \int_a^1 \int_a^1 \frac{1}{xy} dxdy
	\\
	&\lesssim \Big( \frac{k}{n} \Big)^\xi (\log(1/a))^2,
	\\
	\\
	\big| B_{im}^{(m), +-} \big| &\leq \int_a^1 \int_1^{n/k} \frac{\big| R_{im, n}([x, \infty), y) - R_{im}([x, \infty), y) \big|}{xy} dxdy
	\\
	&\leq 2K \Big( \frac{k}{n} \Big)^\xi \int_a^1 \int_1^{n/k} \frac{1}{xy} dxdy
	\\
	&\lesssim \Big( \frac{k}{n} \Big)^\xi (\log(n/k))(\log(1/a)).
\end{align*}

Note further that by \cref{eq:B3}, $B_{ij}^{(m), --}$ admits the same bound as $B_{im}^{(m), --}$. Similarly, $B_{ij}^{(m), +-}$, and by symmetry $B_{ij}^{(m), -+}$, admit the same bound as $B_{im}^{(m), +-}$. Finally, since
\[
	R_{ijm}([x, \infty), [y, \infty), 1) = 1 - R_{jm}(y, 1) - R_{im}(x, 1) + R_{ijm}(x, y, 1)
\]
and the same relation holds for $R_{ijm, n}$, \cref{eq:B,eq:B3} also imply that
\[
	\sup_{x \leq n/k, y \leq n/k} |R_{ijm, n}([x, \infty), [y, \infty), 1) - R_{ijm}([x, \infty), [y, \infty), 1)| \leq 5K \Big( \frac{k}{n} \Big)^\xi.
\]
Deduce that
\begin{align*}
	\big| B_{ij}^{(m), ++} \big| &\leq \int_1^{n/k} \int_1^{n/k} \frac{\big| R_{ijm, n}([x, \infty), [y, \infty), 1) - R_{ijm}([x, \infty), [y, \infty), 1) \big|}{xy} \d x \d y
	\\
	&\leq 5K \Big( \frac{k}{n} \Big)^\xi \int_1^{n/k} \int_1^{n/k} \frac{1}{xy} \d x \d y
	\\
	&\lesssim \Big( \frac{k}{n} \Big)^\xi (\log(n/k))^2.
\end{align*}

\subsection{The stochastic error terms \texorpdfstring{$A$}{A}}
\label{sec:A}

It remains to bound the stochastic error terms $A_\cdot^\cdot$, which entirely depend on the processes $\hat R_J - R_{J, n}$. Recall how, for $\bx_J \in [0, \infty)^{|J|}$, we define $\hat\bx_J$ in \cref{eq:def}. Consider further the relation $\hat R_J(\bx_J) = \hat R_J^0(\hat\bx_J)$. We shall rely on the decomposition
\begin{align}
	\hat R_J(\bx_J) - R_{J, n}(\bx_J) &= \big( \hat R_J^0(\hat\bx_J) - R_{J, n}(\hat\bx_J) \big) + \big( R_{J, n}(\hat \bx_J) - R_{J, n}(\bx_J) \big) \notag
	\\
	&= \big( G_{J, n}(\hat\bx_J) - G_{J, n}(\bx_J) \big) + G_{J, n}(\bx_J) + \big( R_{J, n}(\hat \bx_J) - R_{J, n}(\bx_J) \big), \label{eq:decomp}
\end{align}
where
\begin{equation} \label{eq:G}
	G_{J, n} := \hat R_J^0 - R_{J, n}.
\end{equation}
Accordingly, each $A_\cdot^\cdot$ term is further decomposed into three integrals $A_{\cdot, 1}^\cdot$, $A_{\cdot, 2}^\cdot$ and $A_{\cdot, 3}^\cdot$. For instance,
\begin{align}
	A_{ij}^{(m), --} &= A_{ij, 1}^{(m), --} + A_{ij, 2}^{(m), --} + A_{ij, 3}^{(m), --} \notag
	\\
	&:= \int_{a}^1 \int_{a}^1 \frac{G_{ijm, n}(u_n^{(i)}(x), u_n^{(j)}(y), u_n^{(m)}(1)) - G_{ijm, n}(x, y, 1)}{xy} \d x \d y \notag
	\\
	&\quad + \int_{a}^1 \int_{a}^1 \frac{G_{ijm, n}(x, y, 1)}{xy} \d x \d y \notag
	\\
	&\quad + \int_{a}^1 \int_{a}^1 \frac{R_{ijm, n}(u_n^{(i)}(x), u_n^{(j)}(y), u_n^{(m)}(1)) - R_{ijm, n}(x, y, 1)}{xy} \d x \d y. \label{eq:decompint}
\end{align}

The first of the three terms in \cref{eq:decomp} is proportional to a standard empirical process evaluated at a set corresponding to the difference between $\bx_J$ and $\hat\bx_J$. It will be uniformly bounded by using well known concentration inequalities for empirical processes appearing in \cite{K06} and \cite{mass2000}. The second term, $G_{J, n}$, is now a rescaled sum of $n$ independent and identically distributed (iid) processes which, when integrated as in \cref{eq:decompint}, becomes a sum of iid, bounded random variables. We will be able to control this sum via Bernstein's inequality. Finally, the third term relates to the difference between $\bx_J$ and $\hat\bx_J$. It can be controlled by weighted approximation results on the uniform quantile processes given in \cref{lemm:unappr,lemm:wn}. Two bounds are then proved on the corresponding term in \cref{eq:decompint}, the stronger of which only holds under \cref{assum:r}. This gives rise to the two desired results.

\subsubsection{Technical preliminaries on uniform quantile processes} \label{sec:unifqp}

Before tackling each term in \cref{eq:decompint}, we prove a few properties of the rescaled quantile functions $u_n^{(i)}$ which will be used throughout. In \cref{lemm:unappr}, we first prove an approximation of $u_n^{(i)}$ by a certain Gaussian process. We then establish various properties of $u_n^{(i)}$ and its approximation in \cref{coro:unweight,lemm:wn}.

\begin{lemm} \label{lemm:unappr}
For any fixed $i \in V$, define the random function $u_n^{(i)}$ as in \cref{eq:defuni} and let $0 < \nu < 1/2$. There exist random functions $w_n^{(i)}$ defined on a possibly enriched probability space and universal constants $A,B,C \in (0,\infty)$ such that for any $z>0$,
\[
\Prob\Big(\sup_{x \in [0,n/k]} |u_n^{(i)}(x)-w_n^{(i)}(x)| > k^{-1}(A\log n + z)\Big) \leq B e^{-Cz}.
\]
Moreover, the functions $w_n^{(i)}$, along with constants $\tilde A,\tilde B,\tilde C$ possibly depending on $\nu$, can be chosen such that for any $z > 0$,
\[
\Prob\Big(\max\Big\{ \sup_{0 \leq x \leq 1} \frac{|w_n^{(i)}(x) - x|}{x^\nu}, \sup_{1 \leq x \leq n/k} \frac{|w_n^{(i)}(x) - x|}{x^{1-\nu}} \Big\} > k^{-1/2}(\tilde A+ z)\Big) \leq \tilde B e^{-\tilde Cz^2}.
\]
\end{lemm}

\textit{Proof.} By definition, the quantile function $\hat F_i^-$ appearing in $u_n^{(i)}(x)$ is the right-continuous function
\[
	\hat F_i^-(x) = U_{ni, \lfloor nx \rfloor},
\]
where $U_{ni, j}$ is the $j$th order statistic from the sample $U_{1i}, \dots, U_{ni}$. We use the convention $U_{ni, 0} = 0$. Similarly define the left-continuous quantile function
\[
	\hat F_i^+(x) = U_{ni, \lceil nx \rceil}.
\]
Theorem 1 from \cite{CR1978} states that for every $z>0$,
\[
	\Prob\Big( \sup_{x \in [0,1]} |(\widehat F_i^+(x) -x)- n^{-1/2}B_n(x)| > n^{-1}(A^+ \log n + z) \Big) \leq B^+ e^{-C^+ z},
\]
for positive constants $A^+, B^+, C^+$ and a sequence of Brownian bridges $B_n$. We first establish a similar tail bound for the right-continuous quantile function. Using $\lfloor y \rfloor \geq \lceil y \rceil - 1 = \lceil y-1 \rceil$, note that
\[
	\hat F_i^+(x) \geq \hat F_i^-(x) \geq U_{ni, \lceil nx-1 \rceil} = \hat F_i^+(x - 1/n),
\]
so for every $x \in [0, 1]$, using the convention $\hat F_i^+(x) = \hat F_i^+(0)$ if $x<0$,
\begin{align*}
	0 &\leq \hat F_i^+(x) - \hat F_i^-(x)
    \\
    &\leq \hat F_i^+(x) - \hat F_i^+(x-1/n)
	\\
	&\leq 2\sup_{x \in [0,1]} |(\widehat F_i^+(x) -x)- n^{-1/2}B_n(x)| 
    \\
    & \quad\quad+ \Big|\Big( x + n^{-1/2}B_n(x) \Big) - \Big( x - \frac{1}{n} + n^{-1/2}B_n(x-1/n) \Big)\Big|
	\\
	&\leq 2\sup_{x \in [0,1]} |(\widehat F_i^+(x) -x)- n^{-1/2}B_n(x)| + \frac{1}{n} + n^{-1/2} |B_n(x) - B_n(x-1/n)|.
\end{align*}
Thus
\begin{align*}
&\sup_{x \in [0,1]} |(\widehat F_i^-(x) -x)- n^{-1/2}B_n(x)| 
\\
&\leq 3\sup_{x \in [0,1]} |(\widehat F_i^+(x) -x)- n^{-1/2}B_n(x)| + \frac{1}{n} + n^{-1/2} |B_n(x) - B_n(x-1/n)|.
\end{align*}
Using the covariance function of the standard Brownian bridge, $n^{-1/2} (B_n(x) - B_n(x-1/n))$ is normally distributed with mean 0 and variance upper bounded by $4/n^2$. Thus, its absolute value is upper bounded by $\sqrt{z}/n$ with probability greater than $1 - e^{-z/4}$. Therefore,
\[
	\Prob\Big( \sup_{x \in [0,1]} |(\widehat F_i^-(x) -x)- n^{-1/2}B_n(x)| > n^{-1}(3A^+ \log n + 3z + 1 + \sqrt{z}) \Big) \leq B^+ e^{-C^+ z} + e^{-z/4}
\]
which implies, for the right choice of $A, B, C$,
\begin{equation} \label{eq:rcquantile}
	\Prob\Big( \sup_{x \in [0,1]} |(\widehat F_i^-(x) -x)- n^{-1/2}B_n(x)| > n^{-1}(A \log n + z) \Big) \leq B e^{-C z}.
\end{equation}

\color{black}

Now define
\[
w_n^{(i)}(x) := x + n^{-1/2}  \frac{n}{k} B_n(kx/n).
\]
Then 
\begin{align*}
u_n^{(i)}(x)-w_n^{(i)}(x) &= \frac{n}{k}\widehat F_i^-(kx/n) - x - n^{-1/2} \frac{n}{k} B_n(kx/n)
\\
&= \frac{n}{k}\Big(\widehat F_i^-(kx/n) - \frac{kx}{n} - n^{-1/2}  B_n(kx/n)  \Big)
\end{align*}
Thus for $w_n^{(i)}$ defined above
\begin{align*}
&\Prob\Big(\sup_{x \in [0,n/k]} |u_n^{(i)}(x)-w_n^{(i)}(x)| > k^{-1}(A\log n + z)\Big)
\\
=~ &\Prob\Big(\sup_{x \in [0,n/k]} \frac{n}{k}\Big|\widehat F_i^-(kx/n) - \frac{kx}{n} - n^{-1/2}  B_n(kx/n)  \Big| > k^{-1}(A\log n + z)\Big)
\\
=~ &\Prob\Big(\sup_{x \in [0,1]} \Big|\widehat F_i^-(x) - x - n^{-1/2}  B_n(x)  \Big| > n^{-1}(A\log n + z)\Big)
\\
\leq~ & Be^{-Cz}
\end{align*}
by \cref{eq:rcquantile}, which proves the first claim.

We now prove the second claim. Let 
\[
Z_n(x) := n^{-1/2} \frac{n}{k} B_n(kx/n).
\]
Observe that 
\begin{align*}
\{Z_n(x)\}_{x \in [0,n/k]} &\stackrel{\mathcal{D}}{=} \Big\{n^{-1/2} \frac{n}{k} W_n(kx/n) - n^{-1/2} x W_n(1) \Big\}_{x \in [0,n/k]}
\\
&\stackrel{\mathcal{D}}{=} \Big\{k^{-1/2} W_n(x) - k^{-1/2} kx/n W_n(n/k) \Big\}_{x \in [0,n/k]},
\end{align*}
where $W_n$ are standard Wiener processes on $[0,\infty)$. If the sequences $\sup_{0 < x \leq 1} k^{1/2}|Z_n(x)|/x^\nu$ and $\sup_{1 \leq x \leq n/k} k^{1/2}|Z_n(x)|/x^{1-\nu}$ are uniformly tight, their distributions have finite medians independent of $n$. Hence by Proposition A.2.1 of \cite{VW1996}, there exist constants $\tilde A,\tilde B,\tilde C$, {depending only on those medians,} such that
\[
\Prob\Big( \max\Big\{ \sup_{0 < x \leq 1} \frac{k^{1/2}|Z_n(x)|}{x^\nu}, \sup_{1 \leq x \leq \frac{n}{k}} \frac{k^{1/2}|Z_n(x)|}{x^{1-\nu}} \Big\} > \tilde A + z \Big) \leq \tilde B e^{-\tilde C z^2}.
\]
and the result follows.

To establish tightness, note that since $\nu<1/2$,
\[
\sup_{0 < x \leq 1} \frac{k^{1/2}|Z_n(x)|}{x^\nu} \leq \sup_{0 < x \leq 1} \frac{|W_n(x)|}{x^\nu} + \frac{k}{n} |W_n(n/k)| = \Op{1 + \sqrt{\frac{k}{n}}} = \Op{1}
\]
and
\begin{align*}
\sup_{1 \leq x \leq \frac{n}{k}} \frac{k^{1/2}|Z_n(x)|}{x^{1-\nu}} &\leq \sup_{1 \leq x < \infty} \frac{|W_n(x)|}{x^{1-\nu}} + \Big( \frac{k}{n} \Big)^{1-\nu} |W_n(n/k)|
\\
&= \Op{1 + \Big( \frac{k}{n} \Big)^{1/2 - \nu} \sqrt{\log\log\frac{n}{k}}} = \Op{1},
\end{align*}
where we used the law of the iterated logarithm at 0 and at $\infty$, respectively, for Wiener processes \citep[see for instance][Section 8.11]{Durrett4.1}. Note that the probability bounds on the suprema above, and hence the bounds on their medians, are uniform in $n$ but may depend on $\nu$, hence the dependence of the constants $\tilde A,\tilde B,\tilde C$ on this parameter.
\hfill $\Box$

Let $w_n^{(i)}$ be as in the proof of \cref{lemm:unappr}. For $a \in (0, 1)$ and $\nu_1, \nu_2 \in [0, 1]$, let
\[
	\tilde\Delta_n^{(i)}(a, \nu_1, \nu_2) = \max\bigg\{ \sup_{a \leq x \leq 1} \frac{|w_n^{(i)}(x) - x|}{x^{\nu_1}}, \sup_{1 \leq x \leq n/k} \frac{|w_n^{(i)}(x) - x|}{x^{1-\nu_2}} \bigg\}
\]
and
\[
	\tilde\Delta_n^{(i)}(a, \nu) = \tilde\Delta_n^{(i)}(a, \nu, \nu).
\]
Similarly, let
\[
	\hat\Delta_n^{(i)}(a, \nu_1, \nu_2) = \max\bigg\{ \sup_{a \leq x \leq 1} \frac{|u_n^{(i)}(x) - x|}{x^{\nu_1}}, \sup_{1 \leq x \leq n/k} \frac{|u_n^{(i)}(x) - x|}{x^{1-\nu_2}} \bigg\}
\]
and
\[
	\hat\Delta_n^{(i)}(a, \nu) = \hat\Delta_n^{(i)}(a, \nu, \nu).
\]

It is established in \cref{lemm:unappr} that there are constants $\tilde A,\tilde B,\tilde C$ only depending on $\nu \in (0, 1/2)$ such that
\[
	\Prob\Big( \tilde\Delta_n^{(i)}(0, \nu) > k^{-1/2}(\tilde A+ z) \Big) \leq \tilde B e^{-\tilde Cz^2}.
\]
\Cref{lemm:unappr} also allows us to obtain a certain bound for the terms $\hat\Delta_n^{(i)}$.

\begin{coro}
\label{coro:unweight}
	
	Let $\hat\Delta_n^{(i)}$ be as above. There exist constants $\hat A, \hat B, \hat C$ depending only on $\nu \in (1/2)$ such that for all $z \geq 0$,
	\[
		\Prob\bigg( \hat\Delta_n^{(i)}(0, 0, \nu) > \hat A \Big( \frac{1}{\sqrt{k}} + \frac{\log n}{k} \Big) + \sqrt{\frac{z}{k}} + \frac{z}{k} \bigg) \leq \hat B e^{-\hat C z}.
	\]
	
\end{coro}

\begin{proof}
	
	Write
	\begin{align*}
		\hat\Delta_n^{(i)}(0, 0, \nu) &\leq \tilde\Delta_n^{(i)}(0, 0, \nu) + \big| \hat\Delta_n^{(i)}(0, 0, \nu) - \tilde\Delta_n^{(i)}(0, 0, \nu) \big|
		\\
		&\leq \tilde\Delta_n^{(i)}(0, \nu) + \max\bigg\{ \sup_{0 \leq x \leq 1} |u_n^{(i)}(x) - w_n^{(i)}(x)|, \sup_{1 \leq x \leq n/k} \frac{|u_n^{(i)}(x) - w_n^{(i)}(x)|}{x^{1-\nu}} \bigg\}
		\\
		&\leq \tilde\Delta_n^{(i)}(0, \nu) + \sup_{0 \leq x \leq n/k} |u_n^{(i)}(x) - w_n^{(i)}(x)|,
	\end{align*}
	where the second inequality follows form the fact that for two functions $f$ and $g$ defined on the same domain, $|\sup_x f(x) - \sup_x g(x)| \leq \sup_x |f(x) - g(x)|$.
	
	By \cref{lemm:unappr}, the first term above is larger than $(\tilde A + \sqrt{z})/\sqrt{k}$ with probability at most $\tilde B e^{-\tilde C z}$ and the second one is larger than $(A\log n + z)/k$ with probability at most $B e^{-Cz}$. The result follows by the right choice of $\hat A, \hat B, \hat C$.
\end{proof}

\begin{lemm}
\label{lemm:wn}
	
	With $w_n^{(i)}$ as above, there exists a universal positive constant $c'$ such that for all $a \in (0, 1)$,
	\[
		\Prob\bigg( \sup_{a \leq x \leq n/k} \frac{|w_n^{(i)}(x) - x|}{x} > 1/2 \bigg) \leq 6\exp\Big\{ -c' k \Big( 1 \wedge \frac{a}{\log\log(1/a)} \Big) \Big\}.
	\]
	
\end{lemm}

\begin{proof}
	
	From the proof of \cref{lemm:unappr}, there is a standard Brownian motion $W$ such that $\{w_n^{(i)}(x) - x\}_{x \in [a, n/k]}$ is equal in distribution to the zero-mean Gaussian process
	\[
		\{Z_n(x)\}_{x \in [a, n/k]} := \Big\{W(x)/\sqrt{k} - \sqrt{k}x/n W(n/k) \Big\}_{x \in [a, n/k]}.
	\]
	Assume first that $a \leq e^{-2}$. We are therefore interested in
	\begin{align}
		&\sup_{x \in [a, n/k]} \frac{|Z_n(x)|}{x}\notag
        \\
        &\leq \frac{1}{\sqrt{k}} \bigg( \sup_{x \in [a, n/k]} \frac{|W(x)|}{x} + \frac{k}{n} W(n/k) \bigg) \notag
		\\
		&\leq \frac{1}{\sqrt{k}} \bigg( \sup_{x \in [a, e^{-2}]} \sqrt{\frac{\log\log(1/x)}{x}} \frac{|W(x)|}{\sqrt{x \log\log(1/x)}} + \sup_{x \in [e^{-2}, n/k]} \frac{|W(x)|}{x} + \frac{k}{n} W(n/k) \bigg) \notag
		\\
		&\leq \frac{1}{\sqrt{k}} \bigg( \sqrt{\frac{\log\log(1/a)}{a}} \sup_{x \in [0, e^{-2}]} \frac{|W(x)|}{\sqrt{x \log\log(1/x)}} + \sup_{x \in [e^{-2}, \infty)} \frac{|W(x)|}{x} + \frac{k}{n} W(n/k) \bigg). \label{eq:GPs}
	\end{align}
	By the laws of the iterated logarithm at 0 and at infinity, respectively, the above two suprema of Gaussian processes are tight random variables. It follows that they have finite medians and hence, by Proposition A.2.1 of \cite{VW1996}, that they have sub-Gaussian tails. The same can be said of the uniformly (in $n$) tight random variable $\tfrac{k}{n} W(n/k)$. Therefore,
	\begin{align*}
		\Prob\bigg( \sup_{x \in [a, n/k]} \frac{|Z_n(x)|}{x} > 1/2 \bigg) &\leq \Prob\bigg( \sup_{x \in [0, e^{-2}]} \frac{|W(x)|}{\sqrt{x \log\log(1/x)}} > \frac{\sqrt{ka}}{6\sqrt{\log\log(1/a)}} \bigg)
		\\
		&\quad + \Prob\bigg( \sup_{x \in [e^{-2}, \infty)} \frac{|W(x)|}{x} > \frac{\sqrt{k}}{6} \bigg) + \Prob\bigg( \frac{k}{n} W(n/k) > \frac{\sqrt{k}}{6} \bigg)
		\\
		&\leq 2\exp\Big\{ -c_1 \frac{ka}{\log\log(1/a)} \Big\} + 2e^{-c_2 k} + 2e^{-c_3 k},
	\end{align*}
	for some universal positive constants $c_1$, $c_2$, $c_3$. The result follows for some $c'$ depending on those three constants only.
	
	If, instead, $a > e^{-2}$, then the sum of the last two terms of \cref{eq:GPs} is a valid bound in itself. The rest of the proof goes through and we obtain the bound $4e^{-c' k}$.
\end{proof}

In particular, \cref{lemm:wn} lower bounds the probability that for all $x \in [a, n/k]$, $x/2 \leq w_n^{(i)}(x) \leq 2x$, which will be repeatedly used in \cref{sec:copula}.

The following sections are respectively dedicated to each of the three terms of the decomposition introduced in \cref{eq:decomp}.

\subsubsection{Increments of empirical processes}
\label{sec:ep}

We first consider the terms $A_{\cdot, 1}^\cdot$. In this section we prove that for any $\nu \in (0, 1/2)$ there exists a constant $C_1 < \infty$ (which can also depend on the constant $K$ from \cref{assum:tail}) such that for any $\eps \leq 1$,
\begin{align*}
&\Prob\Big( \max \big| A_{\cdot, 1}^{\cdot} \big| > C_1 (\log(n/k) + \log(1/a))^2 \Big\{ \Big(\frac{\log(n/k)}{k}\Big)^{1/2} \Big((k/n)^\xi + \eps\Big)^{1/2}
\\
& \quad \quad \quad + \frac{\log(n/k)}{k} + \frac{\lambda}{\sqrt{k}} \Big((k/n)^\xi + \eps\Big)^{1/2} + \frac{\lambda^2}{k} \Big\} \Big)
\\
&\leq d^3 e^{-\lambda^2} + \Prob\Big(\max_{i \in V} \hat\Delta_n^{(i)}(a, 0, \nu) > \eps \Big).
\end{align*} 
Consider the following decompositions. For all $x, y \in [a, 1]$, the numerator in the integral $A_{ij, 1}^{(m), --}$ satisfies
\begin{align*}
	&\big| G_{ijm, n}(u_n^{(i)}(x), u_n^{(j)}(y), u_n^{(m)}(1)) - G_{ijm, n}(x, y, 1) \big|
	\\
	&\quad \leq \big| G_{ijm, n}(u_n^{(i)}(x), u_n^{(j)}(y), [1 \wedge u_n^{(m)}(1), 1 \vee u_n^{(m)}(1)]) \big| 
    \\
    &\quad\quad+ \big| G_{ijm, n}(u_n^{(i)}(x), [y \wedge u_n^{(j)}(y), y \vee u_n^{(j)}(y)], 1) \big|
	\\
	&\quad\quad + \big| G_{ijm, n}([x \wedge u_n^{(i)}(x), x \vee u_n^{(i)}(x)], y, 1) \big|.
\end{align*}
The numerators in $A_{im, 1}^{(m), --}$ and $A_{i, 1}^{(m), \ell, -}$ satisfy a similar bound with only the first two terms, up to a logarithmic factor that is everywhere bounded by {$\log(1/a)$} in the case of $A_{i, 1}^{(m), 2, -}$.

For all $x \in [1, n/k]$, $y \in [a, 1]$, the numerator in the integral $A_{ij, 1}^{(m), +-}$ satisfies
\begin{align*}
	&\big| G_{ijm, n}([u_n^{(i)}(x), \infty), u_n^{(j)}(y), u_n^{(m)}(1)) - G_{ijm, n}([x, \infty), y, 1) \big|
	\\
	&\quad \leq \big| G_{ijm, n}([u_n^{(i)}(x), \infty), u_n^{(j)}(y), [1 \wedge u_n^{(m)}(1), 1 \vee u_n^{(m)}(1)]) \big|
	\\
	&\quad\quad + \big| G_{ijm, n}([u_n^{(i)}(x), \infty), [y \wedge u_n^{(j)}(y), y \vee u_n^{(j)}(y)], 1) \big|
    \\
	&\quad\quad + \big| G_{ijm, n}([x \wedge u_n^{(i)}(x), x \vee u_n^{(i)}(x)], y, 1) \big|.
\end{align*}
The numerators in $A_{im, 1}^{(m), +-}$ and $A_{i, 1}^{(m), \ell, +}$ satisfy a similar bound with only the first two terms, up to a logarithmic factor that is everywhere bounded by {$\log(n/k)$} in the case of $A_{i, 1}^{(m), 2, +}$, as well as the numerators in $A_{ij, 1}^{(m), -+}$ by symmetry.

For all $x, y \in [1, n/k]$, the numerator in the integral $A_{ij, 1}^{(m), ++}$ satisfies
\begin{align*}
	&\big| G_{ijm, n}([u_n^{(i)}(x), \infty), [u_n^{(j)}(y), \infty), u_n^{(m)}(1)) - G_{ijm, n}([x, \infty), [y, \infty), 1) \big|
	\\
	&\quad \leq \big| G_{ijm, n}([u_n^{(i)}(x), \infty), [u_n^{(j)}(y), \infty), [1 \wedge u_n^{(m)}(1), 1 \vee u_n^{(m)}(1)]) \big|
	\\
	&\quad\quad + \big| G_{ijm, n}([u_n^{(i)}(x), \infty), [y \wedge u_n^{(j)}(y), y \vee u_n^{(j)}(y)], 1) \big|
\\
	&\quad\quad
    + \big| G_{ijm, n}([x \wedge u_n^{(i)}(x), x \vee u_n^{(i)}(x)], [y, \infty), 1) \big|.
\end{align*}

Define, for any $\eps \in (0, 1]$, $\cF(\eps) := \cup_{i, j, m} \cF_{ijm}(\eps)$ where
\[
\cF_{ijm}(\eps) := \Big\{ \frac{n}{k} \One_{\frac{k}{n} S}: S \in \cS_{ijm}^- \cup \cS_{ijm}^+ \Big\},
\]
and where the classes of sets $\cS_{ijm}^-$ and $\cS_{ijm}^+$ are defined as
\begin{align*}
\cS_{ijm}^- := \big\{ \{w \in [0, \infty)^d: x - \eps \leq w_i \leq x + \eps, &\ a_j \leq w_j \leq b_j, a_m \leq w_m \leq b_m\}:
\\
&a \leq x \leq 1, a_j, b_j, a_m, b_m \in [0, \infty] \big\},
\end{align*}
and
\begin{align*}
\cS_{ijm}^+ := \big\{ \{w \in [0, \infty)^d: x - x^{1-\nu} \eps \leq w_i \leq x + x^{1-\nu} \eps, &\ a_j \leq w_j \leq b_j, 0 \leq w_m \leq 1\}:
\\
&1 \leq x \leq n/k, a_j, b_j, \in [0, \infty] \big\}.
\end{align*}
Recalling the definition of $\hat\Delta_n^{(i)}(a, 0, \nu)$, it follows from the definition of the class $\cF(\eps)$ that whenever
\[
\hat\Delta_n^{(i)}(a, 0, \nu) = \max_{i \in V} \max\Big\{ \sup_{a \leq x \leq 1} |u_n^{(i)}(x) - x|, \sup_{1 \leq x \leq n/k} \frac{|u_n^{(i)}(x) - x|}{x^{1-\nu}} \Big\} \leq \eps,
\]
the numerator inside any of the integrals $A_{\cdot, 1}^\cdot$ can be expressed as a sum of at most three terms of the form $(P_nf_1 - Pf_1) + (P_nf_2 - Pf_2) + (P_nf_3 - Pf_3)$, for functions $f_1, f_2, f_3 \in \cF(\eps)$. Here, $P_n$ is the empirical distribution of the random vectors $\bU_1, \dots, \bU_n$ whereas $P$ is their true distribution. In the case of the terms $A_{i, 1}^{(m), 2, \pm}$, the sum is multiplied by a logarithmic term. In this case, all the integrals $A_{\cdot, 1}^\cdot$ are upper bounded, in absolute value, by
\[
{3 (\log(n/k) + \log(1/a))^2 \sup_{f \in \cF(\eps)} |P_nf - Pf|.}
\]
What we have established so far is that each such term $A_{\cdot, 1}^\cdot$ satisfies, for any $t>0$,
\begin{align*}
\Prob\Big(|A_{\cdot, 1}^\cdot| \geq t\Big) \leq~& \Prob\Big( 3 (\log(n/k) + \log(1/a))^2 \sup_{f \in \cF(\eps)} |P_n f - Pf| \geq t \Big)
\\
& + \Prob\Big(\max_{i \in V}  \max\Big\{ \sup_{a \leq x \leq 1} |u_n^{(i)}(x) - x|, \sup_{1 \leq x \leq n/k} \frac{|u_n^{(i)}(x) - x|}{x^{1-\nu}} \Big\} > \eps \Big). 
\end{align*}

For any triple $(i, j, m)$, $\cF_{ijm}(\eps)$ clearly admits the constant envelope function of the form $n/k$. Moreover it is a VC-subgraph class that satisfies \cref{eq:entr} with universal constants $A$ and $V$ (see for instance \cite{VW1996}, Theorem 2.6.7). Furthermore, the variance of any single function $f$ in $\cF_{ijm}(\eps)$ is bounded by
\begin{align*}
	Pf^2
    &\leq \Big( \frac{n}{k} \Big)^2 \bigg\{ \sup_{a \leq x \leq 1} \Prob\Big( U_i \in \frac{k}{n} [x - \eps, x + \eps] \Big)
    \\
    &\quad\quad\quad\quad \quad\vee \sup_{1 \leq x \leq n/k} \Prob\Big( U_i \in \frac{k}{n} [x - x^{1-\nu}\eps, x + x^{1-\nu}\eps], U_m \leq \frac{k}{n} \Big) \bigg\}
	\\
	&\leq \frac{n}{k} \bigg\{ 2\eps \vee \sup_{1 \leq x \leq n/k} R_{im, n}([x - x^{1-\nu}\eps, x + x^{1-\nu}\eps], 1) \bigg\}
	\\
	&\leq \frac{n}{k} \bigg\{ 2\eps + \sup_{1 \leq x \leq n/k} R_{im}([x - x^{1-\nu}\eps, x + x^{1-\nu}\eps], 1) + 2K\Big( \frac{k}{n} \Big)^\xi \bigg\}
	\\
	&\lesssim \frac{n}{k} \bigg\{ \eps + \Big( \frac{k}{n} \Big)^\xi \bigg\}
\end{align*}
where the last two inequalities follow from \cref{eq:B} and \cref{lemm:boundR}, respectively. By \cref{eq:Gexpect} we therefore have
\begin{align*}
	\E\Big[ \sup_{f \in \cF_{ijm}(\eps)} |P_n f - Pf| \Big] &\lesssim k^{-1/2} \Big((k/n)^\xi + \eps\Big)^{1/2} \log\Big((n/k)^{1/2}\Big((k/n)^\xi + \eps\Big)^{-1/2}\Big)^{1/2}
	\\
	&\quad + k^{-1} \log\Big((n/k)^{1/2}\Big((k/n)^\xi + \eps\Big)^{-1/2}\Big)
	\\
	&\lesssim \Big( \frac{\log(n/k)}{k} \Big)^{1/2} \Big((k/n)^\xi + \eps\Big)^{1/2} + \frac{\log(n/k)}{k}.
\end{align*}
It follows from \cref{eq:Gcontract} that there exists a constant $c$ such that for each triple $(i, j, m)$ and each $\lambda > 0$,
\begin{multline*}
	\Prob\Big( \sup_{f \in \cF_{ijm}(\eps)} |P_n f - Pf| \geq c\Big\{ \Big(\frac{\log(n/k)}{k}\Big)^{1/2} \Big((k/n)^\xi + \eps\Big)^{1/2}
    \\
    + \frac{\log(n/k)}{k} + \frac{\lambda}{\sqrt{k}} {\Big((k/n)^\xi + \eps\Big)^{1/2}} + \frac{\lambda^2}{k} \Big\} \Big) \leq e^{-\lambda^2}.
\end{multline*}
Combined with the union bound, this completes the proof.

\subsubsection{Sums of iid processes}
\label{sec:iid}

We now deal with the terms $A_{\cdot, 2}^\cdot$ involving integrated empirical processes, such as in \cref{eq:decompint}. In this section, we show that there exists a positive constant $c_2$ such that for all $\lambda>0$,
\begin{multline*}
	\Prob\bigg( \max\big| A_{\cdot, 2}^\cdot \big| > \Big( 1 + \Big(\frac{k}{n}\Big)^\xi (\log(n/k))^2 (\log(n/k) + \log(1/a))^2 \Big)^{1/2} \frac{\lambda}{\sqrt{k}}
    \\
    + (\log(n/k) + \log(1/a))^2 \frac{\lambda^2}{k} \bigg) \leq 16d^3 e^{-c_2\lambda^2 / 2}.
\end{multline*}

Starting with $A_{ij, 2}^{(m), --}$, we have by definition of $G_{ijm, n}$
\begin{align*}
	A_{ij, 2}^{(m), --}
    = \int_a^1 \int_a^1 \frac{G_{ijm, n}(x, y, 1)}{xy} dxdy
	= \sum_{t=1}^n \big( V_{t, ijm}^{(m), --} - \E\big[ V_{t, ijm}^{(m), --} \big] \big),
\end{align*}
where $V_{t, ijm}^{(m), --}$, $1 \leq t \leq n$, are independent copies of the random variable
\begin{align*}
	V_{ijm}^{(m), --} &:= \frac{1}{k} \int_a^1 \int_a^1 \frac{ \Ind{U_i \leq \frac{k}{n} x, U_j \leq \frac{k}{n} y, U_m \leq \frac{k}{n}} }{xy} dxdy\\
	&= \frac{1}{k} \log\Big( \frac{k}{nU_i} \wedge a^{-1} \Big) \log\Big( \frac{k}{nU_j} \wedge a^{-1} \Big) \Ind{U_i \leq \frac{k}{n}, U_j \leq \frac{k}{n}, U_m \leq \frac{k}{n}}.
\end{align*}
Recall that by assumption $a<1$. We may then write
\[
	V_{ijm}^{(m), --} = \frac{1}{k} \log(W_i) \log(W_j) \Ind{W_i, W_j > 1},
\]
with
\[
	W_i := \Big( \frac{k}{nU_i} \wedge a^{-1} \Big)^{\Ind{U_m \leq \frac{k}{n}}}
\]
and $W_j$ defined the same way. We easily notice that $0 \leq V_{ijm}^{(m), --} \leq (\log(1/a))^2 / k$. Moreover, an application of \cref{lemm:Elog} (particularly \cref{eq:moment++}) gives
\begin{align*}
	\var\big( V_{ijm}^{(m), --} \big) &\leq \E\big[ \big( V_{ijm}^{(m), --} \big)^2 \big]
	\\
	&= \frac{4}{k^2} \int_a^1 \int_a^1 \frac{\frac{k}{n} R_{ijm, n}(x, y, 1) |(\log x)(\log y)|}{xy} dxdy
	\\
	&\leq \frac{4}{kn} \int_0^1 \int_0^1 \frac{|(\log x)(\log y)|}{\sqrt{xy}} dxdy
	\\
	&= \frac{64}{kn},
\end{align*}
where we used once again that $R_{ijm, n}(x, y, 1) \leq x \wedge y \leq \sqrt{xy}$, along with the formula $\int_0^1 \log(x)/\sqrt{x} \ dx = -4$. We may therefore apply Bernstein's inequality for bounded random variables \citep[][Lemma 2.2.9]{VW1996} with $v = 64/k$, $M = (\log(1/a))^2/k$, which yields
\[
	\Prob\big( \big| A_{ij, 2}^{(m), --} \big| > \lambda \big) \leq 2\exp\Big\{ -\frac{k \lambda^2}{2(64 + \lambda (\log(1/a))^2/3)} \Big\}.
\]

Now considering $A_{ij, 2}^{(m), +-}$, we use the same approach and see that
\[
	A_{ij, 2}^{(m), +-} = \sum_{t=1}^n \big( V_{t, ijm}^{(m), +-} - \E\big[ V_{t, ijm}^{(m), +-} \big] \big),
\]
where
\[
	V_{ijm}^{(m), +-} = -\frac{1}{k} \log(W_i) \log(W_j) \Ind{W_i < 1, W_j > 1}
\]
and $W_i$, $W_j$ are as before. This time, $0 \leq V_{ijm}^{(m), +-} \leq (\log(n/k))\log(1/a) / k $. An application of \cref{lemm:Elog} (this time, \cref{eq:moment+-}) gives
\begin{align*}
	\var\big( V_{ijm}^{(m), +-} \big) &\leq \E\big[ \big( V_{ijm}^{(m), +-} \big)^2 \big]
	\\
	&= \frac{4}{k^2} \int_a^1 \int_1^{n/k} \frac{\frac{k}{n} R_{ijm, n}([x, \infty), y, 1) |(\log x)(\log y)|}{xy} dxdy
	\\
	&\leq \frac{4}{kn} \int_a^1 \int_1^{n/k} \frac{\big( R_{ijm}([x, \infty), y, 1) + 3K (k/n)^\xi \big) |(\log x)(\log y)|}{xy} dxdy,
\end{align*}
by \cref{eq:B,eq:B3}. By \cref{eq:T3}, $R_{ijm}([x, \infty), y, 1) \leq R_{im}([x, \infty), 1) \wedge R_{jm}(y, 1) \leq K x^{-\xi} \wedge y \leq K x^{-\xi/2} y^{1/2}$. The integral above is thus bounded by
\begin{align*}
	&K \int_0^1 \int_1^\infty \frac{(\log x)(-\log y)}{x^{1+\xi/2} y^{1/2}} dxdy + 3K \Big( \frac{k}{n} \Big)^\xi \int_a^1 \int_1^{n/k} \frac{(\log x)(-\log y)}{xy} dxdy
	\\
	&\quad \leq \frac{16K}{\xi^2} + 3K \Big( \frac{k}{n} \Big)^\xi (\log(n/k))^2 (\log(1/a))^2 \leq C_2 \Big( 1 + \Big( \frac{k}{n} \Big)^\xi (\log(n/k))^2 (\log(1/a))^2 \Big),
\end{align*}
for a suitably chosen constant $C_2$ depending on $K$ and $\xi$ only. Bernstein's inequality, with
\[
	v = \frac{4C_2}{k} \Big( 1 + \Big( \frac{k}{n} \Big)^\xi (\log(n/k))^2 (\log(1/a))^2 \Big)
\]
and $M = (\log(n/k))\log(1/a)/k$, therefore implies that for a positive constant $c_2$ depending on $C_2$ only,
\[
	\Prob\big( \big| A_{ij, 2}^{(m), +-} \big| > \lambda \big) \leq 2\exp\Big\{ -c_2 \frac{k \lambda^2}{1 + (k/n)^\xi (\log(n/k))^2 (\log(1/a))^2 + \lambda (\log(n/k))\log(1/a)} \Big\}.
\]
By symmetry, $A_{ij, 2}^{(m), -+}$ admits the same bound.

As for $A_{ij, 2}^{(m), ++}$, we write it as
\[
	A_{ij, 2}^{(m), ++} = \sum_{t=1}^n \big( V_{t, ijm}^{(m), ++} - \E\big[ V_{t, ijm}^{(m), ++} \big] \big),
\]
where
\[
	V_{ijm}^{(m), ++} = \frac{1}{k} \log(W_i) \log(W_j) \Ind{W_i, W_j < 1}
\]
and $W_i$, $W_j$ are as before. Again, $0 \leq V_{ijm}^{(m), ++} \leq (\log(n/k))^2/k$. An application of \cref{lemm:Elog} (this time, \cref{eq:moment--}) gives
\begin{align*}
	\var\big( V_{ijm}^{(m), ++} \big) &\leq \E\big[ \big( V_{ijm}^{(m), ++} \big)^2 \big]
	\\
	&= \frac{4}{k^2} \int_1^{n/k} \int_1^{n/k} \frac{\frac{k}{n} R_{ijm, n}([x, \infty), [y, \infty), 1) (\log x)(\log y)}{xy} dxdy
	\\
	&\leq \frac{4}{kn} \int_1^{n/k} \int_1^{n/k} \frac{\big( R_{ijm}([x, \infty), [y, \infty), 1) + 5K (k/n)^\xi \big) (\log x)(\log y)}{xy} dxdy,
\end{align*}
by \cref{eq:B,eq:B3}. By \cref{eq:T3}, 
\begin{align*}
R_{ijm}([x, \infty), [y, \infty), 1) \leq R_{im}([x, \infty), 1) \wedge R_{jm}([y, \infty), 1) \leq K x^{-\xi} \wedge K y^{-\xi} \leq K x^{-\xi/2} y^{-\xi/2}.
\end{align*}
The integral above is thus bounded by
\begin{align*}
	&K \int_1^\infty \int_1^\infty \frac{(\log x)(\log y)}{(xy)^{1+\xi/2}} dxdy + 5K \Big( \frac{k}{n} \Big)^\xi \int_1^{n/k} \int_1^{n/k} \frac{(\log x)(\log y)}{xy} dxdy
	\\
	&\quad \leq \frac{16K}{\xi^4} + 5K \Big( \frac{k}{n} \Big)^\xi (\log(n/k))^4 \leq C_2 \Big( 1 + \Big( \frac{k}{n} \Big)^\xi (\log(n/k))^4 \Big),
\end{align*}
after possibly enlarging the constant $C_2$. Hence by a similar application of Bernstein's inequality for bounded random variables as before,
\[
	\Prob\big( \big| A_{ij, 2}^{(m), ++} \big| > \lambda \big) \leq 2\exp\Big\{ -c_2 \frac{k \lambda^2}{1 + (k/n)^\xi (\log(n/k))^4 + \lambda (\log(n/k))^2} \Big\},
\]
after possibly decreasing the (still positive) constant $c_2$.

Finally, the terms $A_{i, 2}^{(m), \ell, -}$, $A_{i, 2}^{(m), \ell, +}$, $A_{im, 2}^{(m), --}$ and $A_{im, 2}^{(m), +-}$ can be shown to satisfy similar tail bounds by the same strategy, using \cref{eq:moment-,eq:moment+} instead of \cref{eq:moment--,eq:moment+-,eq:moment++} for $A_{i, 2}^{(m), \ell, -}$ and $A_{i, 2}^{(m), \ell, +}$.

The conclusion of this section is that the positive constant $c_2$ can be chosen sufficiently small (only depending on $K$ and $\xi$) such that all the terms $A_{\cdot, 2}^\cdot$ satisfy
\begin{align*}
&\Prob\big( \big| A_{\cdot, 2}^\cdot \big| > \lambda \big) 
\\
&\leq 2\exp\bigg\{ -c_2 \frac{k\lambda^2}{ 1 + (k/n)^\xi (\log(n/k))^2 (\log(n/k) + \log(1/a))^2 + \lambda (\log(n/k) + \log(1/a))^2 } \bigg\},
\end{align*}
for all $\lambda>0$. The denominator in the exponential above is clearly upper bounded by
\[
	2\max\Big\{ 1 + (k/n)^\xi (\log(n/k))^2 (\log(n/k) + \log(1/a))^2, \lambda (\log(n/k) + \log(1/a))^2 \Big\},
\]
so the whole exponential is upper bounded by
\begin{multline*}
\max\bigg\{ 2\exp\bigg\{ -c_2 \frac{k \lambda^2}{2(1 + (k/n)^\xi (\log(n/k))^2 (\log(n/k) + \log(1/a))^2)} \bigg\},
\\
2\exp\bigg\{ -c_2 \frac{k \lambda}{2(\log(n/k) + \log(1/a))^2} \bigg\}.
\end{multline*}
Deduce that at least one of
\[
	\Prob\bigg( \big| A_{\cdot, 2}^\cdot \big| > \Big( 1 + \Big(\frac{k}{n}\Big)^\xi (\log(n/k))^2 (\log(n/k) + \log(1/a))^2 \Big)^{1/2} \frac{\lambda}{\sqrt{k}} \bigg) \leq 2e^{-c_2\lambda^2 / 2}
\]
or
\[
	\Prob\bigg( \big| A_{\cdot, 2}^\cdot \big| > (\log(n/k) + \log(1/a))^2 \frac{\lambda^2}{k} \bigg) \leq 2e^{-c_2\lambda^2 / 2}
\]
holds. Therefore
\begin{multline*}
\Prob\bigg( \big| A_{\cdot, 2}^\cdot \big| > \Big( 1 + \Big(\frac{k}{n}\Big)^\xi (\log(n/k))^2 (\log(n/k)+ \log(1/a))^2 \Big)^{1/2} \frac{\lambda}{\sqrt{k}}
\\
+ (\log(n/k) + \log(1/a))^2 \frac{\lambda^2}{k} \bigg) \leq 2e^{-c_2\lambda^2 / 2},
\end{multline*}
and a union bound allows us to conclude.

\subsubsection{Increments of rescaled copulae}
\label{sec:copula}

It remains to bound the terms $A_{\cdot, 3}^\cdot$, corresponding to increments of the measures $R_{J, n}$ when the rescaled quantile functions $u_n^{(i)}$ are applied to their arguments. In this section, we prove that under \cref{assum:tail}, there exists a constant $C_3$ such that for any $\lambda, \tau > 0$,
\begin{align*}
	&\Prob\Big( \max \big| A_{\cdot, 3}^\cdot \big| > 3\tau (\log(n/k) + \log(1/a))^2 + C_3 (\log(n/k) + \log(1/a))^2 \Big( \frac{\tilde A + \lambda}{\sqrt{k}} + \Big(\frac{k}{n}\Big)^\xi \Big) \Big)
	\\
	&\quad\leq \Prob\Big( \max_i \tilde\Delta_n^{(i)}(a, \nu) > \frac{\tilde A + \lambda}{\sqrt{k}} \Big) + \Prob\Big( \max_{i} \sup_{x \in [0,n/k]} |u_n^{(i)}(x) - w_n^{(i)}(x)| > \tau \Big)
	\\
	& \quad \quad + \Prob\Big( \max_i \sup_{a \leq x \leq n/k} \frac{|w_n^{(i)}(x) - x|}{x} > 1/2 \Big),
\end{align*}
where the constant $\tilde A$ is as in \cref{lemm:unappr}. Moreover, if \cref{assum:r} is also satisfied, $C_3$ can be modified in a way that only depends on $\eps$ and $K(\beta)$ and such that the slightly larger probability
\[
	\Prob\Big( \max \big| A_{\cdot, 3}^\cdot \big| > 3\tau (\log(n/k) + \log(1/a))^2 + C_3 \Big( \frac{\tilde A + \lambda}{\sqrt{k}} + \Big(\frac{k}{n}\Big)^\xi (\log(n/k) + \log(1/a))^2 \Big) \Big)
\]
admits the same upper bound.

By \cref{assum:tail}, the measure $R_{J, n}$ in any integrand can be replaced by $R_J$ at the cost of adding a deterministic error of the order of $(k/n)^\xi$. After being integrated, such an error is of order at most $(k/n)^\xi(\log(n/k) + \log(1/a))^2$. We will use this fact on multiple occasions by bounding the increments of $R_J$ instead of $R_{J, n}$.

Next we observe that by Lipschitz continuity of $R_{J,n}$ the quantities $u_n^{(i)}(x),u_n^{(j)}(y),u_n^{(m)}(1)$ appearing in the arguments of $R_J$ inside $A_{\cdot, 3}^\cdot$ can be replaced by $w_n^{(i)}(x),w_n^{(j)}(y),w_n^{(m)}(1)$ with an error that is controlled by Lemma~\ref{lemm:unappr}, uniformly over $x,y,i,j,m$. For example
\begin{align*}
& \max_{i,j,m} \sup_{x,y \in [0,n/k]}\Big|R_{ijm, n}(u_n^{(i)}(x), u_n^{(j)}(y), u_n^{(m)}(1)) - R_{ijm, n}(w_n^{(i)}(x), w_n^{(j)}(y), w_n^{(m)}(1))\Big|
\\
\leq~& 3 \max_{i} \sup_{x \in [0,n/k]} |u_n^{(i)}(x) - w_n^{(i)}(x)|.
\end{align*}
Define 
\[
\tilde A_{ij, 3}^{(m), --} := \int_a^1 \int_a^1 \frac{|R_{ijm, n}(w_n^{(i)}(x), w_n^{(j)}(y), w_n^{(m)}(1)) - R_{ijm, n}(x, y, 1)|}{xy} dxdy,
\]
and similarly define other terms $\tilde A_{\cdot, 3}^\cdot$ replacing the different $A_{\cdot, 3}^\cdot$. {The difference between those quantities and the original $A_{\cdot, 3}^\cdot$ terms that they replace can be uniformly controlled as
\begin{align*}
\max \big| A_{\cdot, 3}^{\cdot} - \tilde A_{\cdot, 3}^{\cdot} \big|
\leq 3 (\log(n/k) + \log(1/a))^2 \max_{i} \sup_{x \in [0,n/k]} |u_n^{(i)}(x) - w_n^{(i)}(x)|,
\end{align*}
some of those bounds using the fact that
\begin{equation} \label{eq:intlog}
\int \frac{(\log x)^{\ell-1}}{x} dx = \frac{(\log x)^\ell}{\ell} + \text{constant}.
\end{equation}
We then obtain that for all $\tau>0$,
\[
	\Prob\Big( \max \big| A_{\cdot, 3}^{\cdot} - \tilde A_{\cdot, 3}^{\cdot} \big| > 3 \tau (\log(n/k) + \log(1/a))^2 \Big) \leq \Prob\Big( \max_{i} \sup_{x \in [0,n/k]} |u_n^{(i)}(x) - w_n^{(i)}(x)| > \tau \Big).
\]
}
Hence it remains to bound the terms $\tilde A_{\cdot, 3}^\cdot$ which are defined in the same way as $A_{\cdot, 3}^\cdot$ but with $w_n^{(i)}(x),w_n^{(j)}(y),w_n^{(m)}(1)$ replacing $u_n^{(i)}(x),u_n^{(j)}(y),u_n^{(m)}(1)$.

Note that whenever
\[
	\max_i \sup_{a \leq x \leq n/k} \frac{|w_n^{(i)}(x) - x|}{x} \leq 1/2,
\]
we have for all $i$ and $x \in [a, n/k]$ that $x/2 \leq u_n^{(i)}(x) \leq 2x$. We will assume, for the remainder of the section, that this is realized. \Cref{lemm:wn} allows us to lower bound the probability of that event.

Finally, recall the quantities
\[
	\tilde\Delta_n^{(i)}(a, \nu) = \max\bigg\{ \sup_{x \in [a, 1]} \frac{|w_n^{(i)}(x) - x|}{x^\nu}, \sup_{x \in [1, n/k]} \frac{|w_n^{(i)}(x) - x|}{x^{1-\nu}} \bigg\}
\]
from the discussion after \cref{lemm:unappr}.

\paragraph{The general case.}

We first prove the weaker bound that does not rely on \cref{assum:r}.

Firstly, for $x, y \in [a, 1]$ and every triple $(i, j, m)$,
\begin{align}\notag
	&R_{ijm, n}(w_n^{(i)}(x), w_n^{(j)}(y), w_n^{(m)}(1)) - R_{ijm, n}(x, y, 1)
    \\
    &= R_{ijm, n}(w_n^{(i)}(x), w_n^{(j)}(y), w_n^{(m)}(1)) - R_{ijm, n}(w_n^{(i)}(x), w_n^{(j)}(y), 1) \notag
	\\
	&\quad + R_{ijm, n}(w_n^{(i)}(x), w_n^{(j)}(y), 1) - R_{ijm, n}(w_n^{(i)}(x), y, 1) \notag
	\\
	&\quad + R_{ijm, n}(w_n^{(i)}(x), y, 1) - R_{ijm, n}(x, y, 1). \label{eq:copula---}
\end{align}
Each of the three differences above, by Lipschitz continuity of $R_{ij, m}$, is of course bounded by $\max_i \tilde\Delta_n^{(i)}(a, 0) \leq \max_i \tilde\Delta_n^{(i)}(a, \nu)$, for arbitrary $1/2>\nu >0$. Deduce that
\begin{align*}
	\big| \tilde A_{ij, 3}^{(m), --} \big| &\leq \int_a^1 \int_a^1 \frac{|R_{ijm, n}(w_n^{(i)}(x), w_n^{(j)}(y), w_n^{(m)}(1)) - R_{ijm, n}(x, y, 1)|}{xy} dxdy
    \\
    &\leq (\log(1/a))^2 \max_i \tilde\Delta_n^{(i)}(0, \nu).
\end{align*}
The term $\tilde A_{im, 3}^{(m), --}$ is bounded using the same strategy, following an expansion similar to \cref{eq:copula---} but without the first term. As for the term $\tilde A_{i, 3}^{(m), \ell, -}$, it is also bounded after an expansion similar to the third term in \cref{eq:copula---}, using the indefinite integral~\eqref{eq:intlog}.

Secondly, for all $x \in [1, n/k]$, $y \in [a, 1]$ and every triple $(i, j, m)$,
\begin{align}
	&R_{ijm, n}([w_n^{(i)}(x), \infty), w_n^{(j)}(y), w_n^{(m)}(1)) - R_{ijm, n}([x, \infty), y, 1) \notag
	\\
	&\quad = R_{ijm, n}([w_n^{(i)}(x), \infty), w_n^{(j)}(y), w_n^{(m)}(1)) - R_{ijm, n}([w_n^{(i)}(x), \infty), w_n^{(j)}(y), 1) \notag
	\\
	&\quad\quad + R_{ijm, n}([w_n^{(i)}(x), \infty), w_n^{(j)}(y), 1) - R_{ijm, n}([w_n^{(i)}(x), \infty), y, 1) \notag
	\\
	&\quad\quad + R_{ijm, n}([w_n^{(i)}(x), \infty), y, 1) - R_{ijm, n}([x, \infty), y, 1). \label{eq:copula+--}
\end{align}
{
The first two differences are again uniformly bounded by $\max_i \tilde\Delta_n^{(i)}(a, \nu)$ by Lipschitz continuity. As for the third difference in \cref{eq:copula+--}, let us replace $R_{ijm, n}$ by $R_{ijm}$ as described at the beginning of this section. We are then left with
\begin{equation}\label{eq:helpboundR1}
R_{ijm}([x \wedge w_n^{(i)}(x), x \vee w_n^{(i)}(x)), y, 1) \leq y \frac{|w_n^{(i)}(x) - x|}{x \wedge w_n^{(i)}(x)} \leq 2 x^{-\nu} y \max_i \tilde\Delta_n^{(i)}(a, \nu),
\end{equation}
by \cref{lemm:boundR} and the fact that we are on the event $w_n^{(i)}(x) \geq x/2$. Deduce that
\begin{align*}
\big| \tilde A_{ij, 3}^{(m), +-} \big| &\leq \int_1^{n/k} \int_a^1 \frac{|R_{ijm, n}([w_n^{(i)}(x), \infty), w_n^{(j)}(y), w_n^{(m)}(1)) - R_{ijm, n}([x, \infty), y, 1)|}{xy} dxdy
\\
&\lesssim (\log(n/k))(\log(1/a)) \max_i \tilde\Delta_n^{(i)}(a, \nu) + \Big(\frac{k}{n}\Big)^\xi (\log(n/k))(\log(1/a)),
\end{align*}
where the last term comes from the approximation of $R_{ijm, n}$ by $R_{ijm}$.
}
By symmetry, $A_{ij, 3}^{(m), -+}$ enjoys the same bound. Moreover, $A_{im, 3}^{(m), +-}$ is bounded using the same strategy, following an expansion similar to \cref{eq:copula+--} but without the first term. As for the term $A_{i, 3}^{(m), \ell, +}$, it is also bounded after an expansion similar to the third term in \cref{eq:copula+--}, using \cref{eq:intlog}.

Thirdly, for all $x, y \in [1, n/k]$ and every triple $(i, j, m)$,
\begin{align}
	&R_{ijm, n}([w_n^{(i)}(x), \infty), [w_n^{(j)}(y), \infty), w_n^{(m)}(1)) - R_{ijm, n}([x, \infty), [y, \infty), 1) \notag
	\\
	&\quad = R_{ijm, n}([w_n^{(i)}(x), \infty), [w_n^{(j)}(y), \infty), w_n^{(m)}(1)) - R_{ijm, n}([w_n^{(i)}(x), \infty), [w_n^{(j)}(y), \infty), 1) \notag
	\\
	&\quad\quad + R_{ijm, n}([w_n^{(i)}(x), \infty), [w_n^{(j)}(y), \infty), 1) - R_{ijm, n}([w_n^{(i)}(x), \infty), [y, \infty), 1) \notag
	\\
	&\quad\quad + R_{ijm, n}([w_n^{(i)}(x), \infty), [y, \infty), 1) - R_{ijm, n}([x, \infty), [y, \infty), 1). \label{eq:copula++-}
\end{align}
The first difference is upper bounded similarly to before by using the Lipschitz continuity of $R_{ijm, n}$. As for the second term, we once again replace $R_{ijm, n}$ by its limit $R_{ijm}$ and obtain
{
\begin{align}
	R_{ijm}([w_n^{(i)}(x), \infty), [y \wedge w_n^{(j)}(y), y \vee w_n^{(j)}(y)), 1) &\leq R_{jm}([y \wedge w_n^{(j)}(y), y \vee w_n^{(j)}(y)), 1) \notag
	\\
	&\leq \frac{|w_n^{(j)}(y) - y|}{y \wedge w_n^{(j)}(y)} \leq 2 y^{-\nu} \max_i \tilde\Delta_n^{(i)}(a, \nu). \label{eq:helpboundR2}
\end{align}
The third term of \cref{eq:copula++-} admits the same bound with $x$ replacing $y$. Deduce that
\begin{align*}
	&\big| \tilde A_{ij, 3}^{(m), ++} \big|
    \\
    &\leq \int_1^{\frac nk} \int_1^{\frac nk} \frac{|R_{ijm, n}([w_n^{(i)}(x), \infty), [w_n^{(j)}(y), \infty), w_n^{(m)}(1)) - R_{ijm, n}([x, \infty), [y, \infty), 1)|}{xy} dxdy
	\\
	&\lesssim (\log(n/k))^2 \max_i \tilde\Delta_n^{(i)}(a, \nu) +  \Big(\frac{k}{n}\Big)^\xi (\log(n/k))^2,
\end{align*}
the last term again from replacing $R_{ijm, n}$ by $R_{ijm}$.

We have therefore proved that, for any $1/2>\nu>0$, each term $A_{\cdot, 3}^\cdot$ is upper bounded by a constant multiple of
\[
(\log(n/k) + \log(1/a))^2 \max_i \tilde\Delta_n^{(i)}(a, \nu) +  \Big(\frac{k}{n}\Big)^\xi (\log (n/k) + \log(1/a))^2.
\]
}

\paragraph{Assuming bounded densities.}

{
Let us now suppose that \cref{assum:r} is satisfied with a certain $\eps \in (0, 4)$. While in the general case above $\nu \in (0, 1/2)$ was arbitrary, let now $\nu = 1/2 - \eps/8$.

The various bounds above on the numerators in the integrals $\tilde A_{\cdot, 3}^\cdot$ were for the most part uniform in the integrands $x$ and $y$. By integrating them over a growing domain, a poly-logarithmic factor was paid. We shall now derive more subtle bounds that are proportional to functions $f(x, y)$ such that $f(x, y)/xy$ is integrable over the infinite domain, thus allowing us to remove the extra poly-logarithmic factors.
}

Firstly, for $x, y \in [a, 1]$, consider the three terms in \cref{eq:copula---} and in each one, replace $R_{ijm, n}$ by $R_{ijm}$. By \cref{lemm:boundRsmall} with $\beta = \nu/2$, the third term is then bounded by
\[
	R_{ij}([x \wedge w_n^{(i)}(x), x \vee w_n^{(i)}(x)], y) \lesssim y^{\nu/2} \frac{|w_n^{(i)}(x) - x|}{x^{\nu/2}} \leq (xy)^{\nu/2} \max_i \tilde\Delta_n^{(i)}(a, \nu).
\]
The second term admits the same bound up to a factor of $2^{\nu/2}$, since by assumption $w_n^{(i)}(x) \leq 2x$. As for the first term, now using \cref{lemm:boundRsmall} with $\beta = \nu$, it is upper bounded by both
\[
	R_{im}(2x, [1 \wedge w_n^{(m)}(1), 1 \vee w_n^{(m)}(1)]) \lesssim x^\nu \max_i \tilde\Delta_n^{(i)}(a, \nu)
\]
and
\[
	R_{jm}(2y, [1 \wedge w_n^{(m)}(1), 1 \vee w_n^{(m)}(1)]) \lesssim y^\nu \max_i \tilde\Delta_n^{(i)}(a, \nu),
\]
hence by
\[
	(xy)^{\nu/2} \max_i \tilde\Delta_n^{(i)}(a, \nu)
\]
up to a constant. It then follows that
\begin{align*}
\big| \tilde A_{ij, 3}^{(m), --} \big| &\lesssim \max_i \tilde\Delta_n^{(i)}(a, \nu) \int_0^1 \int_0^1 (xy)^{\nu/2 - 1} dxdy + \Big(\frac{k}{n}\Big)^\xi (\log(1/a))^2
\\
&\lesssim \max_i \tilde\Delta_n^{(i)}(a, \nu) +  \Big(\frac{k}{n}\Big)^\xi (\log(1/a))^2 .
\end{align*}
The bounds on $A_{im, 3}^{(m), --}$ and $A_{i, 3}^{(m), \ell, -}$ follow from the same argument, noting for the latter that
\[
	\int_0^1 \frac{(\log x)^{\ell-1}}{x^{1 - \nu/2}} dx < \infty.
\]

Secondly, for $x \in [1, n/k]$, $y \in [a, 1]$ and every triple $(i, j, m)$, consider the three terms in \cref{eq:copula+--} and in each one, replace $R_{ijm, n}$ by $R_{ijm}$. It was already proved in the general case that by \cref{lemm:boundR}, the third term satisfies (see \cref{eq:helpboundR1})
\begin{align*}
R_{ijm}([x \wedge w_n^{(i)}(x), x \vee w_n^{(i)}(x)], y, 1) &\leq R_{ij}([x \wedge w_n^{(i)}(x), x \vee w_n^{(i)}(x)], y)
\\
&\leq 2x^{-\nu} y \max_i \tilde\Delta_n^{(i)}(a, \nu).
\end{align*}
The second term, by an application of \cref{lemm:boundRmix} with $\beta = -\eps$, is upper bounded by
\begin{align*}
R_{ij}([w_n^{(i)}(x), \infty), [y \wedge w_n^{(j)}(y), y \vee w_n^{(j)}(y)]) &\lesssim w_n^{(i)}(x)^{-\eps} (y \vee w_n^{(j)}(y))^\eps |w_n^{(j)}(y) - y|
\\
&\lesssim x^{-\eps} y^{\eps + \nu} \max_i \tilde\Delta_n^{(i)}(a, \nu).
\end{align*}
The first term of \cref{eq:copula+--} is upper bounded by both
\[
	R_{jm}(w_n^{(j)}(y), [1 \wedge w_n^{(m)}(1), 1 \vee w_n^{(m)}(1)]) \leq w_n^{(j)}(y) |w_n^{(m)}(1) - 1| \lesssim y \max_i \tilde\Delta_n^{(i)}(a, \nu),
\]
by \cref{lemm:boundRsmall}, and
\begin{align*}
R_{im}([w_n^{(i)}(x), \infty), [1 \wedge w_n^{(m)}(1), 1 \vee w_n^{(m)}(1)]) &\lesssim w_n^{(i)}(x)^{-\eps} |w_n^{(m)}(1) - 1| 
\\
&\lesssim x^{-\eps} \max_i \tilde\Delta_n^{(i)}(a, \nu),
\end{align*}
again by \cref{lemm:boundRmix} with $\beta=-\eps$. Hence the first term is in fact bounded by
\[
	x^{-\eps/2} y^{1/2} \max_i \tilde\Delta_n^{(i)}(a, \nu)
\]
up to a constant. It then follows that
\begin{align*}
	\big| \tilde A_{ij, 3}^{(m), +-} \big| &\lesssim \max_i \tilde\Delta_n^{(i)}(a, \nu) \int_0^1 \int_1^\infty \big( x^{-1-\nu} + x^{-1-\eps} y^{\eps+\nu-1} + x^{-1-\eps/2} y^{-1/2} \big) dxdy
	\\
	&\quad + \Big(\frac{k}{n}\Big)^\xi (\log(n/k))(\log(1/a))
	\\
	&\lesssim \max_i \tilde\Delta_n^{(i)}(a, \nu) +  \Big(\frac{k}{n}\Big)^\xi (\log(n/k))(\log(1/a)).
\end{align*}
The same holds for $A_{ij, 3}^{(m), -+}$ by symmetry. The bounds on $A_{im, 3}^{(m), +-}$ and $A_{i, 3}^{(m), \ell, +}$ follow from the same argument, noting for the latter that
\[
	\int_1^\infty \frac{(\log x)^{\ell-1}}{x^{1 + \zeta}} dx < \infty
\]
for any positive $\zeta$.

Finally, for $x, y \in [1, n/k]$ and every triple $(i, j, m)$, consider the three terms in \cref{eq:copula++-} and in each one, replace $R_{ijm, n}$ by $R_{ijm}$. By \cref{lemm:boundRsmall} with $\beta = 1+\eps$, the third term of \cref{eq:copula++-} satisfies
\begin{align*}
R_{ijm}([x \wedge w_n^{(i)}(x), x \vee w_n^{(i)}(x)], [y, \infty), 1) &\leq R_{im}([x \wedge w_n^{(i)}(x), x \vee w_n^{(i)}(x)], 1)
\\
&\lesssim \frac{|w_n^{(i)}(x) - x|}{x^{1+\eps}} \lesssim x^{-\nu - \eps} \max_i \tilde\Delta_n^{(i)}(a, \nu)
\end{align*}
but at the same time, \cref{lemm:boundRmix} with $\beta = -\eps/2$ yields
\begin{align*}
	R_{ijm}([x \wedge w_n^{(i)}(x), x \vee w_n^{(i)}(x)], [y, \infty), 1) &\leq R_{ij}([x \wedge w_n^{(i)}(x), x \vee w_n^{(i)}(x)], [y, \infty))
	\\
	&\lesssim y^{-\eps/2} x^{\eps/2} |w_n^{(i)}(x) - x|
	\\
	&\lesssim x^{1-\nu + \eps/2} y^{-\eps/2} \max_i \tilde\Delta_n^{(i)}(a, \nu).
\end{align*}
The minimum between the bounds above being smaller than their geometric mean, we then have
\[
	R_{ijm}([x \wedge w_n^{(i)}(x), x \vee w_n^{(i)}(x)], [y, \infty), 1) \lesssim x^{-\eps/8} y^{-\eps/4} \max_i \tilde\Delta_n^{(i)}(a, \nu),
\]
recalling that $\nu = 1/2 - \eps/8$, so that $(-\nu -\eps) + (1 - \nu + \eps/2) = 1 - 2\nu - \eps/2 = -\eps/4$. The second term of \cref{eq:copula++-} admits a similar bound since by assumption $w_n^{(i)}(x) \geq x/2$. As for the first term, by \cref{lemm:boundRmix} with $\beta=-\eps$ it is bounded by
\[
	R_{im}([x/2, \infty), [1 \wedge w_n^{(m)}(1), 1 \vee w_n^{(m)}(1)]) \leq x^{-\eps} |w_n^{(m)}(1) - 1| \lesssim x^{-\eps} \max_i \tilde\Delta_n^{(i)}(a, \nu),
\]
but also by
\[
	R_{jm}([y/2, \infty), [1 \wedge w_n^{(m)}(1), 1 \vee w_n^{(m)}(1)]) \leq y^{-\eps} |w_n^{(m)}(1) - 1| \lesssim y^{-\eps} \max_i \tilde\Delta_n^{(i)}(a, \nu),
\]
hence by
\[
	(xy)^{-\eps/2} \max_i \tilde\Delta_n^{(i)}(a, \nu)
\]
up to a constant. It then follows that
\begin{align*}
	\big| \tilde A_{ij, 3}^{(m), ++} \big| &\lesssim \max_i \tilde\Delta_n^{(i)}(a, \nu) \int_1^\infty \int_1^\infty \big( x^{-1-\eps/8} y^{-1-\eps/4}
+ (xy)^{-1 - \eps/2} \big) dxdy
    \\
    &\quad+ \Big(\frac{k}{n}\Big)^\xi (\log(n/k))^2
	\\
	&\lesssim \max_i \tilde\Delta_n^{(i)}(a, \nu) +  \Big(\frac{k}{n}\Big)^\xi (\log(n/k))^2.
\end{align*}
We have therefore proved that each term $A_{\cdot, 3}^\cdot$ is upper bounded, up to a constant, by
\[
\max_i \tilde\Delta_n^{(i)}(a, \nu) + \Big(\frac{k}{n}\Big)^\xi (\log(n/k) + \log(1/a))^2.
\]

\hfill $\Box$

\subsection{Proof of \cref{thm:concentration}}
\label{subsec:proof}

As per the statement of the theorem, pick an arbitrary $\zeta \in (0, 1)$ and assume that $k \geq n^\zeta$. Since the statement is trivial for $\lambda < 1$ (with the right choice of constants), suppose that $1 \leq \lambda \leq \sqrt{k}/(\log n)^4$. Moreover let $a$ satisfy
\begin{equation} \label{eq:a}
	\max\Big\{ \frac{\lambda^2 \log n}{k}, \Big(\frac{k}{n}\Big)^\xi \Big\} \leq a \leq \max\Big\{ \frac{\lambda}{\sqrt{k} \log n}, \Big(\frac{k}{n}\Big)^\xi \Big\}.
\end{equation}
Note that by our choice of $\lambda$, the interval above is always non-empty. Introduce the notation $l_{n,a} := \log(n/k) + \log(1/a)$ and note that by \cref{eq:a}, $l_{n, a} \lesssim \log(n/k)$. Consider the results of \cref{sec:ep,sec:iid,sec:copula} with
\[
\eps := \hat A \Big( \frac{1}{\sqrt{k}} + \frac{\log n}{k} \Big) + \frac{\lambda}{\sqrt{k}} + \frac{\lambda^2}{k} \lesssim \frac{\lambda}{\sqrt{k}} \leq \frac{1}{(\log n)^4}
\]
and
\[
\tau := \frac{A \log n + \lambda^2}{k},
\]
for $\hat A$ and $A$ as in \cref{coro:unweight} and \cref{lemm:unappr}, respectively. Combining these results with those of \cref{sec:B}, we obtain the following simultaneous upper bound on each integral $\cI_\cdot^\cdot$ in \cref{eq:erriml,eq:errimm,eq:errijm}:
{
\begin{align}
	&C_1 l_{n,a}^2 \Big\{ \Big(\frac{\log(n/k)}{k}\Big)^{1/2} \Big(\Big(\frac{k}{n}\Big)^\xi + \eps \Big)^{1/2} + \frac{\log(n/k)}{k} + \frac{\lambda}{\sqrt{k}} \Big(\Big(\frac{k}{n}\Big)^\xi + \eps \Big)^{1/2} + \frac{\lambda^2}{k} \Big\} \notag
	\\
	&\quad + \Big( 1 + \Big(\frac{k}{n}\Big)^\xi (\log(n/k))^2 l_{n,a}^2 \Big)^{1/2} \frac{\lambda}{\sqrt{k}} + l_{n,a}^2 \frac{\lambda^2}{k} \notag
	\\
	&\quad + 3 \tau l_{n,a}^2 + C_3 l_{n,a}^2 \Big( \frac{\tilde A + \lambda}{\sqrt{k}} + \Big(\frac{k}{n}\Big)^\xi \Big) \notag
	\\
	&\quad + O\Big( \Big(\frac{k}{n}\Big)^\xi l_{n,a}^2 \Big). \label{eq:preboundweak}
\end{align}
Note that $(k/n)^\xi (\log(n/k))^2 l_{n, a}^2 \lesssim (k/n)^\xi (\log(n/k))^4$ can be upper bounded by a constant only depending on $\xi$. Using this and the fact that $(x+y)^{1/2} \leq x^{1/2} + y^{1/2}$ for $x,y\geq 0$, we find
\begin{align*}
&l_{n,a}^2 \Big(\frac{\log(n/k)}{k}\Big)^{1/2} \Big(\Big(\frac{k}{n}\Big)^\xi + \eps \Big)^{1/2}
\\
&\lesssim l_{n,a}^2 \Big(\frac{k}{n}\Big)^{\xi/2} \Big(\frac{\log(n/k)}{k}\Big)^{1/2} + l_{n,a}^2  \Big(\frac{\log(n/k)}{k}\Big)^{1/2}\eps^{1/2}
\\
\lesssim& \frac{1}{\sqrt{k}} + l_{n,a}^2  \Big(\frac{\log(n/k)}{k}\Big)^{1/2}\lambda^{1/2}k^{-1/4} \lesssim \frac{\lambda}{\sqrt{k}}
\end{align*}
since $(\log n)^{5/2} \lesssim k^{1/4}$. By similar arguments using that $\eps \lesssim 1/(\log n)^4$
\[
l_{n,a}^2 \Big(\Big(\frac{k}{n}\Big)^\xi + \eps \Big)^{1/2} \lesssim 1.
\]
Moreover,
\[
\tau l_{n, a}^2 \lesssim \frac{(\log n)^3}{k} + l_{n, a}^2 \frac{\lambda^2}{k}.
\]
In addition, notice that by our choice of $\lambda$,
\[
\frac{\lambda^2}{k} \leq l_{n, a}^2 \frac{\lambda^2}{k} \lesssim \frac{\lambda (\log n)^{-2} \sqrt{k}}{k} \leq \frac{\lambda}{\sqrt{k}}
\]
and that since $k \geq n^\zeta$,
\[
	\frac{(\log n)^3}{k} \lesssim \frac{1}{\sqrt{k}}.
\]
Piecing those results together, \cref{eq:preboundweak} can be bounded by
\begin{equation} \label{eq:boundweak}
	C' \Big\{ \Big(\frac{k}{n}\Big)^\xi (\log(n/k))^2 + \frac{(\log(n/k))^2(1 + \lambda)}{\sqrt{k}} \Big\},
\end{equation}
for the right constant $C'$. If \cref{assum:r} is made, the same strategy yields the sharper bound
\begin{align}
	&C_1 l_{n,a}^2 \Big\{ \Big(\frac{\log(n/k)}{k}\Big)^{1/2} \Big(\Big(\frac{k}{n}\Big)^\xi + \eps \Big)^{1/2} + \frac{\log(n/k)}{k} + \frac{\lambda}{\sqrt{k}} \Big(\Big(\frac{k}{n}\Big)^\xi + \eps \Big)^{1/2} + \frac{\lambda^2}{k} \Big\} \notag
	\\
	&\quad + \Big( 1 + \Big(\frac{k}{n}\Big)^\xi (\log(n/k))^2 l_{n,a}^2 \Big)^{1/2} \frac{\lambda}{\sqrt{k}} + l_{n,a}^2 \frac{\lambda^2}{k} \notag
	\\
	&\quad + 3 \tau l_{n,a}^2 + C_3 \Big( \frac{\tilde A + \lambda}{\sqrt{k}} + \Big(\frac{k}{n}\Big)^\xi l_{n,a}^2 \Big) \notag
	\\
	&\quad + O\Big( \Big(\frac{k}{n}\Big)^\xi l_{n,a}^2 \Big) \notag
	\\
	&\leq \bar C' \Big\{ \Big(\frac{k}{n}\Big)^\xi (\log(n/k))^2 + \frac{1 + \lambda}{\sqrt{k}} \Big\} \label{eq:boundstrong},
\end{align}
for the right constant $\bar C'$. It is left to control the deterministic error terms in \cref{eq:erriml,eq:errimm,eq:errijm} arising from the truncation of the integrals. Those terms are upper bounded by a constant multiple of
\begin{align*}
& \Big(\frac{k}{n}\Big)^\xi (\log(n/k)) + l_{n,a}^2k^{-1} + al_{n,a}
\\
&\lesssim 
\Big(\frac{k}{n}\Big)^\xi (\log(n/k)) + \frac{1}{\sqrt{k}} + \max\Big\{ \frac{\lambda}{\sqrt{k}}, \Big(\frac{k}{n}\Big)^\xi (\log(n/k)) \Big\}
\\
&\lesssim  \Big(\frac{k}{n}\Big)^\xi (\log(n/k)) + \frac{1 + \lambda}{\sqrt{k}}
\end{align*}
so they are absorbed into the bounds above. Note that this time we have used the upper bound on $a$ in \cref{eq:a} in order to bound $a l_{n, a}$.
}

The probability that each of the two bounds in \cref{eq:boundweak,eq:boundstrong} holds is at least
\begin{align*}
	&1 - d^3 e^{-\lambda^2} - \Prob\bigg(\max_{i \in V} \hat\Delta_n^{(i)}(a, 0, \nu) > \hat A \Big( \frac{1}{\sqrt{k}} + \frac{\log n}{k} \Big) + \frac{\lambda}{\sqrt{k}} + \frac{\lambda^2}{k} \bigg)
    \\
    &\qquad - 16d^3 e^{-c_2\lambda^2 / 2} - \Prob\Big( \max_i \tilde\Delta_n^{(i)}(a, \nu) > \frac{\tilde A + \lambda}{\sqrt{k}} \Big)
	\\
	&\qquad - \Prob\bigg( \max_{i} \sup_{x \in [0,n/k]} |u_n^{(i)}(x) - w_n^{(i)}(x)| > \frac{A \log n + \lambda^2}{k} \bigg) 
    \\
    &\qquad- \Prob\bigg( \max_i \sup_{a \leq x \leq n/k} \frac{|w_n^{(i)}(x) - x|}{x} > 1/2 \bigg)
	\\
	&\geq 1 - d^3 e^{-\lambda^2} - \hat B d e^{-\hat C \lambda^2} - 16d^3 e^{-c_2\lambda^2 / 2} - \tilde B d e^{-\tilde C \lambda^2} - B d e^{-C \lambda^2}
    \\
    &\qquad - 6d \exp\Big\{ -c' k \Big( 1 \wedge \frac{a}{\log\log(1/a)} \Big) \Big\}
	\\
	&\geq 1 - Md^3 \exp\Big\{ -c \min\Big\{ \lambda^2, \frac{ka}{\log\log(1/a)} \Big\} \Big\},
\end{align*}
for suitable constants $M$ and $c$, where we have used \cref{coro:unweight,lemm:unappr,lemm:wn}. By \cref{eq:a}, since $a \geq \lambda^2 (\log n)/k$, we find
\[
	\frac{ka}{\log\log(1/a)} \geq \frac{\lambda^2 \log n}{\log\log k} \geq \lambda^2,
\]
so that the probability above is equal to
\[
	1 - Md^3 e^{-c\lambda^2}.
\]
Combining this with \cref{eq:estimerror,eq:erriml,eq:errimm,eq:errijm} finally concludes the proof, upon noting that the factor $e_i^{(m), 1} - e_j^{(m), 1}$ appearing in \cref{eq:estimerror} is upper bounded by $1 + K/\xi$ (see the proof of \cref{lemm:boundGamma}) and properly choosing the constants $C$ and $\bar C$ in terms of {$C'$ and $\bar C'$}. \hfill $\Box$

\section{Auxiliary results and proofs}
\label{sec:auxiliary}

\subsection{Discussion of \texorpdfstring{\cref{assum:tail}}{the main assumption}}

We open this section by discussing the relation between \Cref{assum:tail} and the following more standard second order condition on bounded sets.

\begin{assum}[Second order]\label{assum:RJstand}
	The marginal distributions $F_1,\dots,F_d$ are continuous and there exist positive constants $K', \xi'$ such that for all $J \subset V$, $|J| \in \{2,3\}$ and $q \in (0, 1]$,
	\begin{equation}\label{eq:RJstand}
		\sup_{\g x \in [0,1]^{|J|}} \Big|q^{-1} \Prob(F_J(\g X_J) > 1 - q\g x)- R_J(\g x) \Big| \leq K'q^{\xi'}.
	\end{equation}
\end{assum}

In Proposition~\ref{prop:maxstableA3}, we shall prove that max-stable distributions always satisfy Assumption~\ref{assum:RJstand}. As we show below, this condition together with an assumption on the tails of the bivariate $R$-functions $R_{ij}$ implies the stronger Assumption~\ref{assum:tail} and vice versa.

\begin{assum}[Tail]\label{assum:tailR}
	There exist positive constants $K_T, \xi_T$ such that for all $i \neq j \in V$ and $q \in (0, 1]$,
	\begin{equation}\label{eq:tailR}
		1 - R_{ij}(q^{-1}, 1) \leq K_T q^{\xi_T}.
	\end{equation}
\end{assum}

The relation between the above conditions and the main \cref{assum:tail} is summarized in the following proposition.

\begin{prop}\label{prop:equiv}
	If \cref{assum:tail} holds then \cref{assum:RJstand} holds with $K' = 2K$, $\xi' = \xi$ and \cref{assum:tailR} holds with $K_T = K$, $\xi_T = \xi$. Conversely, if \cref{assum:RJstand} holds with $K'$, $\xi'$ and \cref{assum:tailR} holds with $K_T$, $\xi_T$, then \cref{assum:tail} holds with $K = (K' + 2K_T)$, $\xi = \xi'\xi_T/(1 + \xi' + \xi_T)$.
\end{prop}

Note that extremal logistic distributions satisfy Assumption~\ref{assum:tailR} with $\xi_T = \theta^{-1} - 1$, and H\"usler--Reiss distributions satisfy it for any $\xi_T > 0$ provided that all entries of the matrix $\Gamma$ are bounded away from zero and infinity (see \cref{lemm:logisticdensity,lemm:HRdensity} below). Therefore, for extremal logistic and H\"usler--Reiss distributions, \cref{assum:tail} holds with exponent $\xi$ as soon as \cref{assum:RJstand} holds for a strictly larger exponent $\xi'$.

\begin{proof}

We first show that \cref{eq:tail} is sufficient for \cref{assum:RJstand,assum:tailR}, and subsequently prove the converse which turns out to be more involved.

\noindent\textbf{\Cref{assum:tail} implies \cref{assum:RJstand,assum:tailR}:}
Assume that \cref{eq:tail} holds for all $q \in (0, 1]$ and all sets $J \subset V$ of size 3. We then have, for any $i, j, m$,
\begin{align*}
	R_{im}(q^{-1}, 1) &= R_{ijm}(q^{-1}, \infty, 1)
	\\
	&\geq R_{ijm}(q^{-1}, q^{-1}, 1)
	\\
	&\geq q^{-1} \Prob(F_i(X_i) > 0, F_j(X_j) > 0, F_m(X_m) > 1-q) - K q^\xi
	\\
	&= 1 - K q^\xi,
\end{align*}
since the marginal distribution of $F_m(X_m)$ is uniform on $(0, 1)$. Thus \cref{eq:tailR} holds with $K_T = K$, $\xi_T = \xi$.

Now, that \cref{eq:RJstand} follows from \cref{eq:tail} when $|J|=3$ is trivial. For the case where $J$ is a pair, say $(i, m)$, let $x \leq q^{-1}$, $z \leq 1$. We have
\begin{align}
	&\Big| q^{-1} \Prob\big( F_i(X_i) > 1-qx, F_m(X_m) > 1-qz \big) - R_{im}(x, z) \Big| \notag
	\\
	&\quad = \Big| q^{-1} \Prob\big( F_i(X_i) > 1-qx, F_j(X_j) > 1-qq^{-1}, F_m(X_m) > 1-qz \big) - R_{im}(x, z) \Big| \notag
	\\
	&\quad \leq \Big| R_{ijm}(x, q^{-1}, z) - R_{im}(x, z) \Big| + K q^\xi \notag
	\\
	&\quad = R_{ijm}([0, x], [q^{-1}, \infty), [0, z]) + K q^\xi, \label{eq:B2proof}
\end{align}
using \cref{eq:tail} and the representation of $R_{ijm}$ as a non-negative measure. Then, \cref{eq:tailR} implies that the first term above is upper bounded by $R_{jm}([q^{-1}, \infty), 1) \leq K q^\xi$. Hence for pairs $J$, \cref{eq:RJstand} (in fact, a stronger version thereof where one component of $\bx$ is allowed to grow) holds with $K' = 2K$, $\xi' = \xi$.

\noindent\textbf{\Cref{assum:RJstand,assum:tailR} imply \cref{assum:tail}:}
Assume that \cref{eq:RJstand,eq:tailR} hold for all $q \in (0, 1]$ and all pairs and triples $J \subset V$ of indices. As in the statement of the result, let $\xi := \xi'\xi_T / (1 + \xi' + \xi_T)$. Let
\[
	\psi := \frac{\xi'}{1 + \xi' + \xi_T} \in (0, 1),
\]
and note that both $-\psi + (1-\psi)\xi'$ and $\psi\xi_T$ are equal to $\xi$.

We wish to bound
\begin{equation} \label{eq:RJtobound}
	\Big| q^{-1} \Prob(F_i(X_i) > 1-qx, F_j(X_j) > 1-qy, F_m(X_m) > 1-qz) - R_{ijm}(x, y, z) \Big|
\end{equation}
uniformly over all $x, y \in [0, q^{-1}]$, $z \in [0, 1]$. Let us divide the square $[0, q^{-1}]^2$ of possible values of $(x, y)$ into four quadrants defined by the axes $x=q^{-\psi}$ and $y=q^{-\psi}$. First, for all $x, y, z \leq q^{-\psi}$,
\begin{align}
	&\Big| q^{-1} \Prob(F_i(X_i) > 1-qx, F_j(X_j) > 1-qy, F_m(X_m) > 1-qz) - R_{ijm}(x, y, z) \Big| \notag
	\\
	&\quad = \Big| q^{-1} \Prob(F_i(X_i) > 1 - q^{1-\psi} q^\psi x, F_j(X_j) > 1 - q^{1-\psi} q^\psi y, F_m(X_m) > 1 - q^{1-\psi} q^\psi z) \notag
	\\
	&\quad\quad - q^{-\psi} R_{ijm}(q^\psi x, q^\psi y, q^\psi z) \Big| \notag
	\\
	&\quad = q^{-\psi} \Big| q^{\psi-1} \Prob(F_i(X_i) > 1 - q^{1-\psi} q^\psi x, F_j(X_j) > 1 - q^{1-\psi} q^\psi y, F_m(X_m) > 1 - q^{1-\psi} q^\psi z) \notag
	\\
	&\quad\quad - R_{ijm}(q^\psi x, q^\psi y, q^\psi z) \Big| \notag
	\\
	&\quad \leq K' q^{-\psi + (1-\psi)\xi'} = K' q^\xi, \label{eq:B3proof1}
\end{align}
where we applied \cref{eq:RJstand} with $q$ replaced by $q^{1-\psi}$, since $q^\psi (x, y, z) \in [0, 1]^3$. This bounds \cref{eq:RJtobound} for $x, y \leq q^{-\psi}$.

Second, for $q^{-\psi} \leq x, y \leq q^{-1}$, $z \leq 1$,
\begin{align}
	z &\geq R_{ijm}(x, y, z) = z R_{ijm}(x/z, y/z, 1) \geq z R_{ijm}(q^{-\psi}, q^{-\psi}, 1) \notag
	\\
	&= z R_{ijm}([0, q^{-\psi}] \times [0, q^{-\psi}] \times [0, 1]) \notag
	\\
	&\geq z (R_{im}([0, q^{-\psi}], [0, 1]) + R_{jm}([0, q^{-\psi}], [0, 1]) - 1) \notag
	\\
	&\geq z\big( 1 - 2K_T (q^{\psi})^{\xi_T} \big) \notag
	\\
	&\geq z - 2K_T q^\xi, \label{eq:B3proof2}
\end{align}
using \cref{eq:tailR} to lower bound $R_{im}$ and $R_{jm}$. Similarly,
\begin{align*}
	z &\geq q^{-1} \Prob(F_i(X_i) > 1-qx, F_j(X_j) > 1-qy, F_m(X_m) > 1-qz)
	\\
	&\geq q^{-1} \Prob(F_i(X_i) > 1-qq^{-\psi}, F_j(X_j) > 1-qq^{-\psi}, F_m(X_m) > 1-qz)
	\\
	&\geq R_{ijm}(q^{-\psi}, q^{-\psi}, z) - K' q^\xi
	\\
	&= z R_{ijm}(q^{-\psi}/z, q^{-\psi}/z, 1) - K' q^\xi
	\\
	&\geq z R_{ijm}(q^{-\psi}, q^{-\psi}, 1) - K' q^\xi,
\end{align*}
where the third inequality follows from \cref{eq:B3proof1}. Using the developments leading to \cref{eq:B3proof2}, this lower bound is itself lower bounded by
\[
	z - (K' + 2K_T) q^\xi.
\]
Deduce that \cref{eq:RJtobound} is bounded by $(K' + 2K_T) q^\xi$ for $q^{-\psi} \leq x, y \leq q^{-1}$.

Third, let $q^{-\psi} \leq x \leq q^{-1}$, $y \leq q^{-\psi}$, $z \leq 1$; the case where $q^{-\psi} \leq y \leq q^{-1}$ and $x \leq q^{-\psi}$ is handled symmetrically. We will again sandwich the two terms in \cref{eq:RJtobound}. We first have
\begin{align*}
	R_{jm}(y, z) &\geq R_{ijm}(x, y, z)
	\\
	&\geq R_{ijm}(q^{-\psi}, y, z)
	\\
	&= R_{jm}(y, z) - (R_{jm}(y, z) - R_{ijm}(q^{-\psi}, y, z))
	\\
	&\geq R_{jm}(y, z) - (z - R_{im}(q^{-\psi}, z))
	\\
	&\geq R_{jm}(y, z) - (1 - R_{im}(q^{-\psi}, 1))
	\\
	&\geq R_{jm}(y, z) - K_T q^\xi,
\end{align*}
where in the last step we use \cref{eq:tailR}. The other term in \cref{eq:RJtobound} enjoys similar upper and lower bounds: by \cref{eq:B3proof1} and by the preceding lower bound on $R_{ijm}(q^{-\psi}, y, z)$,
\begin{align*}
	&q^{-1} \Prob(F_i(X_i) > 1-qx, F_j(X_j) > 1-qy, F_m(X_m) > 1-qz)
	\\
	&\quad \geq q^{-1} \Prob(F_i(X_i) > 1-qq^{-\psi}, F_j(X_j) > 1-qy, F_m(X_m) > 1-qz)
	\\
	&\quad \geq R_{ijm}(q^{-\psi}, y, z) - K'q^\xi
	\\
	&\quad \geq R_{jm}(y, z) - (K' + K_T) q^\xi.
\end{align*}
Meanwhile,
\begin{align*}
	&q^{-1} \Prob(F_i(X_i) > 1-qx, F_j(X_j) > 1-qy, F_m(X_m) > 1-qz)
	\\
	&\quad \leq q^{-1} \Prob(F_j(X_j) > 1-qy, F_m(X_m) > 1-qz)
	\\
	&\quad \leq q^{-\psi} q^{\psi-1} \Prob(F_j(X_j) > 1 - q^{1-\psi} q^\psi y, F_m(X_m) > 1 - q^{1-\psi} q^\psi z)
	\\
	&\quad \leq q^{-\psi} \big( R_{jm}(q^\psi y, q^\psi z) + K' q^{(1-\psi)\xi'} \big)
	\\
	&\quad = R_{jm}(y, z) + K' q^\xi,
\end{align*}
where we have used \cref{eq:RJstand} with $q$ replaced by $q^{1-\psi}$, since $q^\psi (y, z) \in [0, 1]^2$. Deduce that \cref{eq:RJtobound} is bounded by $(K' + K_T) q^\xi$ for $y \leq q^{-\psi} \leq x \leq q^{-1}$ (and also for $y \leq q^{-\psi} \leq x \leq q^{-1}$ by symmetry).

We have therefore established that for all $x, y \leq q^{-1}$, $z \leq 1$, \cref{eq:RJtobound} is upper bounded by $(K' + 2K_T) q^\xi$, i.e., \cref{assum:tail} is satisfied with the desired values $K$ and $\xi$.
\end{proof}

\subsection{Properties of extremal logistic and H\"usler--Reiss distributions}

Several objects, other than multivariate Pareto distributions, are routinely used to describe tails of multivariate random vectors. For instance, one can show that~\eqref{eq:mPdlimit} implies existence of the limit
\begin{equation} \label{eq:elllimit}
	L(\g x) = \lim_{q \downarrow 0} q^{-1} \Prob(F(\g X) \not \leq 1 - q\g x), \quad \g x \in [0,\infty)^d.
\end{equation} 
The function $L$ is called the stable tail dependence function of $\g X$ and is a popular object in the study of multivariate extremes \citep{H1992,EKS2012,fougeres2015}. The link between the distribution of the multivariate Pareto attractor $\g Y$ of $\bX$ and its stable tail dependence function $L$ is given by the relation
\[
	L(\g x) = \frac{\Prob(\g Y \not \leq 1/\g x)}{\Prob(Y_1 > 1)}, \quad \g x \in [0,1]^d.
\]
Using known expressions for the stable tail dependence functions of the bivariate extremal logistic and H\"usler--Reiss distributions \citep[see, e.g.,][]{seg2010}, we now show that any such distribution satisfies \cref{assum:r,assum:tailR}.

\begin{lemm}
\label{lemm:logisticdensity}
	
	Suppose that $\bY$ has an extremal logistic distribution with parameter $\theta \in (0, 1)$. For any distinct pair $(i, j)$, the bivariate $R$-function $R_{ij}$ satisfies the following.
	\begin{itemize}
		\item[(i)] We have
		\[
			1 - R_{ij}(q^{-1}, 1) \leq \theta q^{\theta^{1-} - 1}, \quad q \in (0, 1].
		\]
		\item[(ii)] The function $R_{ij}$ has density
		\[
			r_{ij}(x, y) = \frac{1-\theta}{\theta} \frac{(xy)^{\theta^{-1} - 1}}{(x^{\theta^{-1}} + y^{\theta^{-1}})^{2-\theta}}, \quad (x, y) \in (0, \infty)^2.
		\]
		For any $\beta \in [1-\theta^{-1}, \theta^{-1}]$, this density enjoys the upper bound
		\[
			r_{ij}(x, y) \leq \frac{1-\theta}{\theta x^\beta y^{1-\beta}}.
		\]
	\end{itemize}
	
\end{lemm}

\begin{proof}
	
	The pair $(Y_i, Y_j)$ has a bivariate extremal logistic distribution with dependence parameter $\theta$, so its stable tail dependence function is
	\[
		L_{ij}(x, y) = \big( x^{\theta^{-1}} + y^{\theta^{-1}} \big)^\theta
	\]
	and $R_{ij}$ is given by
	\[
		R_{ij}(x, y) = x + y - \big( x^{\theta^{-1}} + y^{\theta^{-1}} \big)^\theta.
	\]
	
	\noindent\textbf{Proof of} \textit{(i)}\textbf{:} We have
	\[
		1 - R_{ij}(q^{-1}, 1) = \big( q^{-\theta^{-1}} + 1 \big)^\theta - q^{-1} = \big( q^{-\theta^{-1}} + 1 \big)^\theta - \big( q^{-\theta^{-1}} \big)^\theta \leq \theta \big( q^{-\theta^{-1}} \big)^{\theta-1} = \theta q^{\theta^{-1} - 1},
	\]
	where the inequality holds since the function $x \mapsto x^\theta$ is concave and its derivative is given by $\d x^\theta/\d\theta = \theta x^{\theta-1}$.
	
	\noindent\textbf{Proof of} \textit{(ii)}\textbf{:}
	The density of $R_{ij}$ is defined as
	\[
		r_{ij}(x, y) := \frac{\partial^2}{\partial x \partial y} R_{ij}(x, y) = -\frac{\partial^2}{\partial x \partial y} L_{ij}(x, y).
	\]
	By elementary calculus, we have
	\[
		\frac{\partial}{\partial x} L_{ij}(x, y) = \big( x^{\theta^{-1}} + y^{\theta^{-1}} \big)^{\theta-1} x^{\theta^{-1} - 1},
	\]
	and so
	\[
		\frac{\partial^2}{\partial x \partial y} L_{ij}(x, y) = \frac{\theta-1}{\theta} \big( x^{\theta^{-1}} + y^{\theta^{-1}} \big)^{\theta-2} (xy)^{\theta^{-1} - 1},
	\]
	which yields the desired density $r_{ij}$. As for the upper bound, first consider points $(x, y) \in (0, \infty)^2$ such that $x \leq y$. Note then that $x^\beta y^{1-\beta}$ is decreasing in $\beta$, which means that for $\beta \in [1 - \theta^{-1}, \theta^{-1}]$,
	\begin{align*}
		x^\beta y^{1-\beta} r_{ij}(x, y) &\leq x^{1-\theta^{-1}} y^{\theta^{-1}} r_{ij}(x, y) = \frac{1-\theta}{\theta} \frac{y^{2\theta^{-1} - 1}}{\big( x^{\theta^{-1}} + y^{\theta^{-1}} \big)^{2-\theta}}
        \\
        &= \frac{1-\theta}{\theta} \bigg( \frac{y^{\theta^{-1}}}{x^{\theta^{-1}} + y^{\theta^{-1}}} \bigg)^{2-\theta} \leq \frac{1-\theta}{\theta}.
	\end{align*}
	For points such that $x > y$, by the same argument we have 
    \[
    x^\beta y^{1-\beta} r_{ij}(x, y) \leq x^{\theta^{-1}} y^{1-\theta^{-1}} r_{ij}(x, y) \leq (1-\theta)/\theta.
    \]
\end{proof}

\begin{lemm}
\label{lemm:HRdensity}
	
	Suppose that $\bY$ has a H\"usler--Reiss distribution with parameter matrix $\Gamma$. For any distinct pair $(i, j)$, as long as $\lambda := \sqrt{\Gamma_{ij}} > 0$, its bivariate $R$-function $R_{ij}$ satisfies the following.
	\begin{itemize}
		\item[(i)] For any positive $\xi$, there exists a finite constant $K_\xi$ (which also depends on $\lambda$) such that
		\[
			1 - R_{ij}(q^{-1}, 1) \leq K_\xi q^\xi, \quad q \in (0, 1].
		\]
		\item[(ii)] The function $R_{ij}$ has density
		\[
			r_{ij}(x, y) = \frac{1}{2 \sqrt{2\pi} \lambda \sqrt{xy}} \exp\left\{ -\frac{\lambda^2}{2} - \frac{(\log x - \log y)^2}{8\lambda^2} \right\}, \quad (x, y) \in (0, \infty)^2.
		\]
		For any $\beta \in \R$, this density enjoys the upper bound
		\[
			r_{ij}(x, y) \leq \frac{K(\beta)}{x^\beta y^{1-\beta}}, \quad K(\beta) := \frac{\exp\{ \lambda^2 (2 (\beta - 1/2)^2 - 1/2) \}}{2\sqrt{2\pi} \lambda}.
		\]
	\end{itemize}
	
\end{lemm}

\begin{proof}
	
	The pair $(Y_i, Y_j)$ has a bivariate H\"usler--Reiss distribution with dependence parameter $\lambda^2$, so its stable tail dependence function is
	\[
		L_{ij}(x, y) = x \Phi\bigg( \lambda + \frac{\log x - \log y}{2\lambda} \bigg) + y \Phi\bigg( \lambda + \frac{\log y - \log x}{2\lambda} \bigg),
	\]
	so $R_{ij}$ is given by
	\[
		R_{ij}(x, y) = x \Phi^c\bigg( \lambda + \frac{\log x - \log y}{2\lambda} \bigg) + y \Phi^c\bigg( \lambda + \frac{\log y - \log x}{2\lambda} \bigg),
	\]
	where $\Phi^c$ denotes the standard Gaussian survival function.
	
	\noindent\textbf{Proof of} \textit{(i)}\textbf{:} Fix a number $\xi>0$. First note that if $q \geq e^{-4\lambda^2}$, we trivially have that
	\[
		1 - R_{ij}(q^{-1}, 1) \leq 1 \leq e^{4\lambda^2 \xi} q^\xi,
	\]
	so we shall assume without loss of generality that $q \leq e^{-4\lambda^2}$. This implies that $\log q^{-1} \geq 4\lambda^2$, or equivalently
	\[
		\frac{\log q^{-1}}{2\lambda} - \lambda \geq \frac{\log q^{-1}}{4\lambda}.
	\]
	We then have
	\begin{align*}
		1 - R_{ij}(q^{-1}, 1) &\leq 1 - \Phi^c\bigg( \lambda + \frac{\log q}{2\lambda} \bigg) = \Phi^c\bigg( \frac{\log q^{-1}}{2\lambda} - \lambda \bigg) \leq \Phi^c\bigg( \frac{\log q^{-1}}{4\lambda} \bigg)
		\\
		&\leq \frac{4\lambda}{\sqrt{2\pi} \log q^{-1}} \exp\bigg\{ -\frac{1}{32\lambda^2} \big( \log q^{-1} \big)^2 \bigg\},
	\end{align*}
	the last inequality following from well known bounds on the Gaussian tails \citep[][Theorem 1.2.3]{Durrett4.1}. This is in turn upper bounded by
	\[
		\frac{1}{\sqrt{2\pi} \lambda} q^{(\log q^{-1})/32\lambda^2},
	\]
	which is of smaller order than any power of $q$ since the exponent diverges as $q \downarrow 0$. We can therefore upper bound it by any power $q^\xi$, up to a multiplicative constant depending on both $\xi$ and $\lambda$.
	
	\noindent\textbf{Proof of} \textit{(ii)}\textbf{:}
	The density of $R_{ij}$ is defined as
	\[
		r_{ij}(x, y) := \frac{\partial^2}{\partial x \partial y} R_{ij}(x, y) = -\frac{\partial^2}{\partial x \partial y} L_{ij}(x, y).
	\]
	
	First, we have
	\begin{align*}
		\frac{\partial}{\partial x} x \Phi\bigg( \lambda + \frac{\log x - \log y}{2\lambda} \bigg) &= \Phi\bigg( \lambda + \frac{\log x - \log y}{2\lambda} \bigg) + x \phi\bigg( \lambda + \frac{\log x - \log y}{2\lambda} \bigg) \frac{1}{2\lambda x}
		\\
		&= \Phi\bigg( \lambda + \frac{\log x - \log y}{2\lambda} \bigg) + \frac{1}{2\lambda} \phi\bigg( \lambda + \frac{\log x - \log y}{2\lambda} \bigg),
	\end{align*}
	so
	\begin{align*}
		&\frac{\partial^2}{\partial x \partial y} x \Phi\bigg( \lambda + \frac{\log x - \log y}{2\lambda} \bigg)
        \\
        &= -\frac{1}{2\lambda y} \phi\bigg( \lambda + \frac{\log x - \log y}{2\lambda} \bigg) - \frac{1}{4\lambda^2 y} \phi'\bigg( \lambda + \frac{\log x - \log y}{2\lambda} \bigg)
		\\
		&= -\frac{1}{4\lambda^2 y} \bigg( \lambda + \frac{\log y - \log x}{2\lambda} \bigg) \phi\bigg( \lambda + \frac{\log x - \log y}{2\lambda} \bigg),
	\end{align*}
	where we used the expression $\phi'(t) = -t\phi(t)$ for the derivative of the standard Gaussian density $\phi$. Now by definition of $\phi$, this is equal to
	\begin{align*}
		&-\frac{1}{4\sqrt{2\pi} \lambda^2 y} \bigg( \lambda + \frac{\log y - \log x}{2\lambda} \bigg) \exp\left\{ -\frac{\lambda^2}{2} - \frac{(\log x - \log y)^2}{8\lambda^2} + \frac{\log y - \log x}{2} \right\}
		\\
		&\quad = -\frac{1}{4\sqrt{2\pi} \lambda^2 \sqrt{xy}} \bigg( \lambda + \frac{\log y - \log x}{2\lambda} \bigg) \exp\left\{ -\frac{\lambda^2}{2} - \frac{(\log x - \log y)^2}{8\lambda^2} \right\}.
	\end{align*}
	Adding this to
	\begin{align*}
		&\frac{\partial^2}{\partial x \partial y} y \Phi\bigg( \lambda + \frac{\log y - \log x}{2\lambda} \bigg)
        \\
        & = -\frac{1}{4\sqrt{2\pi} \lambda^2 \sqrt{xy}} \bigg( \lambda + \frac{\log x - \log y}{2\lambda} \bigg) \exp\left\{ -\frac{\lambda^2}{2} - \frac{(\log x - \log y)^2}{8\lambda^2} \right\},
	\end{align*}
	obtained by a symmetric argument, yields the desired density. As for the upper bound, note that for any $\beta \in \R$,
	\[
		r_{ij}(x, y) = \frac{1}{2 \sqrt{2\pi} \lambda x^\beta \sqrt{y}} \exp\left\{ -\frac{\lambda^2}{2} - \frac{(\log x - \log y)^2}{8\lambda^2} + (\beta - 1/2) \log x \right\}.
	\]
	Writing $u$ and $v$ for $\log x$ and $\log y$, the exponent above is
	\[
		-\frac{\lambda^2}{2} - \frac{(u-v)^2}{8\lambda^2} + (\beta - 1/2) u = \frac{-u^2}{8\lambda^2} + \Big( (\beta - 1/2) + \frac{v}{4\lambda^2} \Big) u - \frac{v^2}{8\lambda^2}
	\]
	which is maximized (in $u$) at $u = v + 4\lambda^2 (\beta - 1/2)$, hence
	\begin{align*}
		-\frac{\lambda^2}{2} - \frac{(u-v)^2}{8\lambda^2} + (\beta - 1/2) u &\leq -\frac{\lambda^2}{2} - 2\lambda^2 (\beta - 1/2)^2 + (\beta - 1/2)v + 4\lambda^2 (\beta - 1/2)^2
		\\
		&= (\beta - 1/2)v + \lambda^2 (2 (\beta - 1/2)^2 - 1/2).
	\end{align*}
	Conclude that
	\[
		r_{ij}(x, y) \leq \frac{1}{2 \sqrt{2\pi} \lambda x^\beta y^{1-\beta}} \exp\{ \lambda^2 (2 (\beta - 1/2)^2 - 1/2) \}.
	\]
\end{proof}

\subsection{The moments \texorpdfstring{$e_m^{(m), \ell}$}{emml}}
\label{sec:emml}

Recalling that for any $m$, $Y_m^{(m)}$ has a unit Pareto distribution, and thus that $\log Y_m^{(m)}$ has a unit exponential distribution, it is evident that $e_m^{(m), 1} = 1$ and $e_m^{(m), 2} = 2$. As for the empirical versions $\hat e_m^{(m), \ell}$ of those moments, they are in fact deterministic, since the terms $\hat F_m(U_{tm})$ appearing in the sum are exactly the $k$ smallest such terms $\{1/n, \dots, k/n\}$. Precisely, we have the following result.

\begin{lemm}
\label{lemm:emml}
	
	As long as $k \geq 3$, we have
	\[
		\big| \hat e_m^{(m), 1} - 1 \big| \leq \frac{3\log k}{k}, \quad \big| \hat e_m^{(m), 2} - 2 \big| \leq \frac{8(\log k)^2}{k}.
	\]
	
\end{lemm}

\begin{proof}
	
	By definition, we have
	\begin{equation} \label{eq:emml}
		\hat e_m^{(m), \ell} = \frac{1}{k} \sum_{j=1}^k \{\log(k/j)\}^\ell =
		\begin{cases}
			\log k - \frac{1}{k} \sum_{j=1}^k \log j, \quad &\ell=1
			\\
			(\log k)^2 - \frac{2\log k}{k} \sum_{j=1}^k \log j + \frac{1}{k} \sum_{j=1}^k (\log j)^2, \quad &\ell=2
		\end{cases}.
	\end{equation}
	Note that
	\[
		\sum_{j=1}^k (\log j)^\ell = \sum_{j=2}^k \int_j^{j+1} (\log j)^\ell dt \in \bigg[ \int_1^k (\log t)^\ell dt, \int_2^{k+1} (\log t)^\ell dt \bigg].
	\]
	Evaluating those integrals yields
	\[
		k \{\log k - 1\} + 1 \leq \sum_{j=1}^k \log j \leq (k+1)\{\log(k+1) - 1\} - 2(\log 2 - 1)
	\]
	and
	\begin{align*}
	    &k \big\{ (\log k)^2 - 2\log k + 2 \big\} - 2 \leq \sum_{j=1}^k (\log j)^2
        \\
        &\leq (k+1) \big\{ (\log(k+1))^2 - 2\log(k+1) + 2 \big\} - 2\big\{ (\log 2)^2 - 2\log 2 + 2 \big\}.
	\end{align*}
	Denote by $a_\ell$ and $b_\ell$ the lower and upper bound on $\sum_{j=1}^k (\log j)^\ell$ above, $\ell \in \{1, 2\}$. As long as $k \geq 3$, we have by \cref{eq:emml} and by simple computations
	\[
		k \big| \hat e_m^{(m), 1} - 1 \big| \leq |a_1 - k\log k + k| \vee |b_1 - k\log k + k| \leq 3 \log k
	\]
	and
	\begin{align*}
		k \big| \hat e_m^{(m), 2} - 2 \big| &\leq |a_2 - 2(\log k) b_1 + k(\log k)^2 - 2k| \vee |b_2 - 2(\log k) a_1 + k(\log k)^2 - 2k| 
        \\
        &\leq 8(\log k)^2,
	\end{align*}
	which is the desired result.
\end{proof}

\subsection{Verifying the integral representations of different moments}

We start by deriving general expressions for the moments of logarithms of random vectors which will lead to proving the representations in \cref{eq:erriml,eq:errimm,eq:errijm}. The following result is a multivariate version of the so-called ``Darth Vader rule''.

\begin{lemm}
\label{lemm:DarthVader}
	
	Let $X_1, \dots, X_d$ be non-negative random variables and $p_1, \dots, p_d > 0$. Then
	\[
		\E\bigg[ \prod_{j=1}^d X_j^{p_j} \bigg] = \int_{[0, \infty)^d} \prod_{j=1}^d p_j x_j^{p_j-1} \Prob(X_1 \geq x_1, \dots, X_d \geq x_d) dx_1 \dots dx_d.
	\]
	Moreover, any number of ``$\geq$'' can be replaced by ``$>$'', as this changes the value of the probability, at most, on a Lebesgue-null set.
	
\end{lemm}

\begin{proof}
	
	Letting $(\Omega, \cF, P)$ be the underlying probability space containing all the random variables, we have
	\begin{align*}
		\E\bigg[ \prod_{j=1}^d X_j^{p_j} \bigg] &= \int_\Omega \prod_{j=1}^d X_j(\omega)^{p_j} P(d\omega)
		\\
		&= \int_\Omega \int_{[0, X_1(\omega)^{p_1}] \times \dots \times [0, X_d(\omega)^{p_d}]} du_1 \dots du_d P(d\omega)
		\\
		&= \int_\Omega \int_{[0, \infty)^d} \Ind{ X_1(\omega) \geq u_1^{1/p_1}, \dots, X_d(\omega) \geq u_d^{1/p_d} } du_1 \dots du_d P(d\omega)
		\\
		&= \int_{[0, \infty)^d} \bigg( \int_\Omega \Ind{ X_1(\omega) \geq u_1^{1/p_1}, \dots, X_d(\omega) \geq u_d^{1/p_d} } P(d\omega) \bigg) du_1 \dots du_d
		\\
		&= \int_{[0, \infty)^d} \Prob\big( X_1 \geq u_1^{1/p_1}, \dots, X_d \geq u_d^{1/p_d} \big) du_1 \dots du_d,
	\end{align*}
	where we have used the fact that $X_j(\omega) \geq 0$ for almost every $\omega$ to justify the second equality. The change in the order of integration was allowed by Tonelli's theorem. Finally, applying the change of variable $x_j = u_j^{1/p_j}$, $du_j/dx_j = p_j x_j^{p_j-1}$ produces the desired result.
\end{proof}

\begin{lemm}
\label{lemm:Elog}
Let $X$ and $Y$ be almost surely positive random variables and let $S$ be the distribution function of $(1/X, 1/Y)$, so that for positive $x, y$, $\Prob(X \geq x, Y \geq y) = S(1/x, 1/y)$. Then for any $p \in \{1, 2, \dots\}$,
	\begin{align}
		&\E\big[ (\log X)^p \Ind{X > 1} \big] = p \int_0^1 \frac{S(x, \infty) |\log x|^{p-1}}{x} dx, \label{eq:moment+}
		\\
		&\E\big[ (-\log X)^p \Ind{X < 1} \big] = p \int_1^\infty \frac{S([x, \infty), \infty) |\log x|^{p-1}}{x} dx, \label{eq:moment-}
		\\
		&\E\big[ ((\log X)(\log Y))^p \Ind{X, Y > 1} \big] = p^2 \int_0^1 \int_0^1 \frac{S(x, y) |(\log x)(\log y)|^{p-1}}{xy} dxdy \label{eq:moment++},
		\\
		&\E\big[ (-(\log X)(\log Y))^p \Ind{X < 1, Y > 1} \big] = p^2 \int_0^1 \int_1^\infty \frac{S([x, \infty), y) |(\log x)(\log y)|^{p-1}}{xy} dxdy \label{eq:moment+-},
		\\
		&\E\big[ ((\log X)(\log Y))^p \Ind{X, Y < 1} \big] = p^2 \int_1^\infty \int_1^\infty \frac{S([x, \infty), [y, \infty)) |(\log x)(\log y)|^{p-1}}{xy} dxdy \label{eq:moment--},
	\end{align}
	where $S([x, \infty), y)$ and $S([x, \infty), [y, \infty))$ are shorthand for $S(\infty, y) - S(x, y)$ and $1 - S(x, \infty) - S(\infty, y) + S(x, y)$, respectively.
	
\end{lemm}

\begin{proof}
	
	First, by \cref{lemm:DarthVader} with $d=1$, $p_1=p$,
	\begin{align*}
		\E\big[ (\log X)^p \Ind{X>1} \big] &= p \int_0^\infty u^{p-1} \Prob(\log X \geq u) du
		\\
		&= p \int_0^\infty u^{p-1} S(e^{-u}, \infty) du
		\\
		&= p \int_0^1 \frac{S(x, \infty) (-\log x)^{p-1}}{x} dx,
	\end{align*}
	by the change of variable $x = e^{-u}$. Similarly,
	\begin{align*}
		\E\big[ (-\log X)^p \Ind{X<1} \big] &= p \int_0^\infty u^{p-1} \Prob(\log X \leq -u) du
		\\
		&= p \int_0^\infty u^{p-1} S([e^u, \infty), \infty) du
		\\
		&= p \int_1^\infty \frac{S([x, \infty), \infty) (\log x)^{p-1}}{x} dx,
	\end{align*}
	by the change of variable $x = e^u$. This establishes \cref{eq:moment+,eq:moment-}.
	
	\Cref{eq:moment--,eq:moment+-,eq:moment++} are proved in a similar fashion by using \cref{lemm:DarthVader} with $d=2$, $p_1=p_2=p$. First,
	\begin{align*}
		\E[((\log X)(\log Y))^p \Ind{X, Y > 1}] &= p^2 \int_0^\infty \int_0^\infty (uv)^{p-1} \Prob(\log X \geq u, \log Y \geq v) dudv
		\\
		&= p^2 \int_0^\infty \int_0^\infty (uv)^{p-1} S(e^{-u}, e^{-v}) dudv
		\\
		&= p^2 \int_0^1 \int_0^1 \frac{S(x, y) ((\log x)(\log y))^{p-1}}{xy} dxdy,
	\end{align*}
	using the change of variable $x=e^{-u}$, $y=e^{-v}$. Second,
	\begin{align*}
		&\E[((-\log X)(\log Y))^p \Ind{X < 1, Y > 1}]
        \\
        &= p^2 \int_0^\infty \int_0^\infty (uv)^{p-1} \Prob(\log X \leq -u, \log Y \geq v) dudv
		\\
		&= p^2 \int_0^\infty \int_0^\infty (uv)^{p-1} S([e^u, \infty), e^{-v}) dudv
		\\
		&= p^2 \int_0^1 \int_1^\infty \frac{S([x, \infty), y) (-(\log x)(\log y))^{p-1}}{xy} dxdy,
	\end{align*}
	using the change of variable $x=e^u$, $y=e^{-v}$. Third,
	\begin{align*}
		&\E[((\log X)(\log Y))^p \Ind{X, Y < 1}]
        \\
        &= p^2 \int_0^\infty \int_0^\infty (uv)^{p-1} \Prob(\log X \leq -u, \log Y \leq -v) dudv
		\\
		&= p^2 \int_0^\infty \int_0^\infty (uv)^{p-1} S([e^u, \infty), [e^v, \infty)) dudv
		\\
		&= p^2 \int_1^\infty \int_1^\infty \frac{S([x, \infty), [y, \infty)) ((\log x)(\log y))^{p-1}}{xy} dxdy,
	\end{align*}
	using the change of variable $x=e^u$, $y=e^v$. This establishes \cref{eq:moment++,eq:moment+-,eq:moment--}.
\end{proof}

\begin{lemm}
\label{lemm:intrep}
Under \cref{assum:tail}, \cref{eq:erriml,eq:errimm,eq:errijm} hold for any $a \in (0, 1)$.
	
\end{lemm}

\begin{proof}
	
	Recall that $i, j, m$ are assumed to be distinct indices. {It is already proved in \cite[Section S.7]{EV22} that the moments of interest satisfy}
	\begin{align}
		e_i^{(m), \ell} &= \int_0^1 \frac{R_{im}(x, 1) (-2\log x)^{\ell-1}}{x} \d x - \int_1^\infty \frac{R_{im}([x, \infty), 1) (-2\log x)^{\ell-1}}{x} \d x, \label{eq:eiml}
		\\
		e_{im}^{(m)} &= \int_0^1 \int_0^1 \frac{R_{im}(x, y)}{xy} \d x \d y - \int_0^1 \int_1^\infty \frac{R_{im}([x, \infty), y)}{xy} \d x \d y, \label{eq:eimm}
		\\
		e_{ij}^{(m)} &= \int_0^1 \int_0^1 \frac{R_{ijm}(x, y, 1)}{xy} \d x \d y - \int_0^1 \int_1^\infty \frac{R_{ijm}([x, \infty), y, 1)}{xy} \d x \d y \notag
		\\
		&\quad - \int_1^\infty \int_0^1 \frac{R_{ijm}(x, [y, \infty), 1)}{xy} \d x \d y + \int_1^\infty \int_1^\infty \frac{R_{ijm}([x, \infty), [y, \infty), 1)}{xy} \d x \d y \label{eq:eijm}
	\end{align}
	and that their empirical versions satisfy
	\begin{align}
		\hat e_i^{(m), \ell} &= \int_{1/k}^1 \frac{\bar R_{im}(x, 1) (-2\log x)^{\ell-1}}{x} \d x - \int_1^{n/k} \frac{\bar R_{im}([x, \infty), 1) (-2\log x)^{\ell-1}}{x} \d x, \label{eq:eiml hat}
		\\
		\hat e_{im}^{(m)} &= \int_{1/k}^1 \int_{1/k}^1 \frac{\bar R_{im}(x, y)}{xy} \d x \d y - \int_{1/k}^1 \int_1^{n/k} \frac{\bar R_{im}([x, \infty), y)}{xy} \d x \d y, \label{eq:eimm hat}
		\\
		\hat e_{ij}^{(m)} &= \int_{1/k}^1 \int_{1/k}^1 \frac{\bar R_{ijm}(x, y, 1)}{xy} \d x \d y - \int_{1/k}^1 \int_1^{n/k} \frac{\bar R_{ijm}([x, \infty), y, 1)}{xy} \d x \d y \notag
		\\
		&\quad - \int_1^{n/k} \int_{1/k}^1 \frac{\bar R_{ijm}(x, [y, \infty), 1)}{xy} \d x \d y + \int_1^{n/k} \int_1^{n/k} \frac{\bar R_{ijm}([x, \infty), [y, \infty), 1)}{xy} \d x \d y, \label{eq:eijm hat}
	\end{align}
	where
	\[
		\bar R_J(\bx_J) := \frac{1}{k} \sum_{t=1}^n \Ind{\hat F_i(U_{ti}) \leq \frac{k}{n} x_i, i \in J}, \quad \bx_J := (x_i)_{i \in J} \in [0, \infty)^{|J|}.
	\]
	
	The integrals in \cref{eq:eiml,eq:eimm,eq:eijm} can be truncated above by using \cref{eq:T3}, which allows to upper bound the tails of the functions $R_J$. In particular, we have
	\begin{align*}
		\int_{n/k}^\infty \frac{R_{im}([x, \infty), 1) (2\log x)^{\ell-1}}{x} dx &\lesssim \int_{n/k}^\infty \frac{(\log x)^{\ell-1}}{x^{1+\xi}} dx \lesssim \Big( \frac{k}{n} \Big)^\xi \log(n/k),
		\\
		\\
		\int_0^1 \int_{n/k}^\infty \frac{R_{im}([x, \infty), y)}{xy} dxdy &= \int_0^1 \int_{n/k}^\infty \frac{R_{im}([x/y, \infty), 1)}{x} dxdy
        \\
        &\lesssim \int_0^1 \int_{n/k}^\infty \frac{(x/y)^{-\xi}}{x} dxdy \lesssim \Big( \frac{k}{n} \Big)^\xi,
	\end{align*}
	and
	\begin{align*}
		&\iint_{[1, \infty)^2 \backslash [1, n/k]^2} \frac{R_{ijm}([x, \infty), [y, \infty), 1)}{xy} dxdy
		\\
		= &\int_{n/k}^\infty \int_{n/k}^\infty \frac{R_{ijm}([x, \infty), [y, \infty), 1)}{xy} dxdy
		+ \int_1^{n/k} \int_{n/k}^\infty \frac{R_{ijm}([x, \infty), [y, \infty), 1)}{xy} dxdy
		\\
		& + \int_{n/k}^\infty\int_1^{n/k} \frac{R_{ijm}([x, \infty), [y, \infty), 1)}{xy} dxdy
		\\
		\leq &\int_{n/k}^\infty \int_{n/k}^\infty \frac{R_{ijm}([x, \infty), [y, \infty), 1)}{xy} dxdy
		+\int_1^{n/k} \int_{n/k}^\infty \frac{R_{im}([x, \infty), 1) }{xy} dxdy
        \\
        &
		+ \int_{n/k}^\infty\int_1^{n/k} \frac{R_{jm}([y, \infty), 1)}{xy} dxdy
		\\
		\lesssim &\int_{n/k}^\infty \int_{n/k}^\infty \frac{x^{-\xi} \wedge y^{-\xi}}{xy} dxdy + 2 \int_1^{n/k} \int_{n/k}^\infty \frac{x^{-\xi}}{xy} dxdy
		\\
		\lesssim &\Big( \frac{k}{n} \Big)^\xi \log(n/k).
	\end{align*}
	
	Hence we proved that all integral can be truncated above at $n/k$ while incurring an error of at most {$O((k/n)^\xi \log(n/k))$}. Next we show that the integrals can as well be truncated below.
	
	{Recall that $a \in (0, 1)$.} Since by their definitions, $R_J$ and $\bar R_J$ are both upper bounded by the minimum component of their argument, so is $|\bar R_J - R_J|$. We then have for $\ell \in \{1, 2\}$
	\begin{align*}
		&\int_0^a \frac{|\bar R_{im}(x, 1) - R_{im}(x, 1)| (-2\log x)^{\ell-1}}{x} dx
        \\
        &\leq \int_0^a (-2\log x)^{\ell-1} dx \lesssim a(1 + \log(1/a)),
		\\
		\\
		&\iint_{[0, 1]^2 \backslash [a, 1]^2} \frac{|\bar R_{im}(x, y) - R_{im}(x, y)|}{xy} dxdy
        \\
        &\leq \int_0^a \int_0^a \frac{x \wedge y}{xy} dxdy + 2 \int_a^1 \int_0^a \frac{1}{y} dxdy \lesssim a(1 + \log(1/a)),
		\\
		\\
		&\int_0^a \int_1^{n/k} \frac{|\bar R_{im}([x, \infty), y) - R_{im}([x, \infty), y)|}{xy} dxdy \leq \int_0^a \int_1^{n/k} \frac{1}{x} dxdy \leq a\log(n/k),
	\end{align*}
	and by symmetry
	\[
		\int_1^{n/k} \int_0^a \frac{|\bar R_{im}(x, [y, \infty)) - R_{im}(x, [y, \infty))|}{xy} dxdy
	\]
	admits the same bound. Finally, the integral
	\[
		\iint_{[0, 1]^2 \backslash [a, 1]^2} \frac{|\bar R_{ijm}(x, y,1) - R_{ijm}(x, y,1)|}{xy} dxdy
	\]
	is handled similarly as 
	\[
		\iint_{[0, 1]^2 \backslash [a, 1]^2} \frac{|\bar R_{im}(x, y) - R_{im}(x, y)|}{xy} dxdy.
	\]
	
	We have therefore proved that each of the integrals in \cref{eq:eiml,eq:eimm,eq:eijm,eq:eiml hat,eq:eimm hat,eq:eijm hat} can be truncated below at a point $a$ and above at $n/k$, up to a deterministic additive error which satisfies the bound {$\lesssim (k/n)^\xi \log(n/k) + a (\log(n/k) + \log(1/a))$}. It follows that with probability 1,
	\begin{align*}
		\hat e_i^{(m), \ell} - e_i^{(m), \ell} &= \int_a^1 \frac{\big( \bar R_{im}(x, 1) - R_{im}(x, 1) \big) (-2\log x)^{\ell-1}}{x} \d x
		\\
		&\quad - \int_1^{n/k} \frac{\big( \bar R_{im}([x, \infty), 1) - R_{im}([x, \infty), 1) \big) (-2\log x)^{\ell-1}}{x} \d x
		\\
		&\quad + O\bigg( \Big( \frac{k}{n} \Big)^\xi \log(n/k) + a(\log(n/k) + \log(1/a)) \bigg),
		\\
		\hat e_{im}^{(m)} - e_{im}^{(m)} &= \int_a^1 \int_a^1 \frac{\bar R_{im}(x, y) - R_{im}(x, y)}{xy} \d x \d y
		\\
		&\quad - \int_a^1 \int_1^{n/k} \frac{\bar R_{im}([x, \infty), y) - R_{im}([x, \infty), y)}{xy} \d x \d y
		\\
		&\quad + O\bigg( \Big( \frac{k}{n} \Big)^\xi \log(n/k) + a(\log(n/k) + \log(1/a)) \bigg),
		\\
		\hat e_{ij}^{(m)} - e_{ij}^{(m)} &= \int_a^1 \int_a^1 \frac{\bar R_{ijm}(x, y, 1) - R_{ijm}(x, y, 1)}{xy} \d x \d y
		\\
		&\quad - \int_a^1 \int_1^{n/k} \frac{\bar R_{ijm}([x, \infty), y, 1) - R_{ijm}([x, \infty), y, 1)}{xy} \d x \d y
		\\
		&\quad - \int_1^{n/k} \int_a^1 \frac{\bar R_{ijm}(x, [y, \infty), 1) - R_{ijm}(x, [y, \infty), 1)}{xy} \d x \d y
		\\
		&\quad + \int_1^{n/k} \int_1^{n/k} \frac{\bar R_{ijm}([x, \infty), [y, \infty), 1) - R_{ijm}([x, \infty), [y, \infty), 1)}{xy} \d x \d y
		\\
		&\quad + O\bigg( \Big( \frac{k}{n} \Big)^\xi \log(n/k) + a(\log(n/k) + \log(1/a)) \bigg),
	\end{align*}
	where the error terms are deterministic. All that remains to obtain the desired result is to replace the functions $\bar R_J$ above by $\hat R_J$, which amounts to comparing the left- and right-continuous versions of an empirical tail copula. By the result in Appendix C.1 of \cite{RWZ17}, we have
	\[
		\max_{J: |J| \leq 3} \sup_{\bx_J \in [0, \infty)^{|J|}} |\bar R_J(\bx_J) - \hat R_J(\bx_J)| \leq \frac{3}{k}
	\]
	almost surely, so replacing $\bar R_J$ by $\hat R_J$ in the integrals above adds an error that is at most of the order of $(\log(n/k) + \log(1/a))^2/k$.
\end{proof}

\begin{lemm}
\label{lemm:boundGamma}
	
	Under \cref{assum:tail}, $\max_{m \in V} \big\| \Gamma^{(m)} \big\|_\infty$ admits an upper bound that depends only on $K$ and $\xi$.
	
\end{lemm}

\begin{proof}
	
First, as is pointed out in \cref{sec:emml}, for $\ell \in \{1, 2\}$, $e_m^{(m), \ell} = \ell$.
	
The remaining arguments are based on \cref{eq:T3}, which holds by assumption and states that for all distinct triples $(i, j, m)$,
\[
	R_{ij}([x, \infty), 1) \leq Kx^{-\xi}, \quad R_{ijm}([x, \infty), [y, \infty), 1) \leq K (x \wedge y)^{-\xi}, \quad x, y \geq 1.
\]
Equally important is the fact that every function $R_J$ is upper bounded by its minimum argument. The proof consists of plugging those different bounds in \cref{eq:eiml,eq:eimm,eq:eijm} above, which provided expressions for the moments $e_i^{(m), \ell}$, $e_{im}^{(m)}$ and $e_{ij}^{(m)}$. Repeatedly using the inequality $a \wedge b \leq (ab)^{1/2}$ for positive $a,b$, deduce that
\begin{align*}
\big| e_i^{(m), \ell} \big| &\leq \int_0^1 (-2\log x)^{\ell-1} dx + K \int_1^\infty \frac{(-2\log x)^{\ell-1}}{x^{1+\xi}} dx,
\\
\big| e_{im}^{(m)} \big| &\leq \int_0^1 \int_0^1 (xy)^{-1/2} dxdy + \sqrt{K} \int_0^1 \int_1^\infty x^{-1-\xi/2} y^{-1/2} dxdy,
\\
\big| e_{ij}^{(m)} \big| &\leq \int_0^1 \int_0^1 (xy)^{-1/2} dxdy + 2\sqrt{K} \int_0^1 \int_1^\infty x^{-1-\xi/2} y^{-1/2} dxdy
\\
& \quad + K \int_1^\infty \int_1^\infty (xy)^{-1-\xi/2} dxdy.
\end{align*}
Simply plugging those bounds in \cref{eq:repGamma} yields the result.
\end{proof}
\color{black}

\subsection{Bounds on the measures \texorpdfstring{$R_{ij}$}{Rij}}

Recall the representation of $R_{ij}$ as a non-negative measure, for an arbitrary pair $i \neq j$. The following bounds necessarily hold.

\begin{lemm}
\label{lemm:boundR}
	
	Let $0 < a \leq b$ and $y>0$. Then for every distinct pair $(i, j)$,
	\[
		R_{ij}([a, b], y) \leq y \frac{b-a}{a}.
	\]
	
\end{lemm}

\begin{proof}
	
	The idea is that the rectangle $[a, b] \times [0, y]$ is included in the trapezoid $\{(u, v) \in [0, \infty)^2: a \leq u \leq b, v \leq yu/a\} = S(b) \backslash S(a)$, where
	\[
		S(x) := \{(u, v) \in [0, \infty)^2: u \leq x, v \leq yu/a\}.
	\]
	By homogeneity of $R_{ij}$,
	\[
		R_{ij}(S(b) \backslash S(a)) = (b-a) R_{ij}(S(1)) \leq (b-a) R_{ij}(1, y/a) \leq y \frac{b-a}{a},
	\]
	since $R_{ij}$ is always upper bounded by its smallest argument.
\end{proof}

The following bound assumes more but is considerably more flexible, as $\beta$ can be both smaller and larger than 1.

\begin{lemm}
\label{lemm:boundRsmall}
	
	Under \cref{assum:r}, for every $\beta \in (0, 1+\eps]$ there exists $K(\beta) < \infty$ such that for any $0 < a \leq b$, $y>0$ and every distinct pair $(i, j)$,
	\[
		R_{ij}([a, b], y) \leq \frac{K(\beta)}{\beta} y^\beta \frac{b-a}{a^\beta}.
	\]
	
\end{lemm}

\begin{proof}
	
	The bound in \cref{assum:r} gives
	\begin{align*}
		R_{ij}([a, b], y) &= \int_0^y \int_a^b r_{ij}(u, v) dudv \leq K(\beta) \int_0^y v^{\beta-1} dv \int_a^b u^{-\beta} du
        \\
        &\leq \frac{K(\beta)}{\beta} y^\beta \int_a^b a^{-\beta} du = \frac{K(\beta)}{\beta} y^\beta \frac{b-a}{a^\beta}.
	\end{align*}
\end{proof}

\begin{lemm}
\label{lemm:boundRmix}
	
	Under \cref{assum:r}, for every $\beta \in [-\eps, 0)$ there exists $K(\beta) < \infty$ such that for any $0 < a \leq b$, $y>0$ and every distinct pair $(i, j)$,
	\[
		R_{ij}([a, b], [y, \infty)) \leq \frac{K(\beta)}{-\beta} y^\beta b^{-\beta} (b-a).
	\]
	
\end{lemm}

\begin{proof}
	
	Following the proof of \cref{lemm:boundRsmall},
	\begin{align*}
		R_{ij}([a, b], [y, \infty)) &\leq K(\beta) \int_y^\infty v^{\beta-1} dv \int_a^b u^{-\beta} du
        \\
        &\leq \frac{K(\beta)}{-\beta} y^\beta \int_a^b b^{-\beta} du = \frac{K(\beta)}{-\beta} y^\beta b^{-\beta} (b-a).
	\end{align*}
\end{proof}

\subsection{Technical results from empirical process theory}\label{sec:s_mislan}

In this section, we collect two fundamental inequalities from empirical process theory that are used in \cref{sec:ep}. Denote by $\cG$ a class of real-valued functions that satisfies $|f(x)| \leq F(x) \leq U$ for every $f \in \cG$ and let $\sigma^2 \geq \sup_{f \in \cG} Pf^2$. Additionally, suppose that for some positive $A$, $V$ and for all $\eps > 0$,
\begin{equation} \label{eq:entr}
	N(\eps,\cG,L_2(\Prob_n)) \leq \Big(\frac{A \|F\|_{L^2(\Prob_n)}}{\eps} \Big)^V
\end{equation}
almost surely. In that case, the symmetrization inequality and inequality (2.2) from \cite{K06} yield
\begin{equation} \label{eq:Gexpect}
	\E[\|\Prob_n - P\|_{\cG}] \leq c_0 \Big[\sigma \Big(\frac{V}{n} \log \frac{A \|F\|_{L^2(P)}}{\sigma}\Big)^{1/2} + \frac{VU}{n} \log \frac{A \|F\|_{L^2(P)}}{\sigma}  \Big]
\end{equation}
for a universal constant $c_0>0$ provided that $1 \geq \sigma^2 > \text{const}\times n^{-1}$. In fact, the inequality in \cite{K06} is for $\sigma^2 = \sup_{f \in \cG} Pf^2$. However, this is not a problem since we can replace $\cG$ by $\cG \sigma / (\sup_{f \in \cG} Pf^2)^{1/2}$.

The second inequality (a refined version of Talagrand's concentration inequality) states that for any countable class of measurable functions $\cF$ with elements mapping into $[-M,M]$,
\begin{equation} \label{eq:Gcontract}
	\Prob\Big( \|\Prob_n - P\|_{\cF} \geq 2 \E[\|\Prob_n - P\|_{\cF}] + c_1 n^{-1/2 }\Big(\sup_{f \in \cF} Pf^2\Big)^{1/2} \sqrt{t} + n^{-1}c_2 M t \Big) \leq e^{-t},
\end{equation}
for all $t>0$ and some universal constants $c_1,c_2 >0$. This is a special case of Theorem 3 in \cite{mass2000} (in the notation of that paper, set $\eps = 1$).

\subsection{Discussion of multivariate Pareto distributions and some domain of attraction conditions}
\label{sec:doa}

\newcommand{\CP}{{\rm{CP}}}

In this section, it is shown that the basic domain of attraction condition in \cref{eq:mPdlimit} implies other limit relations. We denote convergence in distribution in Euclidean spaces by~``$\wc$''. Unless stated otherwise, all the convergence statements are understood as $q \downarrow 0$.

\begin{prop} \label{prop:doa}
	Under \cref{eq:mPdlimit},
	\[
		\frac{q}{1-F(\g X)} \, \Big | \, \{F_m(X_m) > 1-q\} \wc \bY^{(m)}.
	\]
\end{prop}

\begin{proof}

	Let $\bY_q \in \R^d$ denote a random vector distributed according to the conditional law of $\tfrac{q}{1 - F(\bX)} \mid \|\tfrac{q}{1 - F(\bX)}\|_\infty > 1$. Writing $\CP(f)$ for the set of continuity points of a function $f$, the assumption in \cref{eq:mPdlimit} can be written as
	\[
		\Prob(\bY_q \leq \bx) \too \Prob(\bY \leq \bx), \quad \bx \in \CP(\Prob(\bY \leq \cdot)),
	\]
	which defines convergence in distribution of $\bY_q$ to $\bY$. Note that this is further equivalent to
	\[
		\Prob(\bY_q > \bx) \too \Prob(\bY > \bx), \quad \bx \in \CP(\Prob(\bY > \cdot)),
	\]
	since $\bY_q \wc \bY$ if and only if $-\bY_q \wc -\bY$ and since $\CP(\Prob(-\bY \leq \cdot)) = \CP(\Prob(-\bY < \cdot))$. Using this principle, it suffices to show that for any $\bx \in \CP(\Prob(\bY^{(m)} > \cdot))$, $\Prob(\bY_q > \bx \mid (\bY_q)_m > 1) \to \Prob(\bY^{(m)} > \bx)$. Before establishing this, we shall study the continuity points of $\Prob(\bY^{(m)} > \cdot)$.

	\textbf{The continuity points where $x_m > 1$:}
	Let $\bx \in \CP(\Prob(\bY^{(m)} > \cdot))$ be such that $x_m > 1$. Then $\bx$ is also a continuity point of $\Prob(\bY > \cdot)$, since these two survival functions are (locally) proportional to each other: for $\by$ in an open neighborhood of $\bx$,
	\[
		\Prob(\bY > \by) = \Prob(\bY > \by, Y_m > 1) = \Prob(Y_m > 1) \Prob(\bY^{(m)} > \by).
	\]

	\textbf{The continuity points where $x_m \leq 1$:}
	Let $\bx \in \CP(\Prob(\bY^{(m)} > \cdot))$ be such that $x_m \leq 1$. We claim that necessarily, the point $\bx^{(m)} := (x_1, \dots, x_{m-1}, 1, x_{m+1}, \dots, x_d)$ belongs to $\CP(\Prob(\bY > \cdot))$. Indeed, $\bx \in \CP(\Prob(\bY^{(m)} > \cdot))$ and $\bx^{(m)} \in \CP(\Prob(\bY > \cdot))$ are equivalent to
	\begin{equation} \label{eq:YmCP}
		\Prob(\bY^{(m)} \in \partial(\bx, \infty)) = 0
	\end{equation}
	and
	\begin{equation} \label{eq:YCP}
		\Prob(\bY \in \partial(\bx^{(m)}, \infty)) = 0,
	\end{equation}
	respectively, where $\partial A$ denotes the boundary of $A$ in $\R^d$. The latter probability is upper bounded by
	\begin{multline*}
		\Prob(Y_m = 1) + \sum_{j \notin \{m, j\}} \Prob(Y_j = x_j, \bY_{\setminus \{m, j\}} \geq \bx_{\setminus \{m, j\}}, Y_m \geq 1)
        \\
		= \sum_{j \notin \{m, j\}: x_j = 0} \Prob(Y_j = 0, \bY_{\setminus \{m, j\}} \geq \bx_{\setminus \{m, j\}}, Y_m \geq 1),
	\end{multline*}
	using the fact that $\Prob(Y_j = x) = 0$ for each $j$ and every $x>0$. To check the latter, suppose that $\Prob(Y_j = x) = p > 0$ for some $x>0$ with $j=1$, say. Then one can identify an infinite number of mutually disjoint sets of the form $( \{ax\} \times [0, \infty) \times \dots \times [0, \infty) ) \cap \cL$ with $1 \leq a \leq 2$, which by homogeneity all have a $\bY$-mass of at least $p/2$, a contradiction. Thus in order for \cref{eq:YCP} to fail, there needs to exist at least one index $j^*$ such that $x_{j^*} = 0$ and $\Prob(Y_{j^*} = 0, \bY_{\setminus \{m, j^*\}} \geq \bx_{\setminus \{m, j^*\}}, Y_m \geq 1) > 0$. Supposing this is the case, however, we have
	\begin{align*}
		\Prob(\bY^{(m)} \in \partial(\bx, \infty)) &\geq \Prob(\bY \in \partial(\bx, \infty), Y_m>1)
        \\
        &\geq \Prob(Y_{j^*} = 0, \bY_{\setminus \{m, j^*\}} \geq \bx_{\setminus \{m, j^*\}}, Y_m > 1) > 0,
	\end{align*}
	so that \cref{eq:YmCP} also fails. Deduce that \cref{eq:YmCP} implies \cref{eq:YCP}.

	\textbf{Proof of the proposition:}
	We are now ready to establish convergence of the conditional survival function of $\bY_q \mid (\bY_q)_m > 1$. We have
	\[
		\Prob(\bY_q > \bx \mid (\bY_q)_m > 1) = \frac{\Prob(\bY_q > \bx, (\bY_q)_m > 1)}{\Prob((\bY_q)_m > 1)} \sim \frac{\Prob(\bY_q > \bx, (\bY_q)_m > 1)}{\Prob(Y_m > 1)},
	\]
	since it was already shown that $\Prob(Y_m = 1) = 0$. If $\bx \in \CP(\Prob(\bY^{(m)} > \cdot))$ and $x_m > 1$, it was shown that $\bx$ is also a continuity point of $\Prob(\bY > \cdot)$, hence $\Prob(\bY_q > \bx, (\bY_q)_m > 1) = \Prob(\bY_q > \bx) \to \Prob(\bY > \bx)$. If on the other hand $\bx \in \CP(\Prob(\bY^{(m)} > \cdot))$ and $x_m \leq 1$, it was shown that $\bx^{(m)} \in \CP(\Prob(\bY > \cdot))$, hence $\Prob(\bY_q > \bx, (\bY_q)_m > 1) = \Prob(\bY_q > \bx^{(m)}) \to \Prob(\bY > \bx^{(m)})$. In both cases, we therefore have
	\[
		\Prob(\bY_q > \bx \mid (\bY_q)_m > 1) \too \frac{\Prob(\bY > \bx, Y_m > 1)}{\Prob(Y_m > 1)} = \Prob(\bY^{(m)} > \bx),
	\]
	from which we may finally conclude that $\bY_q \mid \{(\bY_q)_m > 1\} \wc \bY^{(m)}$.
\end{proof}

\begin{prop} \label{prop:subR}
	Assume that \cref{eq:mPdlimit} holds, that the marginal distribution functions $F_j$ are eventually continuous, and that the multivariate Pareto distribution $\bY$ satisfies $\Prob(Y_j=0)$, $j \in V$. For each subset $J \subseteq \{1, \dots, d\}$, define
	\[
		R_J(\bx_J) := \lim_{q \downarrow 0} q^{-1} \Prob(F_J(\bX_J) > 1 - q\bx_J).
	\]
	Then for each such $J$,
	\begin{itemize}
		\item[(i)] for every $\bx_J \in [0, \infty)^{|J|}$ such that $\|1/\bx_J\|_\infty > 1$,
		\[
			R_J(\bx_J) = \frac{\Prob(\bY_J > 1/\bx_J)}{\Prob(Y_1 > 1)};
		\]
		\item[(ii)] for every $\bx_J \in [0, \infty)^{|J|}$
		\[
			\lim_{\bx_{\setminus J} \to \infty} R(\bx) = R_J(\bx_J).
		\]
	\end{itemize}
\end{prop}

\begin{proof}

	\textbf{Proof of (i):}
	It suffices to prove the case where $J = \{1, \dots, d\}$, i.e., that for every $\bx \in [0, \infty)^d$ such that $\|1/\bx\|_\infty > 1$,
	\begin{equation} \label{eq:Y2R}
		R(\bx) = \frac{\Prob(\bY > 1/\bx)}{\Prob(Y_1 > 1)}.
	\end{equation}
	Indeed, for $J \subset \{1, \dots, d\}$, the random vector $\bX_J$ itself satisfies \cref{eq:mPdlimit} with the multivariate Pareto distributed $\bY$ replaced by $\bY_J \mid \{\|\bY_J\|_\infty > 1\}$ (the components of which do not have a positive mass at 0) in the sense that
	\begin{align*}
		&\Prob\Big( \frac{q}{1 - F_J(\bX_J)} > \bx_J \, \Big | \, \Big\| \frac{q}{1 - F_J(\bX_J)} \Big\|_\infty > 1 \Big)
        \\
        &= \frac{\Prob\Big( \frac{q}{1 - F_J(\bX_J)} > \bx_J, \Big\| \frac{q}{1 - F_J(\bX_J)} \Big\|_\infty > 1 \Big)}{\Prob\Big( \Big\| \frac{q}{1 - F_J(\bX_J)} \Big\|_\infty > 1 \Big)}
		\\
		&= \frac{\Prob\Big( \frac{q}{1 - F_J(\bX_J)} > \bx_J, \Big\| \frac{q}{1 - F_J(\bX_J)} \Big\|_\infty > 1 \, \Big | \, \Big\| \frac{q}{1 - F(\bX)} \Big\|_\infty > 1 \Big)}{\Prob\Big( \Big\| \frac{q}{1 - F_J(\bX_J)} \Big\|_\infty > 1 \, \Big | \, \Big\| \frac{q}{1 - F(\bX)} \Big\|_\infty > 1 \Big)}
		\\
		&\too \frac{\Prob(\bY_J > \bx_J, \|\bY_J\|_\infty > 1)}{\Prob(\|\bY_J\|_\infty > 1)}
		\\
		&= \Prob(\bY_J > \bx_J \mid \|\bY_J\|_\infty > 1).
	\end{align*}
	In order to guarantee the convergence of all the probabilities above, we have used the assumption that $\Prob(Y_j=0)=0$, along with the fact shown in the proof of \cref{prop:doa} that $Y_j$ cannot have a point mass in $(0, \infty)$, to deduce that $\Prob(\bY > \cdot)$ is continuous everywhere. Therefore if \cref{eq:Y2R} holds, then it also holds that for $\bx_J \in [0, \infty)^{|J|}$ such that $\|1/\bx_J\|_\infty > 1$,
	\[
		R_J(\bx_J) = \frac{\Prob(\bY_J > 1/\bx_J \mid \|\bY_J\|_\infty > 1)}{\Prob(Y_1 > 1 \mid \|\bY_J\|_\infty > 1)} = \frac{\Prob(\bY_J > 1/\bx_J)}{\Prob(Y_1 > 1 )}.
	\]

	To prove \cref{eq:Y2R}, simply apply \cref{eq:mPdlimit} to obtain, for $\bx$ such that $1/\bx \in \cL$,
	\begin{align*}
		\frac{\Prob(\bY > 1/\bx)}{\Prob(Y_1 > 1)} &= \lim_{q \downarrow 0} \frac{\Prob(F(\bX) > 1-q\bx \mid F(\bX) \not\leq 1-q)}{\Prob(F_1(X_1) > 1-q \mid F(\bX) \not\leq 1-q)}
        \\
		&= \lim_{q \downarrow 0} \frac{\Prob(F(\bX) > 1-q\bx)}{\Prob(F_1(X_1) > 1-q)}
		= \lim_{q \downarrow 0} \frac{\Prob(F(\bX) > 1-q\bx)}{q},
	\end{align*}
	since the eventual continuity of $F_1$ guarantees that $\Prob(F_1(X_1) > 1-q) = q$ for sufficiently small $q$. The limit on the right-hand side above is by definition equal to $R(\bX)$.

	\textbf{Proof of (ii):}
	Let $j \in V$ be arbitrary and $\bx \in [0, \infty)^d$ be such that $\|1/\bx_{\setminus j}\|_\infty > 1$. Using (i), we find that as $x_j \to \infty$,
	\[
		R(\bx) = \frac{\Prob(\bY > 1/\bx)}{\Prob(Y_1 > 1)} \too \frac{\Prob(\bY_{\setminus j} > 1/\bx_{\setminus j}, Y_j > 0)}{\Prob(Y_1 > 1)} = \frac{\Prob(\bY_{\setminus j} > 1/\bx_{\setminus j})}{\Prob(Y_1 > 1)}
	\]
	where the convergence is due to continuity from below, and the last equality uses the assumption that $\Prob(Y_j > 0) = 1$. The right-hand side above is equal to $R_{\setminus j}(\bx_{\setminus j})$, again by (i). The proof follows by induction on the size of $J$.
\end{proof}

\subsection{Discussion of max-stable distributions}
\label{sec:ms}

In this section, we take $\bX$ to be distributed according to the max-stable distribution associated to an arbitrary multivariate Pareto $\bY$ with stable dependence function $L$; recall the definition of this function from \cref{eq:elllimit}. That is, the copula of $\bX$ is given by
\[
	\Prob(F(\bX) \leq \bx) = \exp\{ -L(-\log \bx) \}.
\]
We shall demonstrate the following result. Note that the constant $K' = 48$ therein is not particularly sharp, and can be improved at the cost of more detailed calculations.

\begin{prop}\label{prop:maxstableA3}
	
	The max-stable random vector $\bX$, assuming that its marginal distributions are continuous, satisfies \cref{assum:RJstand} with $K' = 48$ and $\xi' = 1$.
	
\end{prop}

\begin{proof}
	
	First note that for $q>1/2$ and $\bx \in [0, 1]^{|J|}$, we have
	\[
		\Big| q^{-1} \Prob(F_J(\bX_J) > 1 - q\bx) - R_J(\bx) \Big| \leq 1 \leq 2q,
	\]
	so we may without loss of generality consider only $q \leq 1/2$.
	
	Now let $J$ be a subset of size 2 or 3. For $q \leq 1/2$, we have
	\begin{align}
		\Prob(F_J(\bX_J) \not\leq 1-q\bx) &= 1 - \exp\{ -L_J(-\log(1-q\bx)) \} \notag
		\\
		&\in \Big[ L_J(-\log(1-q\bx)) - \frac{1}{2} L_J(-\log(1-q\bx))^2, L_J(-\log(1-q\bx)) \Big], \label{eq:boundell}
	\end{align}
	using a Taylor expansion of the exponential around the origin: for $z > 0$, there exists $0 \leq \bar z \leq z$ such that $e^{-z} = 1 - z + e^{-\bar z} z^2/2 \in [1-z, 1-z+z^2/2]$. Here $L_J$ is the stable tail dependence function of the subvector $\bX_J$; it is obtained by evaluating $L$ at a point the components of which in positions $J^c$ are zero. Using a similar decomposition of $-\log(1 - z)$ around $z=0$, we find that
	\[
		z \leq -\log(1 - z) \leq z + 2z^2.
	\]
	Now recall that $\bx \in [0, 1]^{|J|}$. Using the properties of stable tail dependence functions, namely that $L$ (and $L_J$) is component-wise monotone, convex, homogeneous and upper bounded by the sum of its arguments, we find
	\begin{align*}
		L_J(q\bx) &\leq L_J(-\log(1-q\bx)) \leq L_J(q\bx + 2q^2\bone)
        \\
        &= 2L_J\big( \tfrac{1}{2} q\bx + \tfrac{1}{2} 2q^2\bone \big) \leq L_J(q\bx) + L_J(2q^2\bone) \leq L_J(q\bx) + 2|J|q^2.
	\end{align*}
	Finally, deduce from \cref{eq:boundell} that
	\[
		q^{-1} \Prob(F_J(\bX_J) \not\leq 1-q\bx) \leq L_J(\bx) + 6q,
	\]
	and that
	\begin{align*}
		&q^{-1} \Prob(F_J(\bX_J) \not\leq 1-q\bx)
        \\
        &\geq
		\min\Big\{
			L_J(\bx) - q^{-1} \frac{1}{2} L_J(q\bx)^2,
			L_J(\bx) + 6q - q^{-1} \frac{1}{2} \big( L_J(q\bx) + 6q^2 \big)^2
		\Big\}
		\\
		&\geq \min\Big\{ L_J(\bx) - \frac{9}{2} q, L_J(\bx) + 6q - q^{-1} \frac{1}{2} (3q + 3q)^2 \Big\}
		\\
		&= L_J(\bx) - 12q.
	\end{align*}
	We have established an approximation similar to what is desired, but for the probabilities $\Prob(F_J(\bX_J) \not\leq 1-q\bx)$ by the functions $L_J$:
	\begin{equation} \label{eq:approxell}
		\Big| q^{-1} \Prob(F_J(\bX_J) \not\leq 1-q\bx) - L_J(\bx) \Big| \leq 12q.
	\end{equation}
	We shall use this result to complete the proof.
	
	For $i \in V$, let $E_i = \{F_i(X_i) > 1 - qx_i\}$. Now, if $J = (i, j)$ has size 2, then
	\[
		\Prob(F_J(\bX_J) > 1 - q\bx) = \Prob(E_i \cap E_J) = qx_i + qx_j - \Prob(E_i \cup E_j)
	\]
	and
	\[
		R_J(\bx) = x_i + x_j - L_J(\bx),
	\]
	so the result follows from \cref{eq:approxell}. If $J = (i, j, m)$ has size 3,
	\begin{align*}
		\Prob(F_J(\bX_J) > 1 - q\bx) &= \Prob(E_i \cap E_j \cap E_m)
		\\
		&= \Prob(E_i \cup E_j \cup E_m) - \Prob(E_i \cup E_j) - \Prob(E_i \cup E_m) - \Prob(E_j \cup E_m)
		\\
		&\quad + \Prob(E_i) + \Prob(E_j) + \Prob(E_m)
		\\
		&= \Prob(E_i \cup E_j \cup E_m) - \Prob(E_i \cup E_j) - \Prob(E_i \cup E_m) - \Prob(E_j \cup E_m)
		\\
		&\quad + qx_i + qx_j + qx_m,
	\end{align*}
	and similarly
	\[
		R_J(\bx) = L_J(\bx) - L_{ij}(x_i, x_j) - L_{im}(x_i, x_m) - L_{jm}(x_j, x_m) + x_i + x_j + x_m,
	\]
	so the result again follows by applying \cref{eq:approxell} to approximate each of the four probabilities above by the corresponding $L$ terms.
\end{proof}

\end{document}